\documentclass[reqno]{amsart} \usepackage[centertags]{amsmath}
\usepackage{graphicx} \usepackage{amsfonts} \usepackage{amssymb}
\usepackage{amsthm} \usepackage[top=1in, bottom=1in, left=1in,
right=1in]{geometry}

\usepackage[pdfborder={0 0 0}, pdftitle={SUBDIVISION RULES FOR SPECIAL CUBULATED
GROUPS}, pdfauthor={BRIAN RUSHTON}, pdftex, bookmarks=true,
bookmarksnumbered = true]{hyperref}

\theoremstyle{plain} \newtheorem{thm}{Theorem} \theoremstyle{plain}
\newtheorem{lemma}[thm]{Lemma} \theoremstyle{definition}
\newtheorem*{defi}{Definition} \theoremstyle{remark}
 \theoremstyle{plain}
\newtheorem{cor}[thm]{Corollary}

\newcommand{\parlengths}{\setlength{\parindent}{0pt}}
\begin{document} \date{\today}

\pdfbookmark[1]{SUBDIVISION RULES FOR SPECIAL CUBULATED GROUPS}{user-title-page}

\title{Subdivision rules for special cubulated groups}

\author{Brian Rushton} \address{Department of Mathematics, Temple
University, Philadelphia, PA 19122, USA}
\email{brian.rushton@temple.edu}

\begin{abstract} We find explicit subdivision rules for all special cubulated groups. A subdivision
rule for a group produces a sequence of tilings on a sphere which encode all quasi-isometric information for a group. We show how these tilings detect properties such as growth, ends, divergence, etc. We include figures of several worked out examples. \end{abstract}

\maketitle\parlengths \section{Introduction}\label{Introduction}

A subdivision rule is a rule for creating a sequence of tilings of a manifold (usually a sphere), where each tiling is a subdivision of the previous tiling. In this paper, we will find subdivision rules for groups, where the sequence of tilings produced by the subdivision rule will correspond to spheres in the Cayley graph. These subdivision rules for groups are a geometric way of studying the combinatorial
behavior of the groups near infinity. In fact, subdivision rules for groups encode
all the quasi-isometry invariants of the group. In this paper, we obtain
subdivision rules for all right-angled Artin groups and all special cubulated groups.
The sequence of tilings we use will be the tilings that spheres in the universal cover inherit from the walls in wall-structure of the universal cover.

We will define subdivision rules more formally in Section
\ref{FormalSection}, but we introduce the basic notation here. A \textbf{subdivision rule }$R$ acts on
some complex $X$ to produce a sequence of complexes $\{R^n(X)\}_{n=1}^{\infty}$, where
each complex $R^n(X)$ is a \emph{subdivision} of $R^{n-1}(X)$ (meaning that
every cell in $R^{n-1}(X)$ is a union of cells in $R^{n}(X)$). A
\textbf{finite subdivision rule} is one where the subdivision is locally
determined by finitely many \textbf{tile types}, meaning that each cell
of $R^n(X)$ is labelled by a tile type, and that if two cells have the
same tile type, their subdivisions are cellularly isomorphic by a map
that preserves labels. Barycentric subdivision is the classic example.

Our main theorems are the following:

\begin{thm} \label{BiggerTheorem} Every right-angled Artin group has a finite
subdivision rule. \end{thm}

\begin{thm} \label{BiggestTheorem} The fundamental group of a compact special cube complex $X$
has a subdivision rule. If there is a local isometry of $X$ into a Salvetti complex of a RAAG $A$, then the subdivision rule for $A$ contains a copy of the subdivision rule for $X$.  \end{thm}

These theorems are proved in Sections \ref{RAAGSection} and \ref{CubulatedSection}, respectively. They have two main consequences: 

First, the subdivision rules defined in this paper can be used to study quasi-isometry properties
of special cubulated groups, as described in the next section.

Second, they provide a rich family of explicit examples of subdivision rules for hyperbolic groups. Such subdivision rules have been studied extensively \cite{conformal, Rich, subdivision, hyperbolic, Combinatorial}, but there have been very few explicit examples \cite{PolySubs}.

\subsection{Subdivision rules and quasi-isometry invariants} Cannon was
the first to study the quasi-isometry properties of groups using
subdivision rules. He showed that a $\delta$-hyperbolic group with a
2-sphere at infinity is quasi-isometric to hyperbolic 3-space if and
only if its subdivision rules are `conformal', meaning that tiles do not
become distorted under subdivision \cite{conformal,hyperbolic}.

The subdivision rules created in this paper can also be used to study
quasi-isometry invariants of groups. For every subdivision rule on a manifold $X$, we can
construct a graph called the \textbf{history graph} which will be quasi-isometric to the Cayley graph of the group it
is associated to. The history graph is the disjoint union of the dual
graphs of $R^n(X)$ for $n=0,1,2,...$, together with a collection of
edges which connect each vertex of the dual graph of $R^n(X)$ (where
each vertex corresponds to a cell of highest dimension) to each vertex
corresponding to cells contained in its subdivision in $R^{n+1}(X)$.

 To provide more variety in possible history graphs, we may consider some of the tile types to be \textbf{ideal
cells}, meaning that they are not assigned vertices in the history graph. Ideal cells only subdivide into other ideal cells, and frequently will never subdivide at all.

One example is the `middle thirds' subdivision rule commonly used to
create a Cantor set. There are two tile types, $A$ and $B$, and we consider $B$ as `ideal' (see Figure \ref{MiddleThirds}). Its
history graph is a trivalent tree except for a single vertex of valence 2 (see
Figure \ref{CantorHistoryGraph}).

 \begin{figure}
\scalebox{.35}{\includegraphics{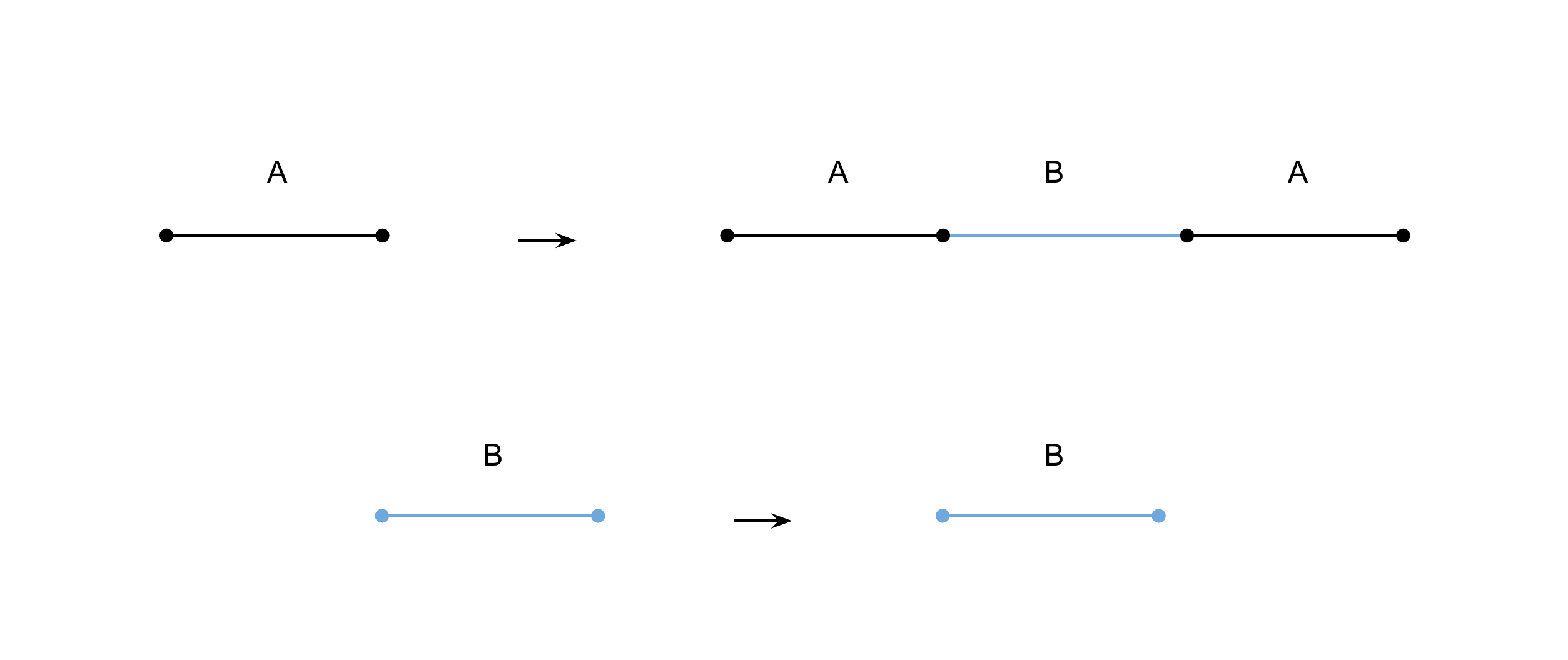}}
 \caption{The middle thirds subdivision rule
for the Cantor set. Type B is ideal.}
\label{MiddleThirds} \end{figure}
 \begin{figure}
 \scalebox{.35}{\includegraphics{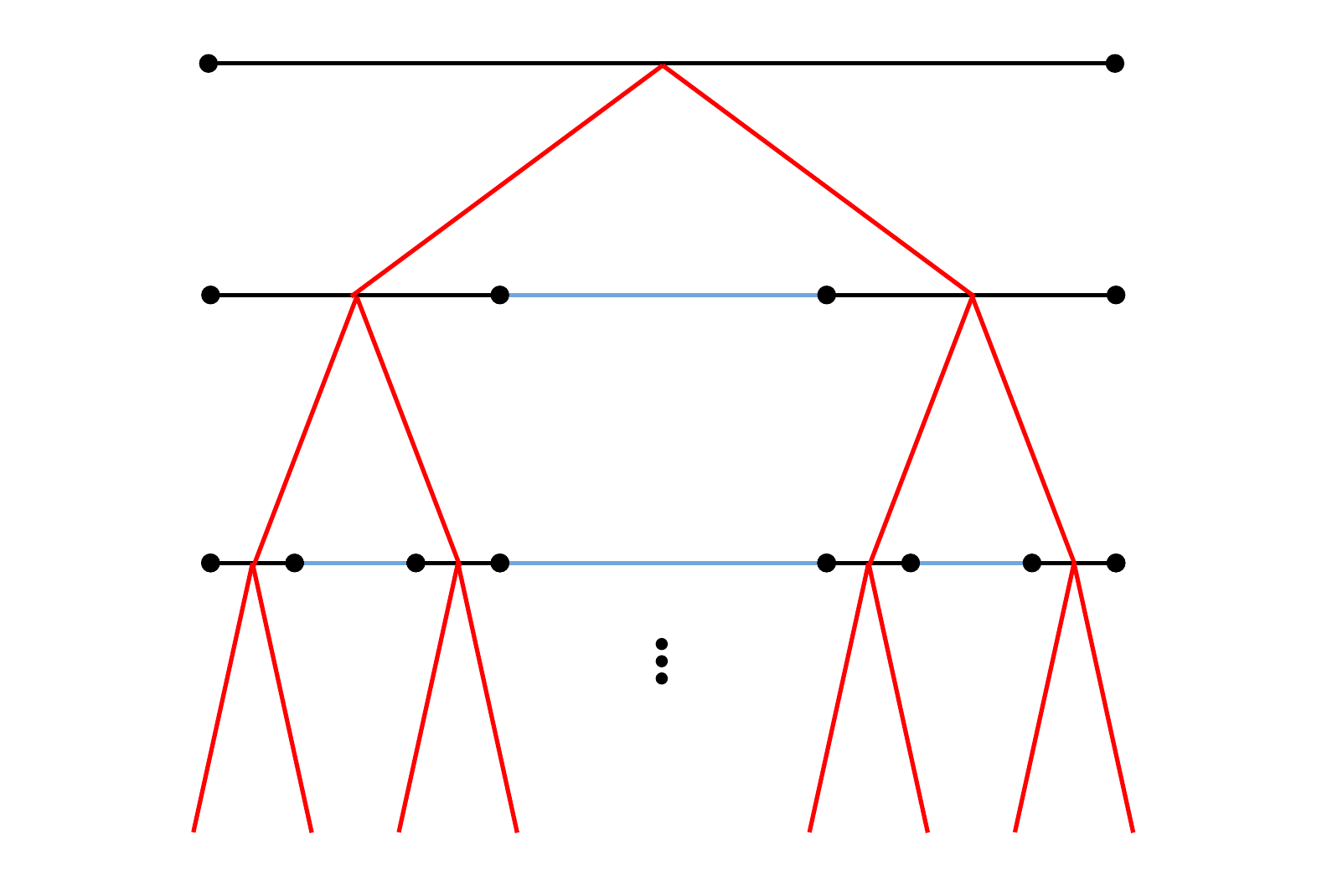}}
 \caption{The history graph of the middle
thirds Cantor set is an infinite trivalent tree (except for one vertex
at the top).}\label{CantorHistoryGraph}
\end{figure}

It is a consequence of Theorem \ref{BigTheorem}, stated in Section \ref{OutlineSection}, that the history graph of the subdivision rules for right-angled Artin groups (and cubulated groups) obtained in this paper are quasi-isometric to the Cayley graphs of those groups. For such groups, we have the following:

\begin{thm}\label{Quasi1Thm}
The growth function of $G$ is the number of non-ideal tiles in $R^n(X)$.
\end{thm}

\begin{thm}\label{Quasi2Thm}
If the subdivision rule $R$ has mesh approaching 0 (meaning that each path crossing non-ideal tiles eventually gets subdivided), then the group $G$ is $\delta$-hyperbolic.
\end{thm}

\begin{thm}\label{Quasi3Thm}
The number of ends of $G$ is the same as the number of components of $X\setminus \mathop{\bigcup}\limits_n R_I^n(X)$, where $R^n_I$ is the union of the ideal tiles of $R^n(X)$.
\end{thm}

\begin{thm}\label{Quasi4Thm}
If $R$ is a subdivision rule acting on a space $X$ associated to a group $G$, then the diameter $diam_X(n)$ of $R^n(X)$ is an upper bound on the divergence of $G$, i.e. $div_G(n)\preceq diam_X(n)$. Conversely, if there are 2 geodesics $\alpha,\beta$ in the history graph of $\{R^n(X)\}$ realizing $diam_X(n)$ (i.e. with $d_{R^n(X)}=diam_X(n)$) then $diam_X(n)\preceq div_G(n)$.
\end{thm}

The exact definitions of the terms used in these theorems as well as their proofs can be found in Section \ref{QuasiProofs}.
These four properties (growth, hyperbolicity, ends, and divergence) are among the most commonly used quasi-isometry invariants of groups:
\begin{enumerate}

\item Growth can be used to study the algebraic structure of groups; in fact, Gromov's theorem on groups of polynomial growth shows that such groups must be virtually nilpotent \cite{Gromov1981groups}. 

\item Although the only hyperbolic RAAG's are free groups (as all others contain a copy of $\mathbb{Z}^2$), hyperbolic subgroups of RAAG's are plentiful \cite{Agol}. One can begin looking for hyperbolic subgroups of RAAG's by searching for portions of the subdivision rule where the combinatorial mesh goes to 0. However, the converse of Theorem \ref{Quasi2Thm} is not true; hyperbolic groups can have subdivision rules with combinatorial mesh not approaching 0 (as in Section 2.6 of \cite{myself2})

\item The number of ends of a group will distinguish free groups and elementary groups (i.e. 2-ended groups) from other groups. Stallings' theorem \cite{stallings1971group} shows that groups with infinitely many ends are either free products with amalgamation over finite groups, or HNN extensions over finite subgroups.

\item Divergence is a somewhat newer invariant \cite{Gersten}, which is frequently useful in distinguishing groups when the simpler invariants fail. It essentially measures the growth in circumference of spheres of radius $n$ in the space.
\end{enumerate}

In all of these theorems, a quasi-isometry invariant of $G$ can be detected by counting tiles in the subdivisions of $X$. In general, subdivision rules will give a combinatorial version of every quasi-isometry invariant, but it may take a more complicated form than those given here.

These theorems do not allow us to distinguish between all quasi-isometry types of right-angled Artin groups or cubulated groups. The quasi-isometric classification of right-angled Artin groups is not yet known, although much progress has been made in recent years \cite{behrstock2012divergence, behrMann, behrMannJan, AsympGeom}. But Theorems \ref{Quasi3Thm} and \ref{Quasi4Thm} allow us to distinguish RAAG's with infinitely many ends or with linear divergence, which correspond to RAAG's that are free products or direct products, respectively\cite{behrstock2012divergence}. It is unknown what more complicated quasi-isometry invariants have simple analogues in subdivision rules.

As an sample application of these theorems, we show that special cubulated groups have a \textbf{growth dichotomy}:
\begin{thm}
If $G$ is a special cubulated group, then its growth is either exponential or polynomial.
\end{thm}
\begin{proof}
By Theorems \ref{BiggestTheorem} and \ref{Quasi1Thm}, the growth of $G$ is given by the number of tiles in $R^n(X)$ for some subdivision rule $R$ acting on a complex $X$. The number of tiles in each stage is given by a linear recurrence relation, i.e. if $N(i,j)$ is the number of tiles of type $i$ at stage $j$, then $N(i,j+1)=C_{i1}N(1,j)+C_{i2}N(2,j)+\cdots+C_{ik}N(k,j)$, where the $C_{ij}$ are independent of $j$. Such a linear recurrence relation can always be solved by standard linear algebra techniques (such as those in Section 5.1 of \cite{myself}) to give a function of either polynomial growth or exponential growth.
\end{proof}
This result was first proved by Hsu and Wise \cite{hsu1999linear} when they showed that all special cubulated groups are linear. Linear groups were previously known to have a growth dichotomy by work of Tits \cite{tits1972free}.

It is not known if the divergence of special cubulated groups has a similar dichotomy, although it is known that they can have exponential divergence or polynomial divergence of any given degree \cite{behrstock2012cubulated}.

\subsection{Future work}
These theorems can be pushed farther than we have done in this work. The author and David Futer have been able to show that there is a continuous group action on the spherical subdivision complex which generalizes the action of a hyperbolic group on its boundary, with analogs of the domain of discontinuity and limit set. This action has several other properties similar to the action of hyperbolic groups.

\subsection{Acknowledgements}

The author thanks David Futer for his careful reading of several drafts and for his numerous helpful suggestions.

\subsection{Outline}\label{OutlineSection}

We now give an outline of the remainder of the paper. The core theorem (whose terms will be defined in later sections) is the following:

\newtheorem*{BigTheorem}{Theorem \ref{BigTheorem}}
\begin{BigTheorem} Let $M$ be a right-angled manifold with a fundamental domain consisting of polytopes $P_1,...,P_n$. Then $M$ has a
finite subdivision rule. The tile types are in 1-1
correspondence with the facets of the inflations $I(\partial P_1)$,...,$I(\partial P_n)$. Each tile corresponding to a facet $I(K)\subseteq I(\partial P_i)$ is subdivided into a complex isomorphic to $\overline{I(\partial P_i)\setminus IS(K)}$, i.e. the complement of the inflated star. \end{BigTheorem}

Most of the paper is devoted to the proof of this theorem, from which Theorems \ref{BiggerTheorem} and \ref{BiggestTheorem} follow easily.

In Section \ref{FormalSection}, we give the formal definition of a subdivision rule. This section is technical, and may be omitted on first reading.

In Section \ref{GeneralSection}, we outline the general strategy for creating subdivision rules from right-angled objects. We illustrate the proof strategy by two fundamental examples in Section \ref{ExampleSection}, which will be referred to frequently. Section \ref{GluingCollapseSection} develops the two tools (gluing and collapse) that will be used in the proof of Theorem \ref{BigTheorem}.

Section \ref{MainSection} is dedicated to the proof of Theorem \ref{BigTheorem}.

In Section \ref{RAAGSection}, we create a complex for right-angled Artin groups that satisfies the conditions of Theorem \ref{BigTheorem}, and thus obtain subdivision rules for RAAGs. In Section \ref{CubulatedSection}, we extend these results to all cubulated groups, obtaining Theorem \ref{BiggestTheorem}.

Finally, Section \ref{QuasiProofs} contains the proofs of Theorems \ref{Quasi1Thm}-\ref{Quasi4Thm} on quasi-isometry invariants of subdivision rules.

\section{Formal Definition of a Subdivision Rule}\label{FormalSection}

At this point, it will be helpful to give a concrete definition of
subdivision rule. Cannon, Floyd and Parry gave the first
definition of a finite subdivision rule (for instance, in
\cite{subdivision}); however, their definition only applies to
subdivision rules on the 2-sphere, as this is the main case of interest
in Cannon's Conjecture. In this paper, we find subdivision rules for
RAAG's that subdivide or act on the $n$-sphere. In \cite{CubeSubs}, we
defined a subdivision rule in higher dimensions in a way analogous to
subdivision rules in dimension 2. We repeat that definition here. A
\textbf{finite subdivision rule $R$ of dimension $n$} consists of:
\begin{enumerate} \item A finite $n$-dimensional CW complex $S$,
called the \textbf{subdivision complex}, with a fixed cell structure
such that $S$ is the union of its closed $n$-cells (so that the complex is pure dimension $n$). We assume that for
every closed $n$-cell $\tilde{s}$ of $S$ there is a CW structure $s$
on a closed $n$-disk such that any two subcells that intersect do so in
a single cell of lower dimension, the subcells of $s$ are contained in
$\partial s$, and the characteristic map $\psi_s:s\rightarrow S$ which
maps onto $\tilde{s}$ restricts to a homeomorphism onto each open cell.
\item A finite $n$-dimensional complex $R(S)$ that is a subdivision of $S$.
\item A \textbf{subdivision map} $\phi_R: R(S)\rightarrow S$, which is a
continuous cellular map that restricts to a homeomorphism on each open
cell. \end{enumerate}

Each cell $s$ in the definition above (with its appropriate
characteristic map) is called a \textbf{tile type} of $S$. We will often describe an
$n$-dimensional finite subdivision rule by the subdivision of every tile
type, instead of by constructing an explicit complex.

Given a finite subdivision rule $R$ of dimension $n$, an $R$-\textbf{complex}
consists of an $n$-dimensional CW complex $X$ which is the union of its
closed $n$-cells together with a continuous cellular map $f:X\rightarrow
S$ whose restriction to each open cell is a homeomorphism. All tile
types with their characteristic maps are $R$-complexes.

We now describe how to subdivide an $R$-complex $X$ with map
$f:X\rightarrow S$, as described above. Recall that $R(S)$ is a
subdivision of $S$. We simply pull back the cell structure on $R(S)$
to the cells of $X$ to create $R(X)$, a subdivision of $X$. This gives
an induced map $f:R(X)\rightarrow R(S)$ that restricts to a
homeomorphism on each open cell. This means that $R(X)$ is an
$R$-complex with map $\phi_R \circ f:R(X)\rightarrow S$. We can
iterate this process to define $R^n(X)$ by setting $R^0 (X) =X$ (with
map $f:X\rightarrow S$) and $R^n(X)=R(R^{n-1}(X))$ (with map $\phi^n_R
\circ f:R^n(X)\rightarrow S$) if $n\geq 1$.

We will use the term `subdivision rule' throughout to mean a finite
subdivision rule of dimension $n$ for some $n$. As we said earlier, we
will describe an $n$-dimensional finite subdivision rule by a
description of the subdivision of every tile type, instead of by
constructing an explicit complex.

\section{General strategy}\label{GeneralSection}

In the following sections, we will describe subdivision rules for
various right-angled objects. Our goal is to find subdivision rules for right-angled Artin groups, but the easiest way is to first find subdivision rules for a more general class of right-angled objects, described in the next section. We will use as as examples two manifolds whose fundamental groups are right-angled Artin groups. These examples
provide valuable intuition and motivation, and expand on previous
results.

\subsection{Right-angled objects}

An abstract \textbf{polytope} $P$ of dimension $n$ is a CW-complex whose barycentric subdivision is a simplicial complex (this presupposes that a barycentric subdivision exists, which requires that every characteristic map of a $k$-cell extends to an embedding of the closed $k$-cell). A \textbf{facet} of a polytope is the closure
of a codimension-1 cell. A \textbf{ridge} is the closure of a cell of
codimension 2.

For the purposes of this paper, a \textbf{right-angled manifold of
dimension $d$} is a manifold with a fundamental domain that is an
abstract polytope, where the manifold and the fundamental domain are
locally modeled on the $d$-torus and its fundamental domain of a cube.
Thus, we require the link of each vertex of the polytope to be a
simplex. We require all facets to glue up in pairs, we require each all
ridges (which, we recall, are codimension-2 cells of the polytope) to be
glued together in groups of four (hence the name right-angled). The link
of each vertex in the universal cover (or in the manifold) will be a
$(d-1)$-dimensional orthoplex.

We also consider \textbf{right-angled manifolds with corners}, where we
allow the link of each vertex in the manifold to be either an orthoplex
or half of an orthoplex, or even a half of a half of an orthoplex
(essentially, these manifolds are manifolds with boundary that can be thought of as lying inside a right-angled manifold without boundary).

As a regularity condition, we only consider polytopes without
1- or 2-circuits and without prismatic 3-circuits. A
$k$\textbf{-circuit} is a chain of distinct facets $A_1,...,A_k$ such
that each pair of neighboring facets $A_i,A_{i+1}$ intersect in a ridge
(as well as $A_1$ and $A_k$). A circuit is \textbf{prismatic} if the set
of all such ridges is pairwise disjoint. Prismatic $k$-circuits for $k \leq 3$ correspond to
positive curvature in a sense, and are thus prohibited (for instance, a plane intersecting a prismatic 3-circuit would result in a triangle with three right-angles, a positively curved object).

One of the important properties of right-angled polytopes is that any
two facets that intersect at all share a ridge. This is due to the link
of each vertex being a simplex.

Another of the most important properties of a right-angled polytope is
the following:

\begin{lemma}\label{PolytopeLemma} Let $P$ be a right-angled polytope.
Then every cell of dimension $d-k$ is the unique intersection of the $k$
facets of $P$ containing it. Conversely, for every set $A_1,...,A_{k}$ of
$k$ mutually intersecting facets the intersection $\mathop{\cap} A_i$ is a single
cell of dimension $d-k$. \end{lemma}

\begin{proof} The lemma is
locally true, since the link of each vertex is a simplex. Also, no facet
can intersect the same vertex twice. Near each vertex, there can only be
a single cell of dimension $d-k$ in the intersection; if there were two
such cells that were disjoint, there would be a prismatic 2-circuit,
which is not allowed. \end{proof}

\subsection{Method of obtaining a subdivision rule}\label{MethodSection}

Given a subdivision rule $R$ and a complex $X$, one can obtain a
sequence of complexes $R^n(X)$ which are all homeomorphic to $X$ and
where each $R^{n+1}(X)$ is a subdivision of $R^n(X)$ (after identifying via the
homeomorphisms). Our strategy is to reverse this process; we produce a
sequence of complexes $R^n(X)$ such that for each $n$: \begin{enumerate}
\item $R^n(X)$ is homeomorphic to $X=R^0(X)$, and \item $R^{n+1}(X)$ is
a subdivision of $R^n(X)$ after identifying both with $X$.
\end{enumerate}

From such a sequence of spaces we can extract a finite subdivision rule by
showing that there are only finitely many ways a cell in some $R^k(X)$
will be subdivided in $R^{k+1}(X)$.

The tilings are obtained by looking at larger and larger balls in the
universal cover. Let $B(0,0)$ be a single fundamental domain in the
universal cover. Let $B(0,1)$ be the union of $B(0,0)$ and all
fundamental domains sharing a facet with $B(0,0)$. Let $B(0,2)$ be
$B(0,1)$ together with all fundamental domains intersecting ridges
(codimension-2 cells) of $B(0,0)$. In general, let $B(1,j)$ be the union
of $B(0,j-1)$ with all fundamental domains intersecting codimension-$j$
cells of $B(0,0)$.

Now, let $B(i,0)=B(i-1,d)$, and let $B(i,j)$ be the union of $B(i,j-1)$ with all fundamental domains sharing
codimension-$j$ cells with $B(i,j-1)$.

We let $S(i,j)$ be the boundary of $B(i,j)$. In many situations, this
boundary will be a sphere, and it is this sphere (or sequence of
spheres) that we will subdivide.

\textbf{Note:} We frequently use the word \textbf{wall} in this paper.
There are several interpretations of this word in the literature. We use
the word wall to mean a hyperplane in the dual cube complex of a
complex, or equivalently an extension of a facet in the original
complex. This terminology will be used throughout.

\subsection{Convex, flat, and concave ridges} As we said above,
right-angled manifolds are locally modelled on the $n$-torus (away from
the boundary). When constructing the universal cover, the neighborhood
of each ridge looks locally like Figure \ref{LocalPicture}, where the
ridge corresponds to the vertex in the center. In the universal cover of
the 2-dimensional torus, this picture is exactly the neighborhood of the ridge (a
vertex). In all other spaces, the neighborhood of the ridge looks like the product of
this picture with a cube of dimension $d-2$.
\begin{figure}
\scalebox{.35}{\includegraphics{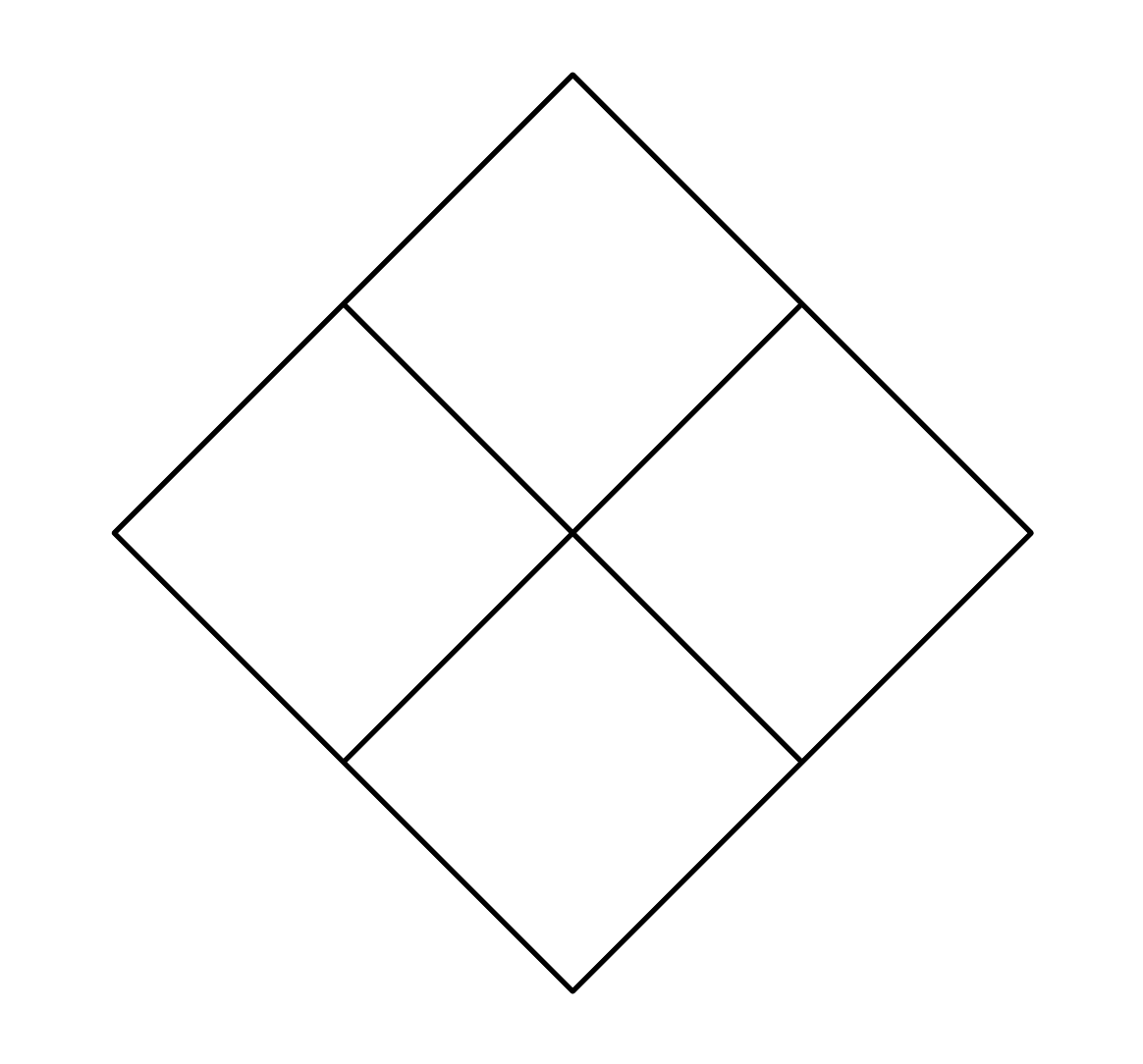}} \caption{The local picture about a ridge
in the universal cover.}
\label{LocalPicture}
\end{figure}

In constructing the balls $B(i,j)$, each ridge goes through one of a few
predictable progressions. The simplest is shown in Figure
\ref{UnloadedProgression}. A ridge begins as the corner of a single
fundamental domain. We then glue fundamental domains onto each of the
neighboring facets, so that there are 3 fundamental domains about the
ridge. Finally, we place one more fundamental domain on to cover up the
ridge. We call the ridges touching one fundamental domain in $B(i,j)$
\textbf{convex}. We call ridges touching three fundamental domains
\textbf{concave}. Finally, those with four fundamental domains no longer
appear in the boundary $S(i,j)$, so we ignore them.

\begin{figure}
\scalebox{.35}{\includegraphics{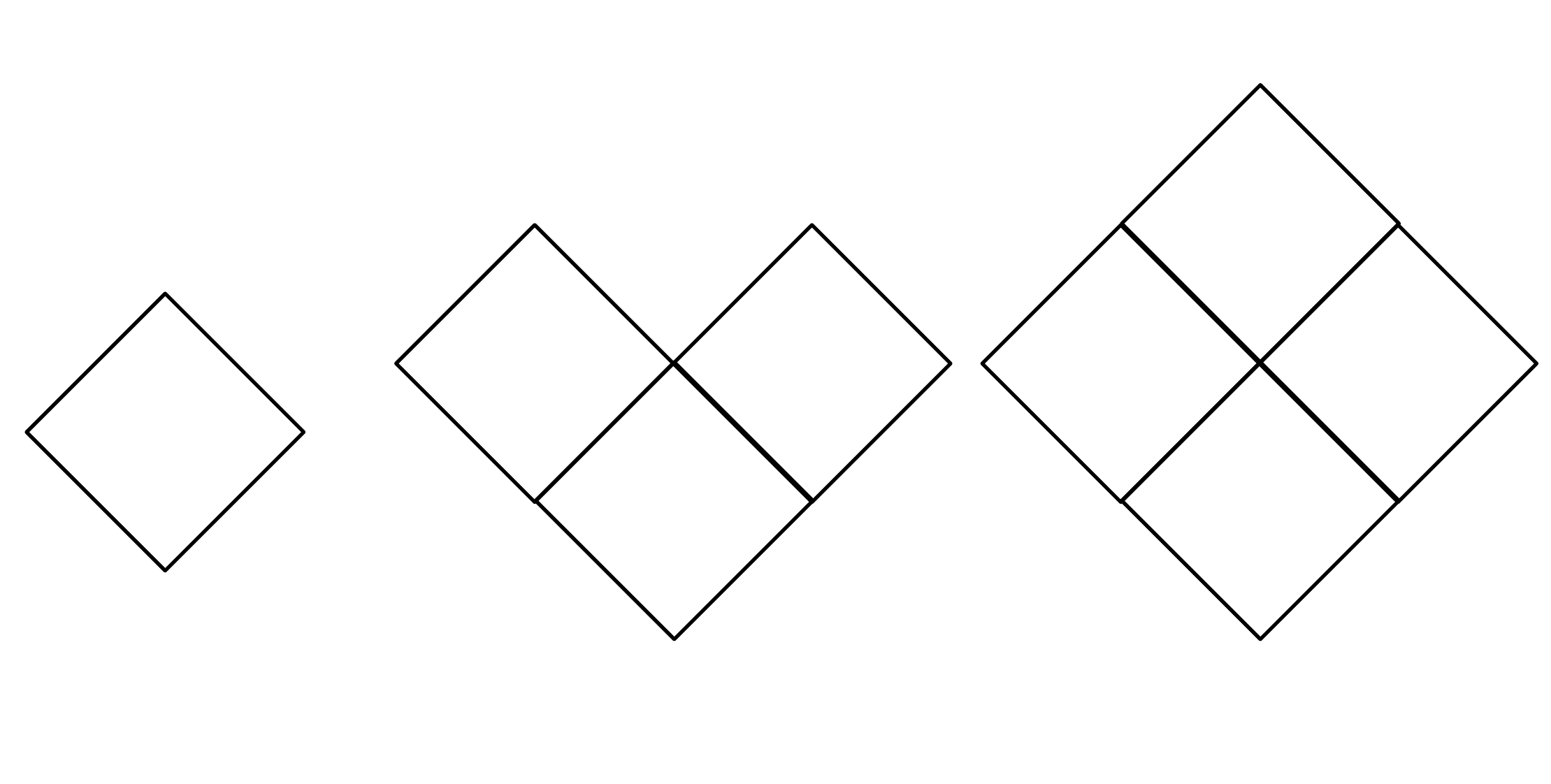}}\caption{The progression about a convex
ridge. It goes from convex, to concave, to covered
up.}\label{UnloadedProgression} \end{figure}

It will happen often that we get a ridge that touches only two
fundamental domains. See Figure \ref{FragilePicture}. Such ridges are
called \textbf{flat}. When we glue a fundamental domain onto both facets
containing a flat ridge, the new facets to either side must be identified
to satisfy the local homeomorphism condition, and we say the facets
\textbf{collapse}.
 \begin{figure}
\scalebox{.35}{\includegraphics{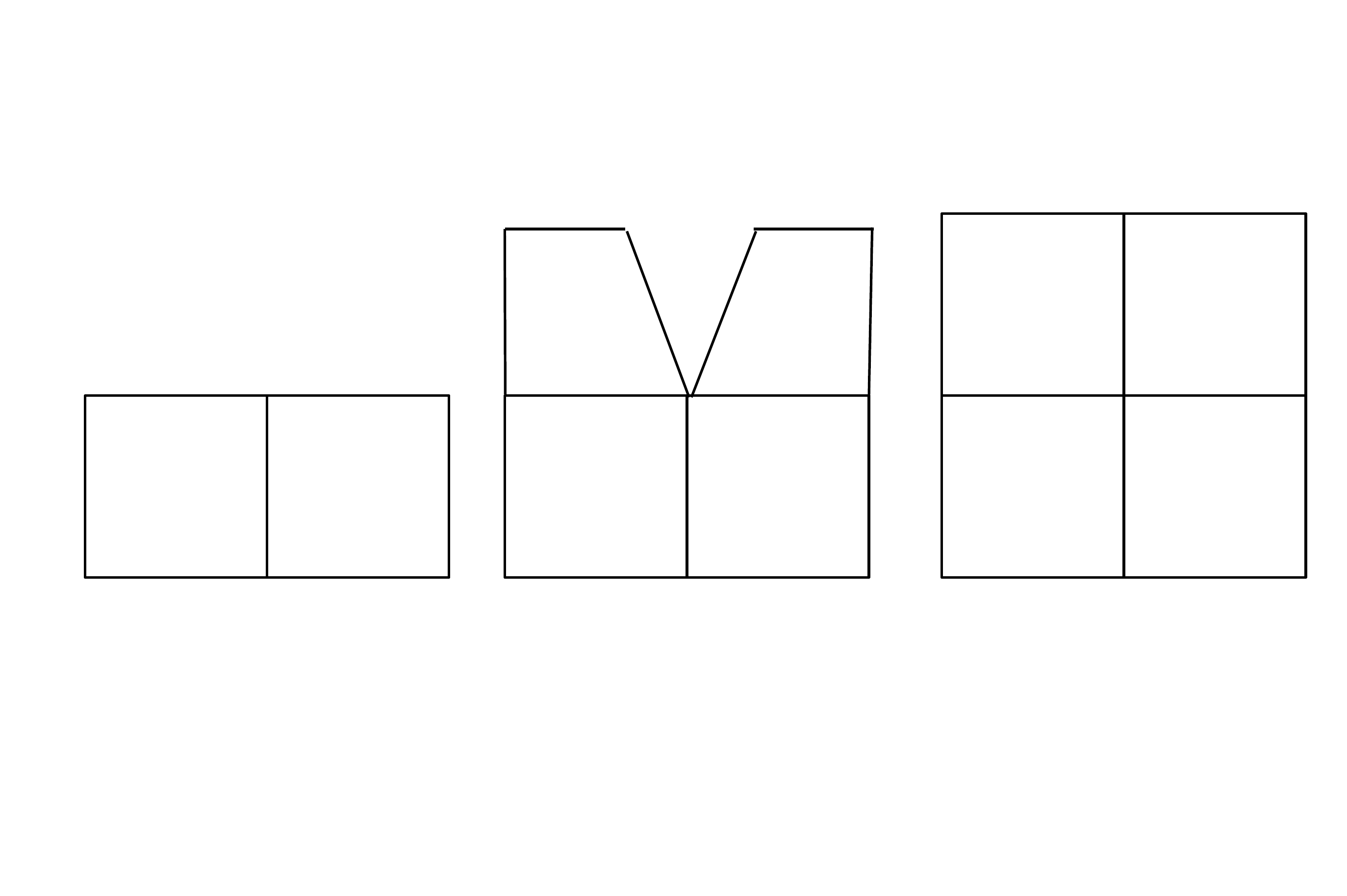}} \caption{The progression
about a flat ridge. It goes from flat, on the left, to covered up, on
the right. The second picture is not an independent step, it just shows
that we should identify the two diagonal edges to get the third
}
\label{FragilePicture}
\end{figure}

As described earlier, we also have boundary facets, which are facets that
never get identified with anything. We call these \textbf{ideal} facets.
The ridges in their intersection with non-ideal facets are called
\textbf{ideal} ridges. They are identified in groups of two, and they
follow the progression in Figure \ref{FragileIdealPicture}. When an
ideal ridge is contained in two ideal facets of $S(i,j)$, we say it is
\textbf{flat ideal}; it doesn't collapse ever, but the name reminds us
of the similarity with flat ridges. We often use the word `flat' to include flat ideal ridges.
\begin{figure}\label{FragileIdealPicture}
\scalebox{.35}{\includegraphics{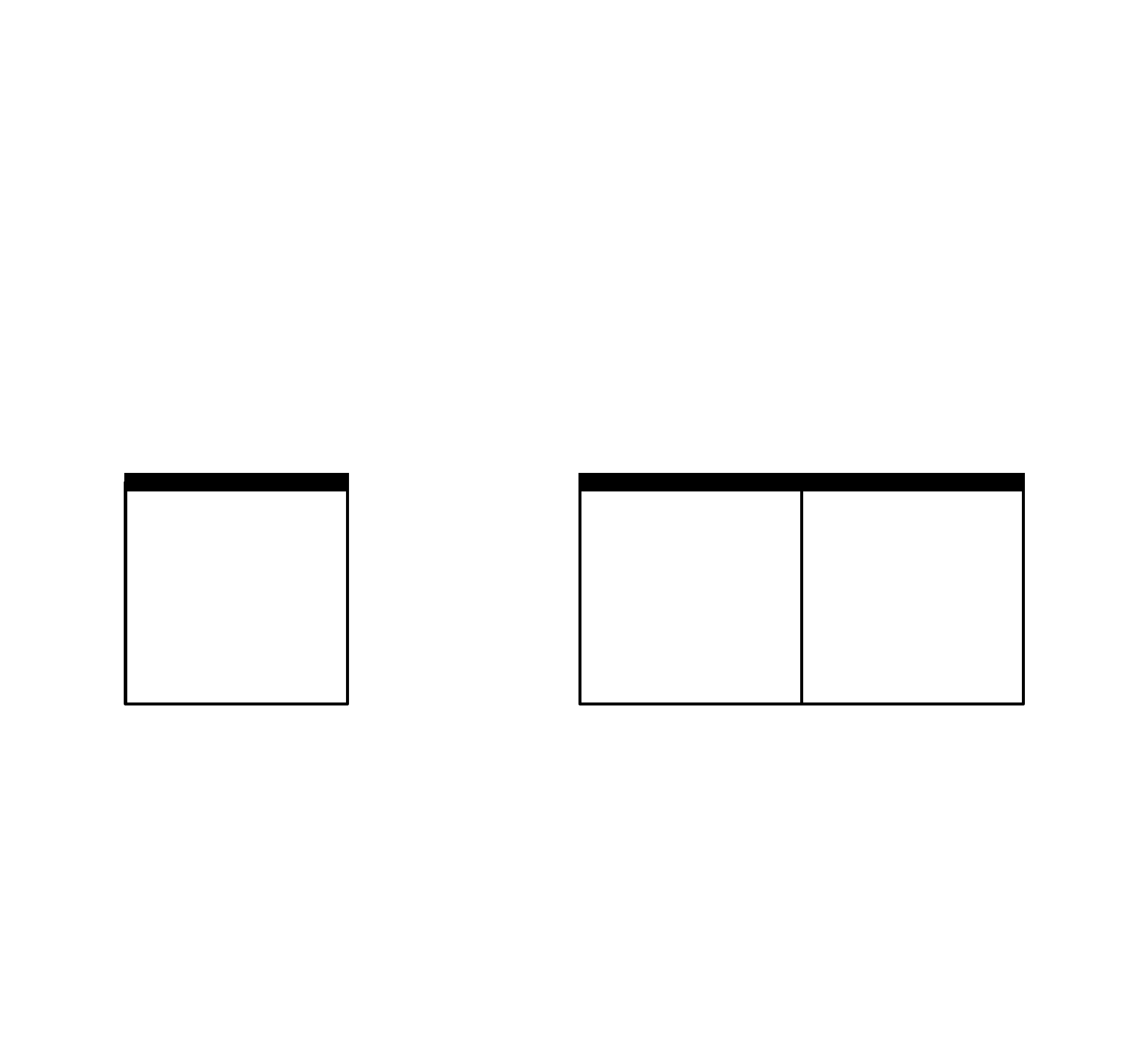}} \caption{The progression about an ideal ridge. It goes
from ideal to flat ideal.}\end{figure}

\section{Examples}\label{ExampleSection}

We now present two fundamental examples, both of which are 3-manifolds with fundamental groups that are right-angled Artin groups. The first is the 3-torus, whose fundamental group is $\mathbb{Z}^3$. The second is the product of a punctured torus with a circle, whose fundamental group is $\mathbb{Z}\times F_2$, the product of the integers with the free group on two generators. These examples should be referred to frequently, and represent almost the entire spectrum of possibilities in creating these subdivision rules.

\subsection{The 3-torus}\label{CubeSection}

 \begin{figure}
\scalebox{.35}{\includegraphics{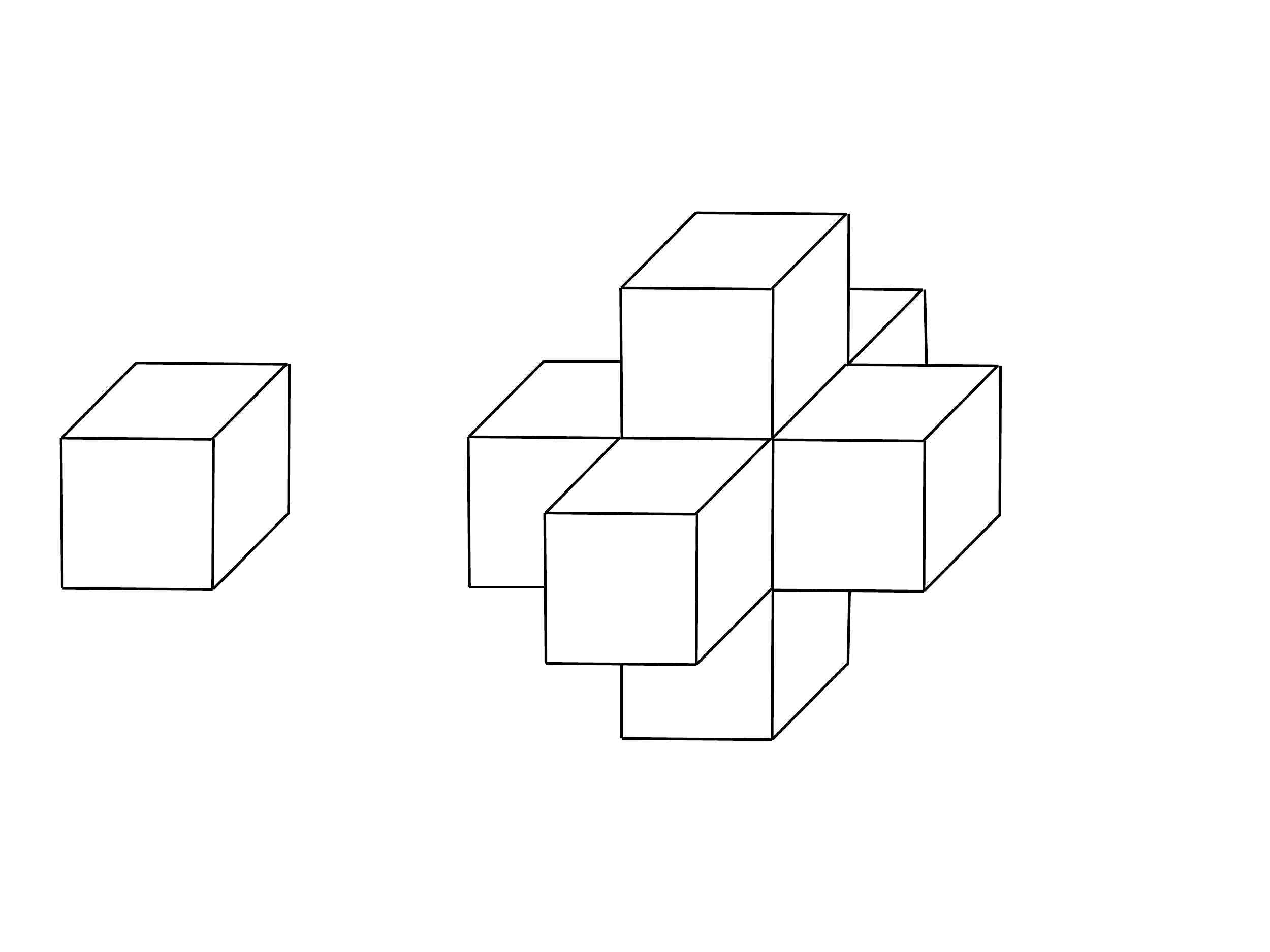}} \caption{The stages $S(0,0)$, $S(0,1)$ for
the cube.}\label{StageS11}\end{figure}

 \begin{figure} \scalebox{.3}{\includegraphics{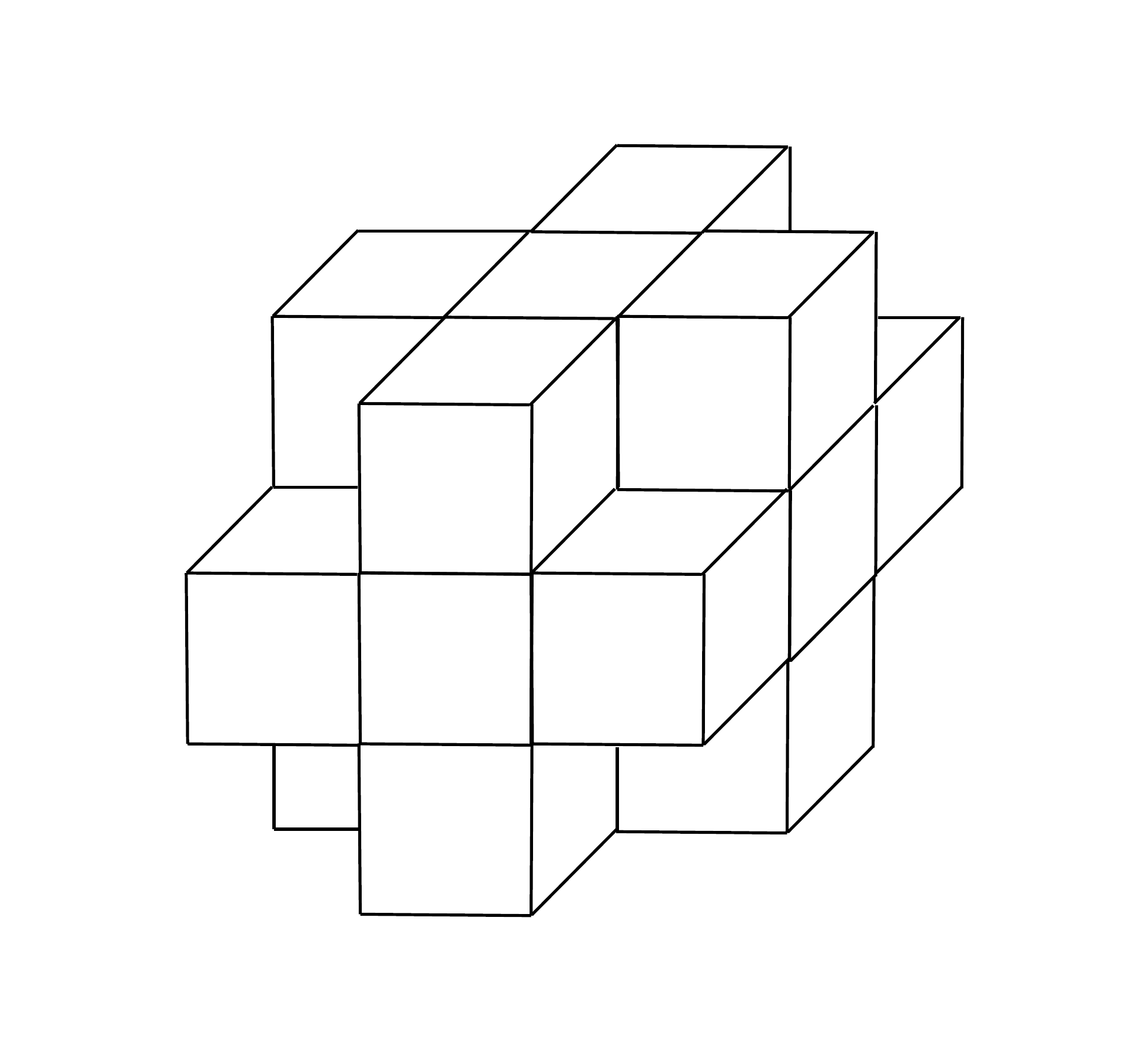}}
\caption{The stage $S(0,2)$ for the
cube.}\label{StageS12}\end{figure}

 \begin{figure}
\scalebox{.3}{\includegraphics{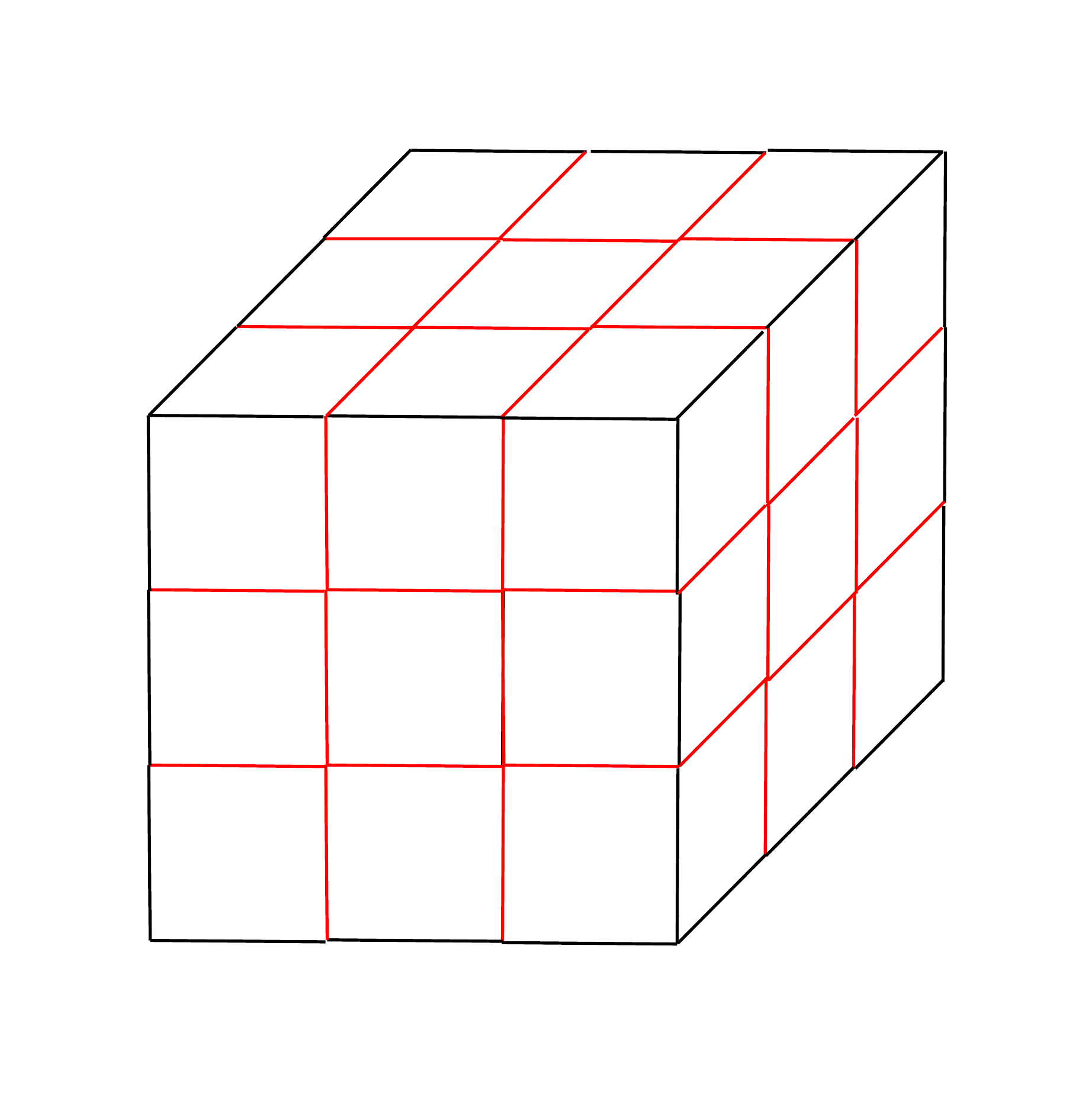}}\caption{The stage $S(0,3)=S(1,0)$ for the
cube, with all flat edges marked.}\label{FlatS0} \end{figure}

 \begin{figure}
\scalebox{.3}{\includegraphics{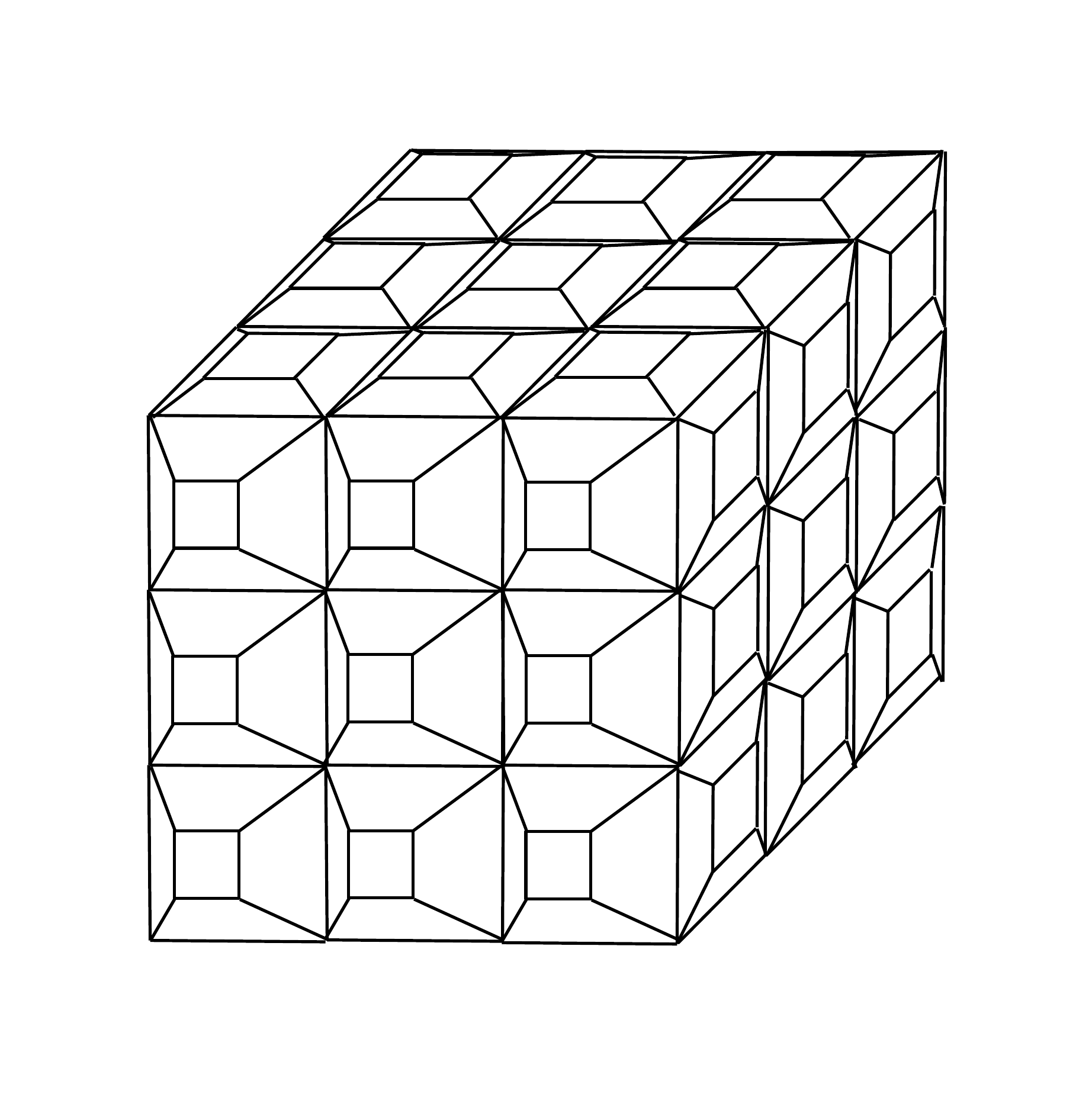}} \caption{The stage $S(1,1)$ for the cube,
before collapsing edges}\label{Stage21Prime}\end{figure}

 \begin{figure}
\scalebox{.25}{\includegraphics{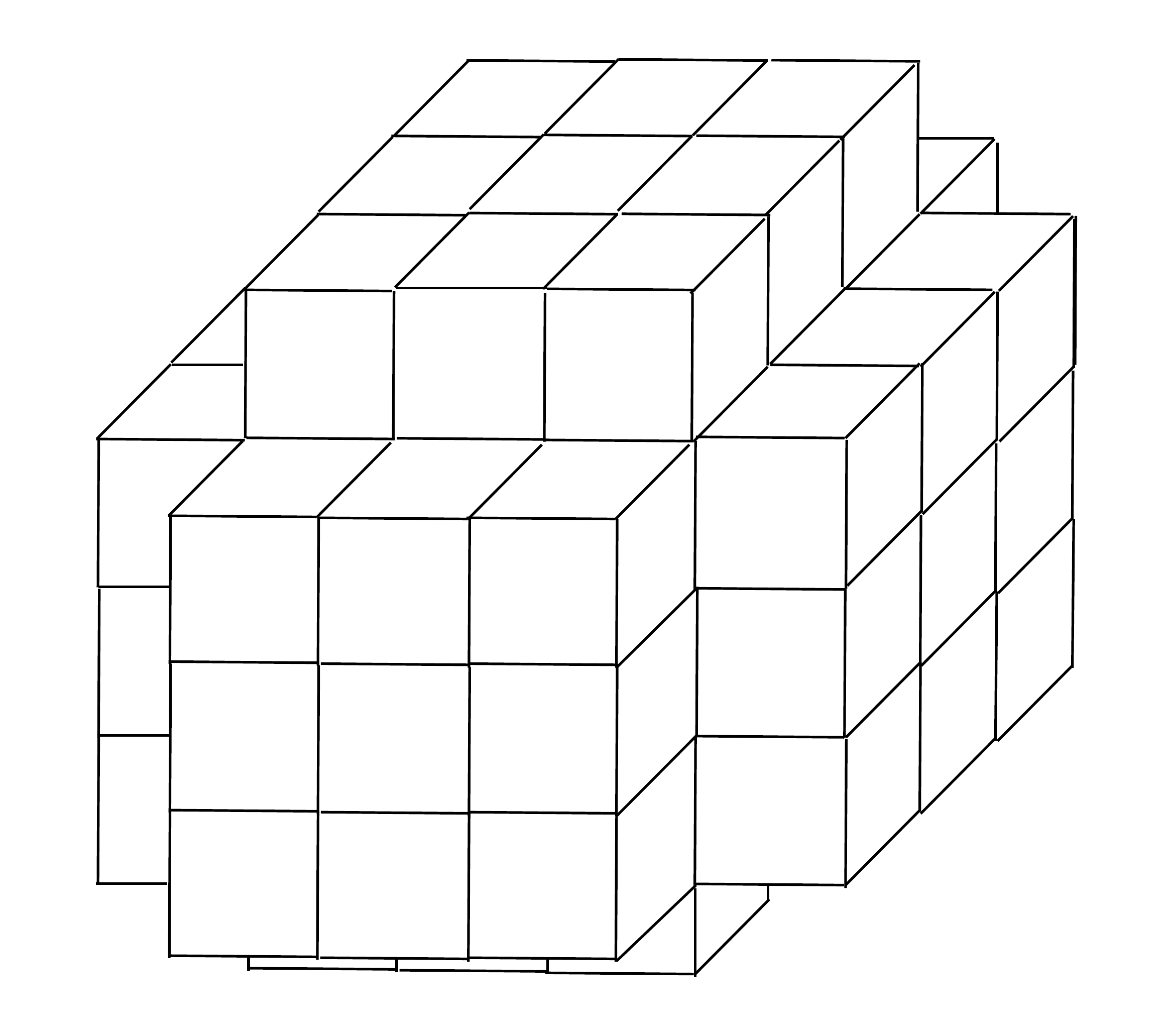}} \caption{The stage $S(1,1)$ for the cube,
after collapsing edges.}\label{Stage21}\end{figure}

 \begin{figure}
\scalebox{.25}{\includegraphics{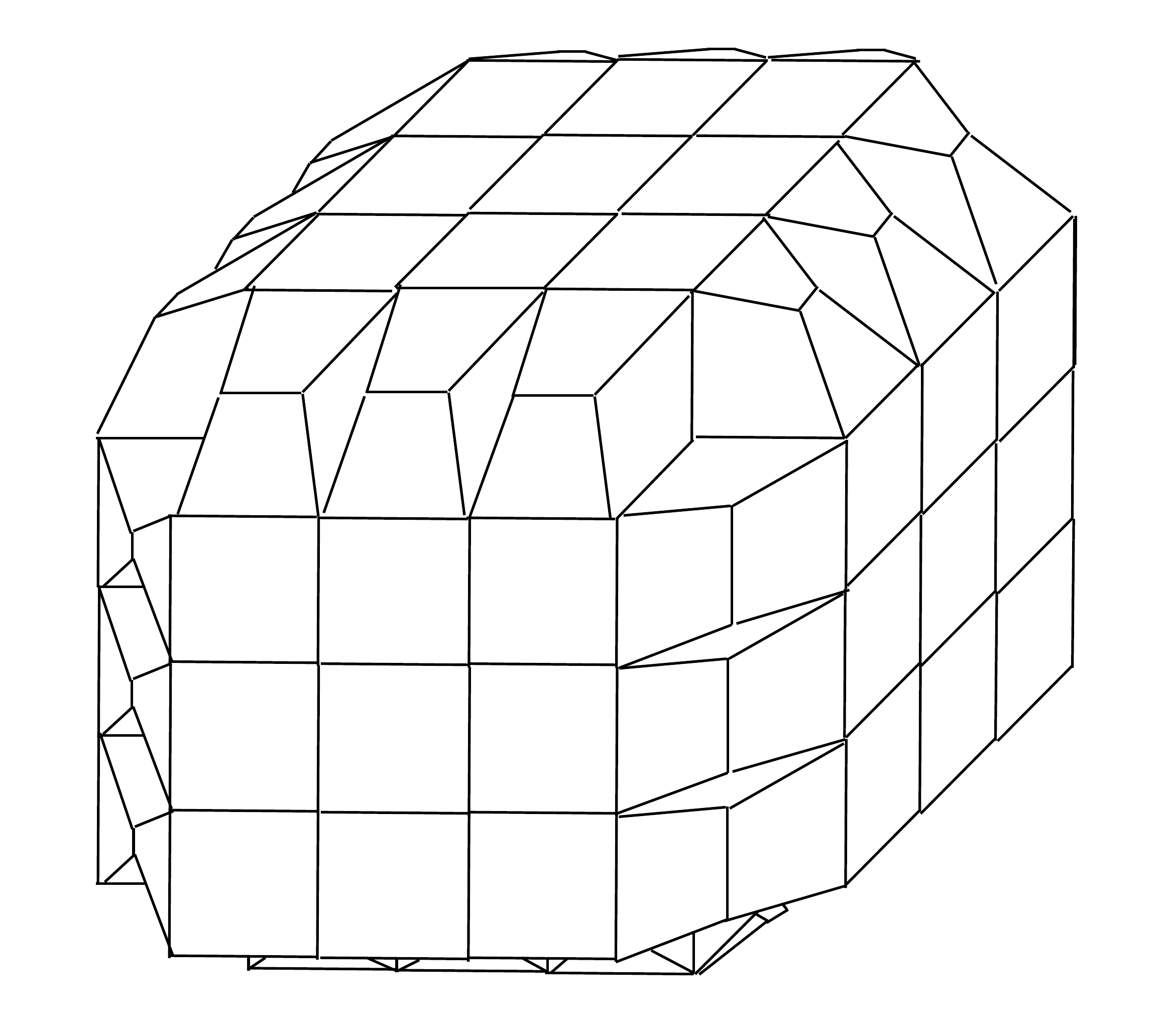}} \caption{The stage $S(1,2)$ for the cube,
before collapsing edges}\label{Stage22Prime}\end{figure}

 \begin{figure}
\scalebox{.25}{\includegraphics{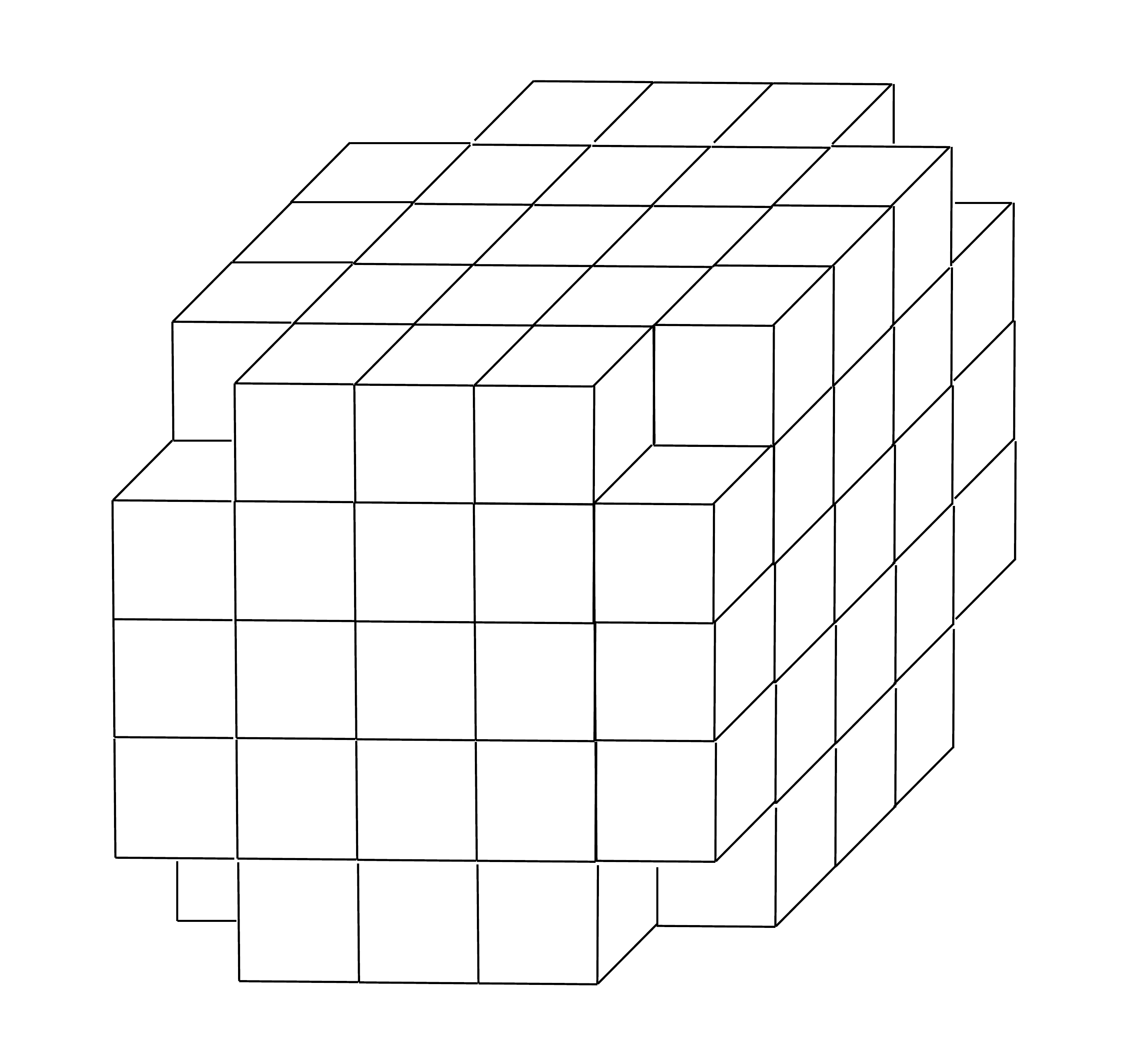}} \caption{The stage $S(1,2)$ for the cube,
after collapsing edges.}\label{StageS22}\end{figure}

 \begin{figure}
\scalebox{.25}{\includegraphics{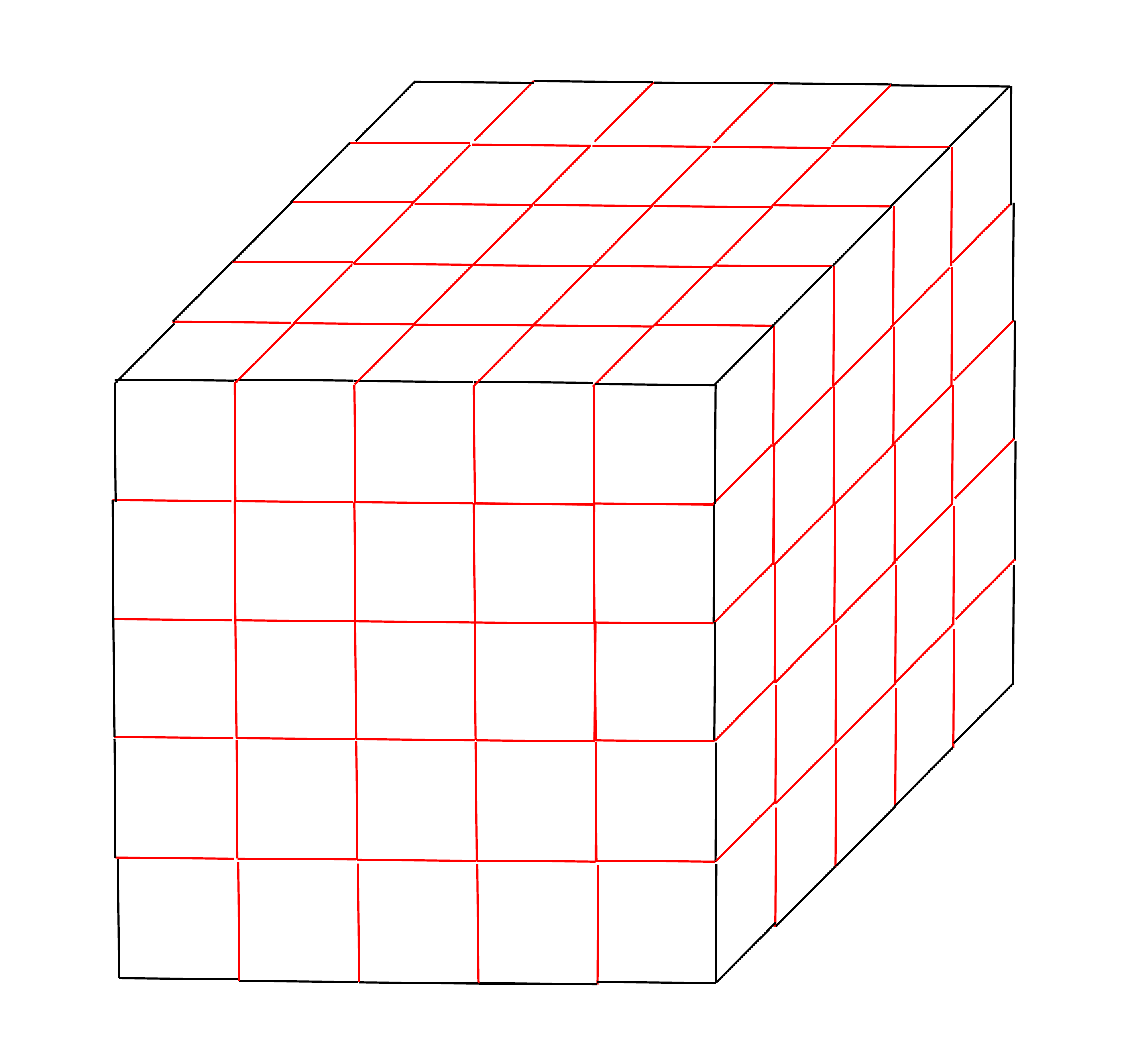}}\caption{The stage $S(1,3)=S(2,0)$ for the
cube, with all flat edges marked.}\label{FlatS1} \end{figure}

As an example of our general strategy, consider the 3-torus, with
fundamental domain a cube. Stages $S(0,0)$ through $S(1,3)$ are shown in
Figures \ref{StageS11} through \ref{FlatS1}.

Stage $S(0,0)$ is a single cube, with all edges convex. We create stage
$S(0,1)$ by gluing a cube onto every exposed face. This makes all edges
from $S(0,0)$ concave and all new edges convex.

Concave edges are associated with pairs of faces; notice that at a
concave edge, only one cube can be placed, and it covers two faces. Each
such pair of faces is called a \textbf{concave pair}. We glue cubes onto them to
create $S(0,2)$ (as in Figure \ref{StageS12}). Now some edges are
flat. Notice that, in this case, the flat edges are exactly the intersections of the
new cubes with the concave pairs. Notice also that we have more
concave edges, which now gather in groups of three; each set of three
concave faces together with the three concave edges is called a
\textbf{concave triple}. The new concave edges which are
part of the concave triples are also boundaries of concave pairs, just
like the flat edges. We glue a single cube onto each concave triple to
create $S(0,3)=S(1,0)$ (shown in Figure \ref{FlatS0}). All edges now are either convex or
flat, and every flat in this stage edge was a boundary of a concave pair or
concave triple at some previous stage.

In creating $S(1,1)$, we again glue a single cube onto every exposed
face, as shown in Figure \ref{Stage21Prime}. There are no concave edges, so every cube is glued onto a single
face. However, because there are several flat edges, many faces
`collapse', meaning that they are identified together in pairs. This
concept is described in more detail in Section \ref{CollapseSection}.
All faces of cubes in $S(1,1)$ that touch a flat edge of $S(1,0)$
collapse and disappear, forming the complex shown in Figure \ref{Stage21}. There are also several concave edges in
$S(1,1)$; these are exactly the convex edges of $S(1,0)$ (all of which
are still `peeking through' in $S(1,1)$) , and the faces containing them
again meet up in concave pairs, some of which have flat edges. To create
$S(1,2)$, we glue a single cube onto each concave pair, and collapse
flat faces, as shown in Figures \ref{Stage22Prime} and \ref{StageS22}; note that in this situation, each face that collapses has 2
flat edges, and shares both flat edges with the same face. Finally, we
have concave triples in $S(1,2)$, none of which have any flat edges. We
glue on a cube to each concave triple, and start over again at
$S(1,3)=S(2,0)$, shown in Figure \ref{FlatS1}.

\begin{figure}
\scalebox{.5}{\includegraphics{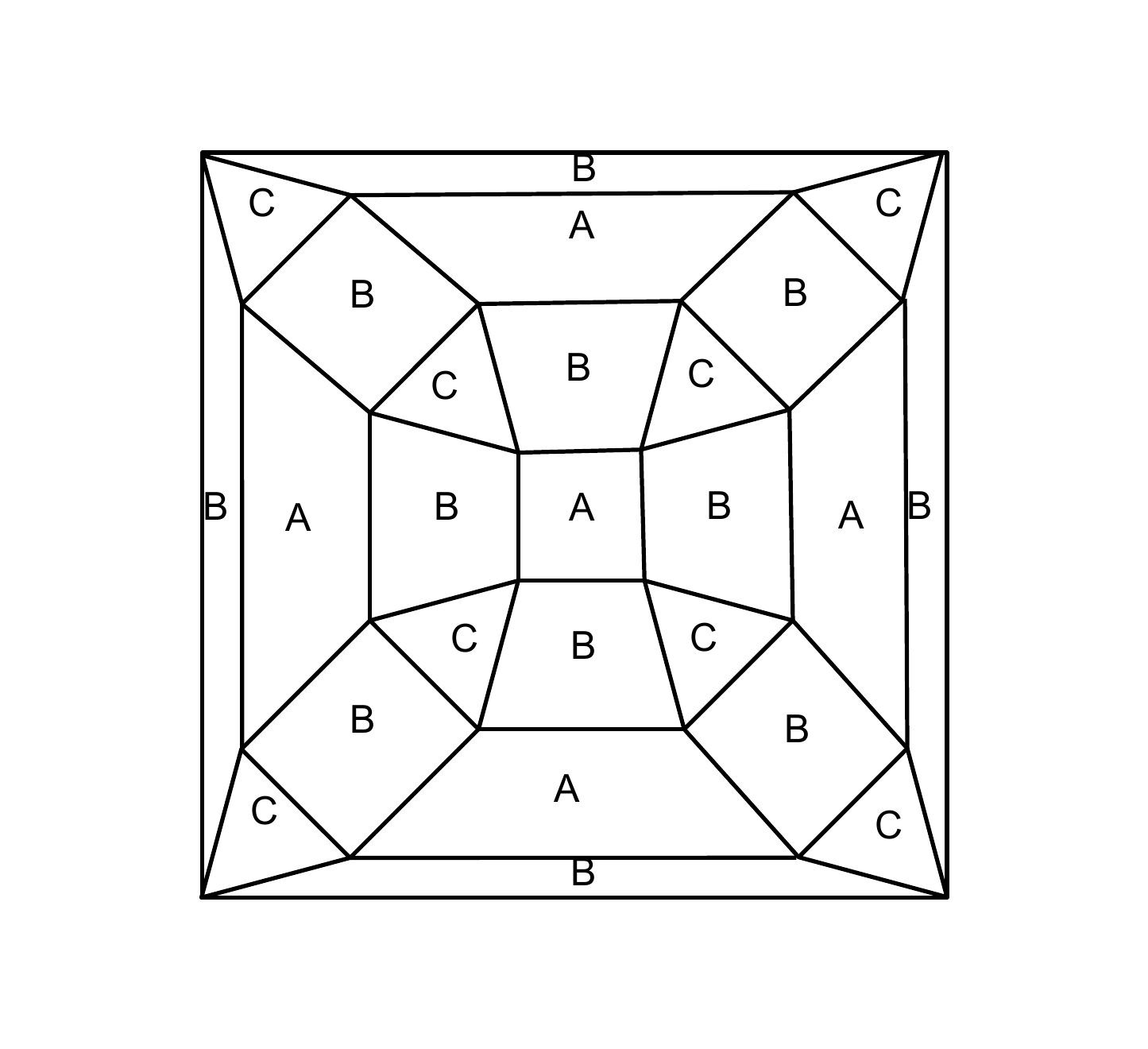}} \caption{The complex $X$ that the
subdivision rule of Figure \ref{TorusSubdivision} acts on. The cell
structure is the same as that given by the flat edges in Figure
\ref{FlatS0}. The `outside face' is a type $A$ tile. }\label{TorusS0}\end{figure}

\begin{figure}
\scalebox{.5}{\includegraphics{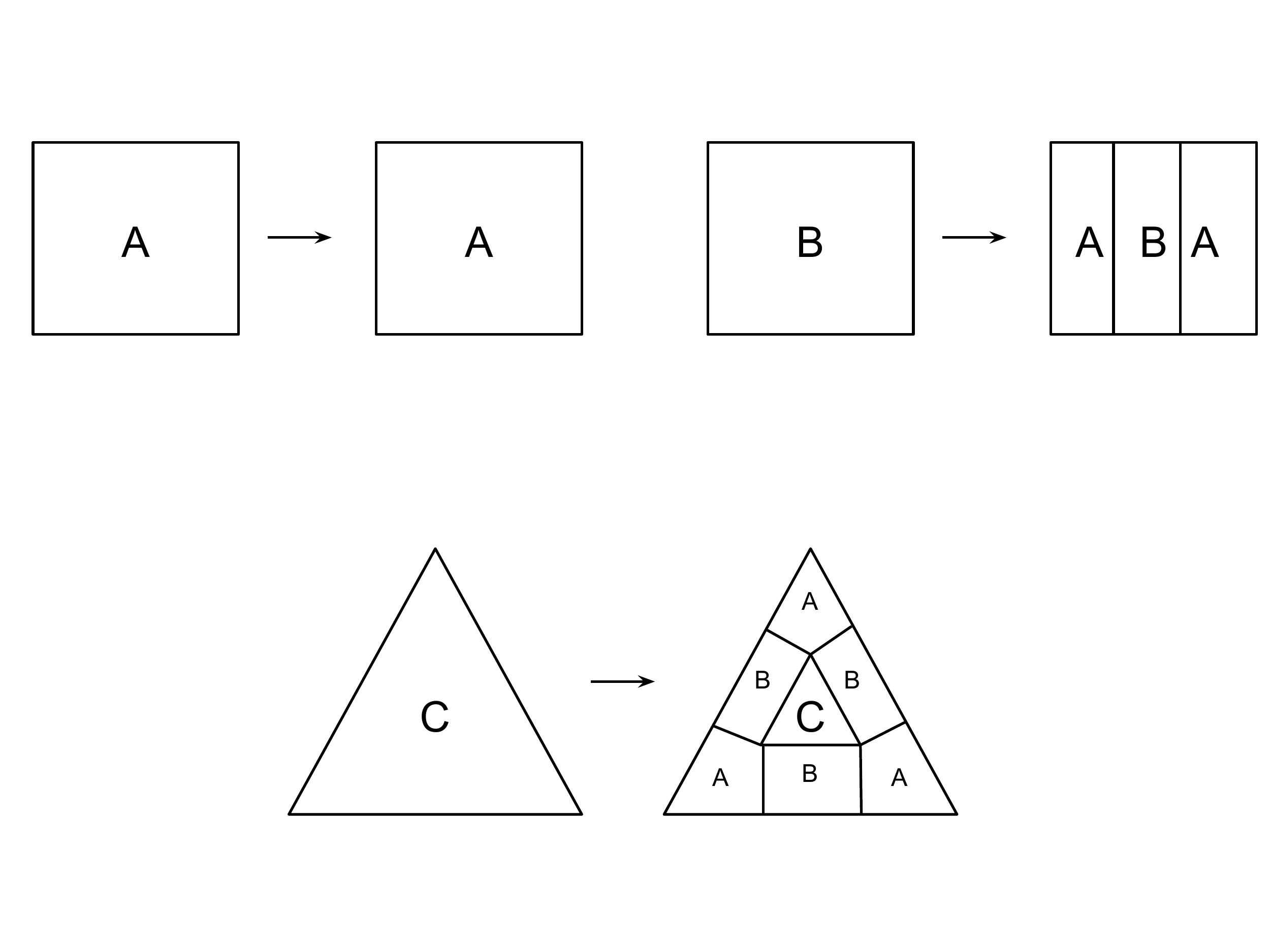}} \caption{The three tile types for the
subdivision rule of the torus.}\label{TorusSubdivision}\end{figure}

\begin{figure}
\scalebox{.5}{\includegraphics{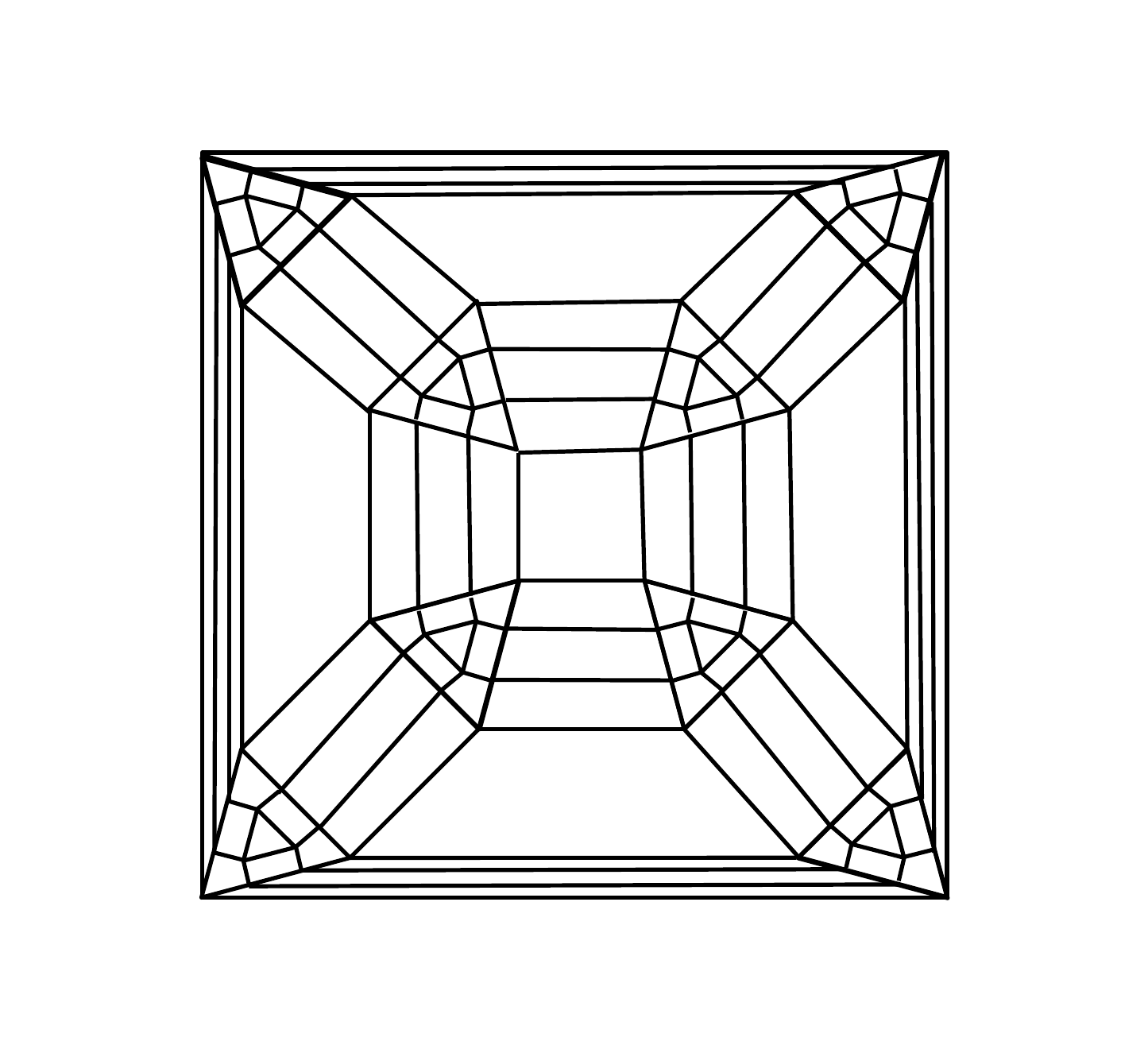}} \caption{The complex $R(X)$ given by
subdividing the complex $X$ in Figure \ref{TorusS0} using the
subdivision rule in \ref{TorusSubdivision}. The cell structure is the
same as that given by the flat edges in Figure
\ref{FlatS1}.}\label{TorusS1}\end{figure}

As described earlier in Section \ref{MethodSection}, to get a
subdivision rule we need to obtain a sequence of homeomorphic complexes
where each is a subdivision of the next. The boundary of each stage $S(i,j)$
is a sphere, so all the complexes are homeomorphic, but it is difficult to make one stage a subdivision of the next. Notice that
all edges are eventually covered up by fundamental domains, and
there is no obvious homeomorphism making one $S(i,j)$ a subdivision of
the next.

On the other hand, notice that in each $S(k,0)$, the set of flat edges
is the intersection of a family of walls (extensions of faces) with the
spherical boundary. If we let $R^n(X)$ be the cell structure on the
sphere given by just the flat edges of $S(n+1,0)$, there is a natural
way to make $R^{n+1}(X)$ a subdivision of $R^n(X)$, because the walls
associated to the flat edges of $S(n+1,0)$ extend to a subset of the
walls associated to flat edges of $S(n+2,0)$. This gives a sequence of
homeomorphisms (well-defined up to cellular isotopy) identifying each
$R^{n+1}(X)$ as a subdivision of $R^n(X)$.

We use $R^n(X)=S(n+1,0)$ in the proceeding paragraph because $S(0,0)$
has no flat edges. The cell structure of $X=S(1,0)$ is
shown in Figure \ref{TorusS0}. Each cell subdivides in one of three
ways, as shown in Figure \ref{TorusSubdivision}. These three ways of subdividing correspond to whether the tile is in the center of a facet on an edge of the cube, or on the corner of a cube. The subdivision complex is shown in \ref{TorusS0}, and its first subdivision is shown in Figure \ref{TorusS1}.

This example is important, because the neighborhood of each
codimension-3 cell of a right-angled object is locally modeled on the
3-torus, and so the process described above is essentially all that ever
happens.

\subsection{A manifold with product geometry}\label{ProductSection}

\begin{figure}
\scalebox{.5}{\includegraphics{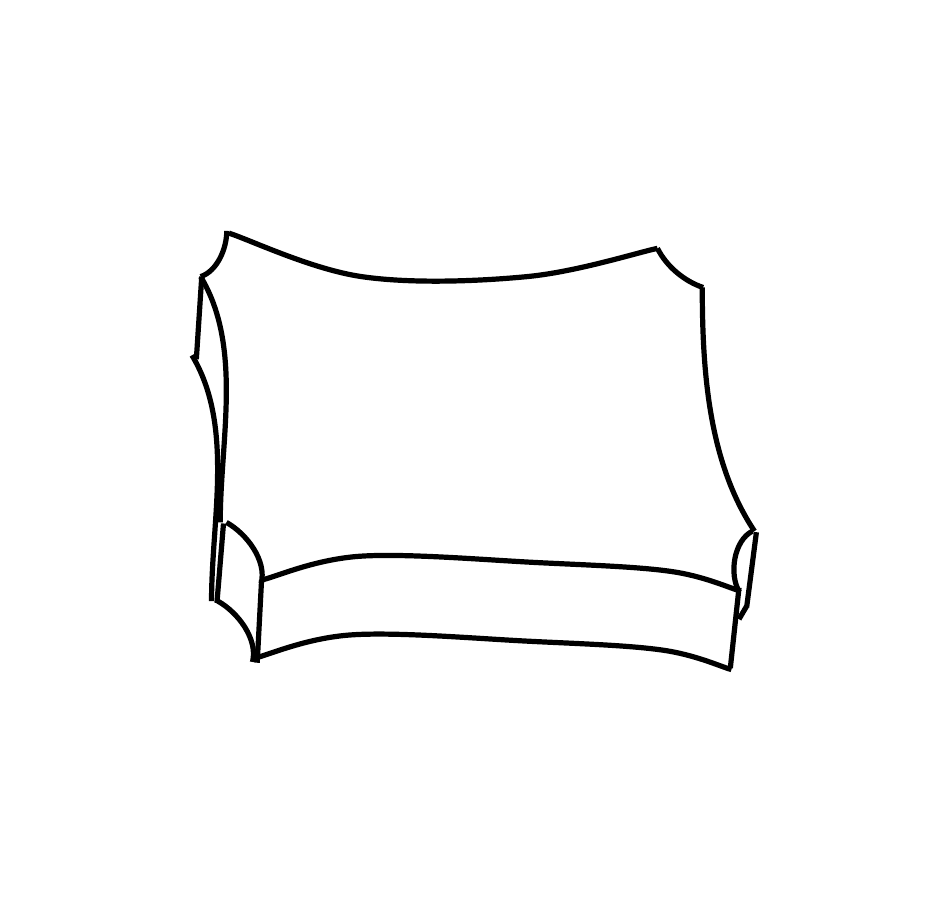}} \caption{The stage $S(0,0)$ for the product
manifold.}\label{ProductS00}\end{figure}

\begin{figure}
\scalebox{.5}{\includegraphics{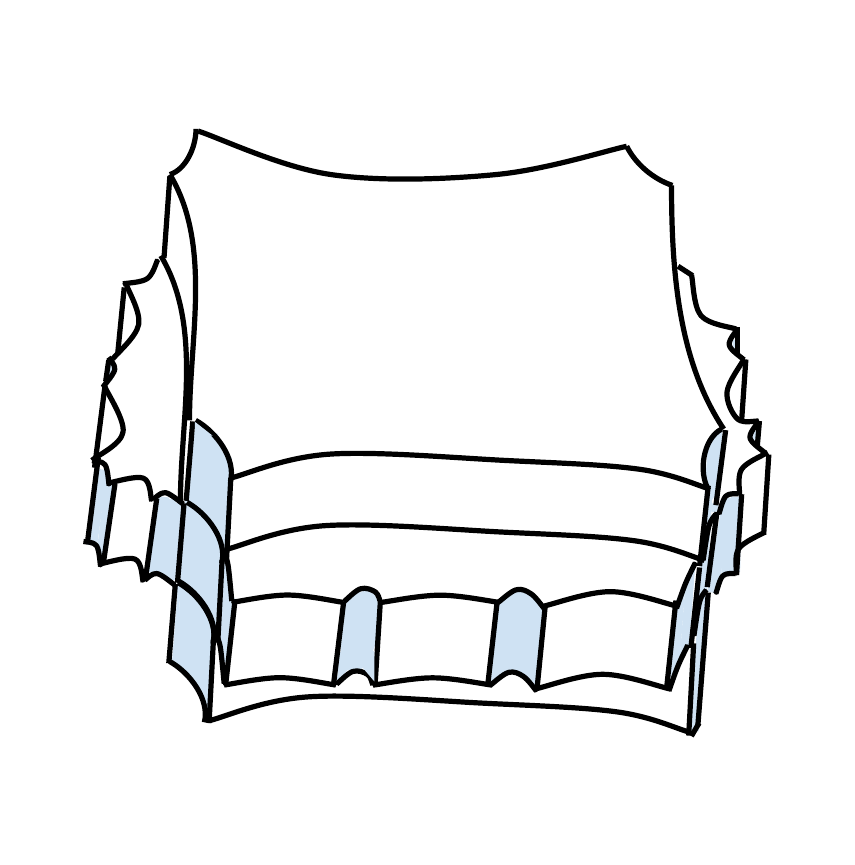}} \caption{The stage $S(0,1)$ for the product
manifold.}\label{ProductS11}\end{figure}

\begin{figure}
\scalebox{.5}{\includegraphics{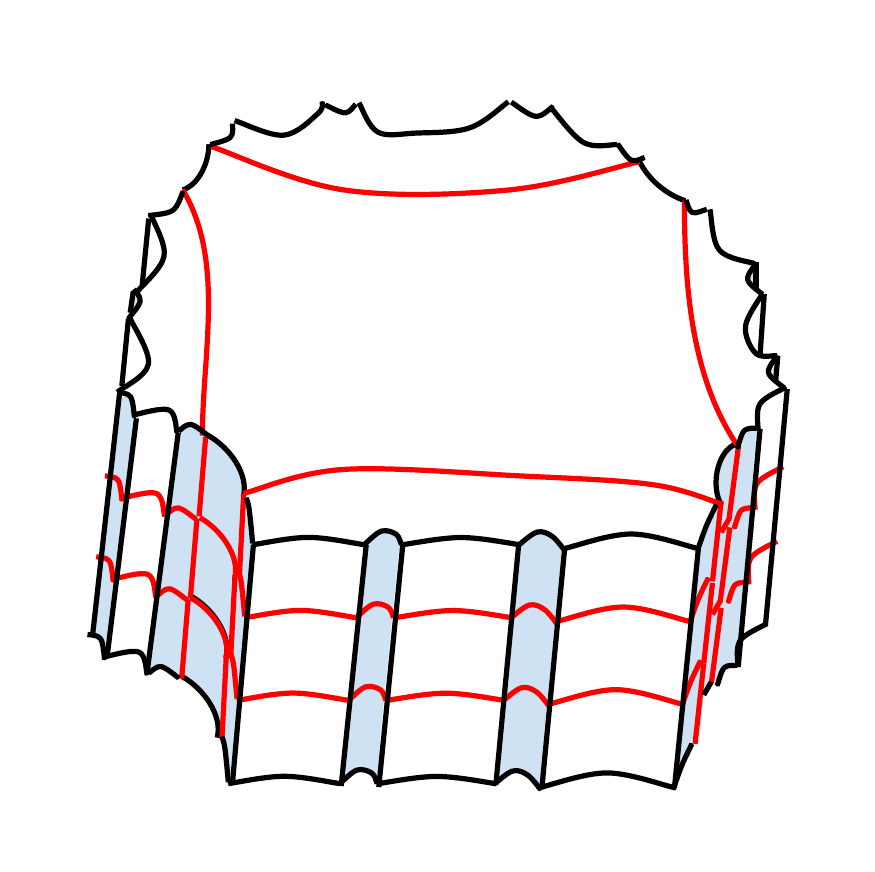}} \caption{The stage $S(0,2)$ for the product
manifold, with all flat edges marked.}\label{FlatProductS0}\end{figure}

Our second example introduces the concept of `ideal' faces, which are
faces that never are glued to anything and never subdivide; they arise
as the boundaries of manifolds. One example is the 3-manifold $M$ which
is the product of a circle with a punctured torus. We enlarge the
puncture by removing an open disk instead of a point, giving $M$ a
boundary which is itself a torus. A fundamental domain for $M$ is shown
in Figure \ref{ProductS00}, where the faces corresponding to the
boundary are slightly shaded in.

Figure \ref{ProductS00} is also the first stage $S(0,0)$ in our sequence
of spheres. After gluing on a domain to each non-ideal face, we obtain
$S(0,1)$, shown in Figure \ref{ProductS11}. As in the cube example,
there are now concave edges, and all faces with concave edges meet up in
pairs. We create $S(0,2)$, shown in Figure \ref{FlatProductS0}, by gluing
on a domain to each concave pair. There are now no concave edges at all,
so we start over, setting $S(0,2)=S(1,0)$.

At this point the figures become too complicated to depict accurately.
But just as in the 3-torus, we create $S(1,1)$ by gluing a
domain onto each non-ideal face, with every flat edge causing the new
faces around it to collapse. There will again be concave edges contained
in concave pairs. We then create $S(1,2)=S(2,0)$ by gluing domains onto
every pair, and continuing onward.

\begin{figure}
\scalebox{.45}{\includegraphics{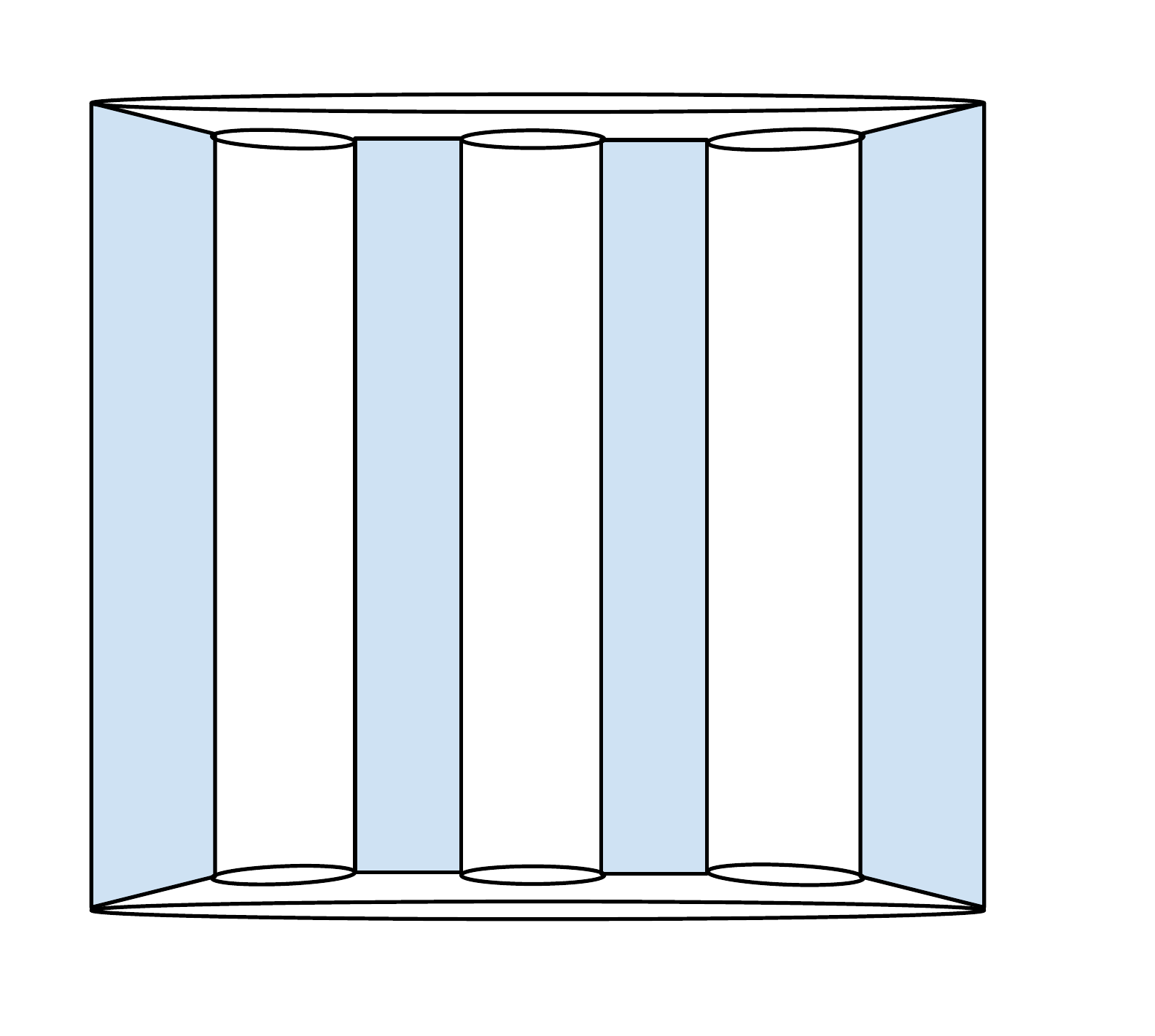}} \caption{The complex $X$ that the
subdivision rule of Figure \ref{ProductSubdivision} acts on. The cell
structure is the same as that given by the flat edges in Figure
\ref{FlatProductS0}. The tile types of the faces are determined by the
number of edges.}\label{ProductS0}\end{figure}

\begin{figure}
\scalebox{.5}{\includegraphics{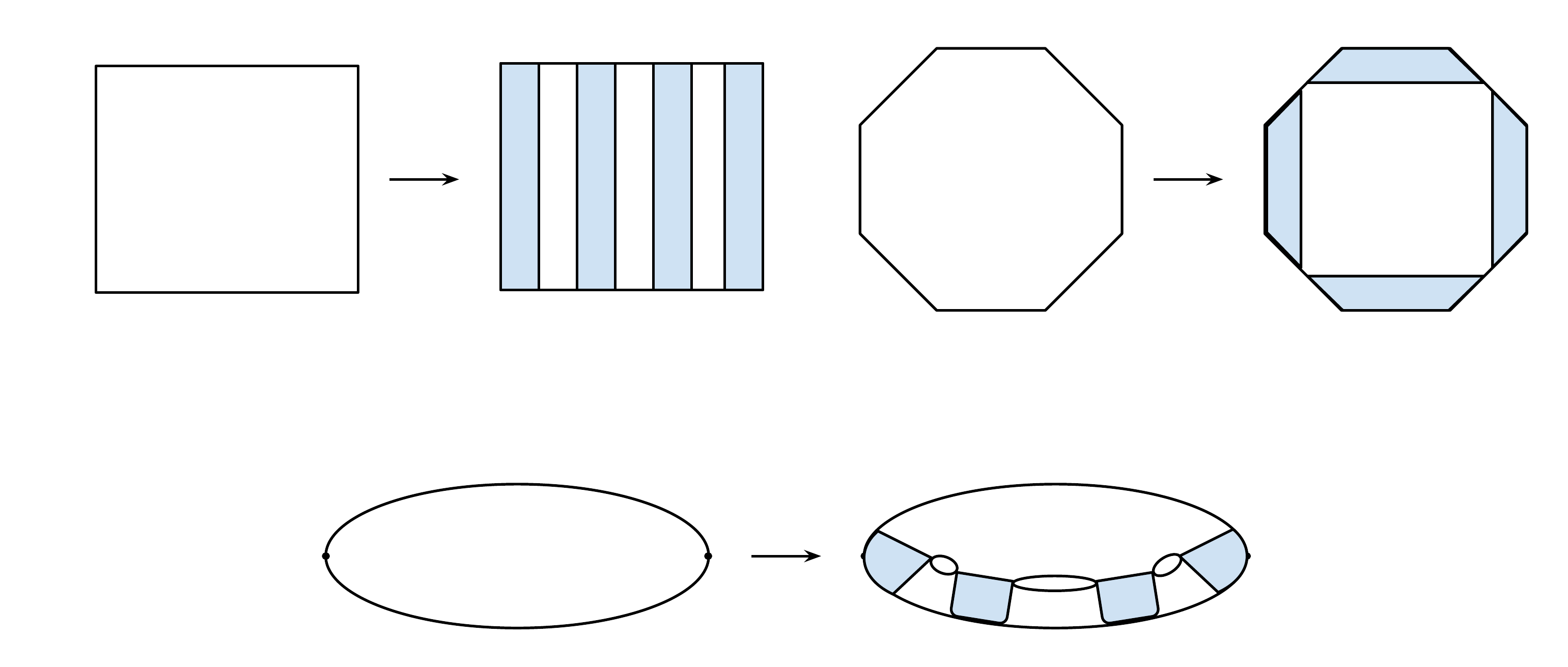}} \caption{The three tile types for the
subdivision rule of $F_2 \times \mathbb{Z}$.}\label{ProductSubdivision}\end{figure}

Taking the cell structure on the sphere determined by the flat edges of $S(n+1,0)$ as our complexes $R^n(X)$, we obtain
the subdivision rule depicted in Figure \ref{ProductSubdivision}, with
initial complex $X$ shown in Figure \ref{ProductS0}. The bigons here can
be seen in Figure \ref{FlatProductS0}, and are a characteristic feature
of the subdivision rules in this paper for manifolds with boundary
(including all non-abelian right-angled Artin groups).

\section{Gluing and collapse}\label{GluingCollapseSection}

In this section, we describe the two main operations we use to construct our sequence of tilings: gluing and collapse. These techniques will eventually be used to prove Theorem \ref{BigTheorem}. While the lemmas in this section hold in some generality, we will set out some standing assumptions that are rather strict.

We assume that the following conditions hold for the complex $X$ and polytopes $P_1,...,P_n$ used in each lemma (all terms will be defined before being used):

\begin{enumerate}
\item the complex $X$ is the boundary $\partial Y$ of a space $Y$, where $Y$ is the union of right-angled polytopes and is homeomorphic to a ball,
\item if they exist, the concave regions of $X$ are all concave $k$-stars for a fixed $k$ depending on $X$,
\item the intersection of any two concave stars contains the intersection of their respective center cells, and that intersection has codimension 1 in each cell, and
\item if there are concave regions their flat ridges are found only in their intersections with other concave regions, and
\item the flat ridges of the concave regions of $X$ form part of a coherent set.
\end{enumerate}

The terms \emph{concave region},\emph{concave star}, and \emph{center cell} will all be defined in Section \ref{ConcaveSection}, while the term \emph{coherent} will be defined in Section \ref{CollapseSection}.

We wish to generalize our two examples to other right-angled objects.
Recall from Section \ref{MethodSection} that there are two essential
requirements that a sequence of spaces $R^n(X)$ needed to satisfy for us
to extract a subdivision rule. For each $n$, we required:

\begin{enumerate} \item $R^n(X)$ is homeomorphic to $X=R^0(X)$, and
\item $R^{n+1}(X)$ is a subdivision of $R^n(X)$ after identifying both
with $X$. \end{enumerate}

Given such a sequence of spaces, we can construct a subdivision rule $S$
that recreates the sequence $R^n(X)$; this subdivision rule will be
finite if there are only finitely many ways that each tile subdivides.

These requirements were satisfied in the cube example. Let's look at
what happens in the general case. In creating the space $B(i,j+1)$ from
$B(i,j)$, we use two operations: \begin{enumerate} \item attaching
fundamental domains to $B(i,j)$ directly (i.e. \textbf{gluing}), and \item
attaching these new domains to each other (i.e. \textbf{collapsing}).
\end{enumerate}

These two processes can be recast as combinatorial operations on the
boundary $S(i,j)$. To obtain a subdivision rule, these operations must
not change the topology of $S(i,j)$ and must allow for a natural way of
embedding each $S(n,0)$ into the next. We describe these two operations
in the following two sections.

\subsection{Gluing and concave regions}\label{ConcaveSection}

When we glue a fundamental domain $A$ onto $B(i,j)$, we take a closed
region $\Omega_A\subseteq \partial A$ which is the union of several facets and attach it to
$B(i,j)$ via some cellular embedding $f:\Omega_A\rightarrow S(i,j)=\partial
B(i,j)$. The size of the attaching region
$f(\Omega_A)\simeq \Omega_A$ is determined by the number of concave ridges. In particular, if two facets $E,F$ of
$S(i,j)$ share a concave ridge $e$, then any fundamental domain $A$ that
is attached to $E$ must also be attached to $F$; otherwise, $e$ would be
contained in at least 5 fundamental domains in the universal cover,
which is not possible.

Thus, we form an equivalence relation on the facets of $S(i,j)$ by
letting $E\sim F$ if $E$ and $F$ share a concave ridge, and then
extending this to the smallest transitive relation satisfying this
condition. The equivalence classes of facets of $B(i,j)$ under this
relation are called \textbf{concave regions}. We call the portion
$\Omega_A$ of a domain $A$ that is attached to a concave region the
\textbf{gluing region}.

The combinatorial effect of gluing a domain $A$ onto a concave region
$f(\Omega_A)$ is to delete the interior of $f(\Omega_A)$ and replace it
with a copy of $A\setminus \Omega_A$.

If we want the topology to stay the same after attaching $A$, then we
require that the gluing region $\Omega_A$ and its closed complement
$\overline{\partial A\setminus \Omega_A}$ are balls.

\begin{lemma}\label{GluingLemma} Let $X$ be a cell complex. If $A$ is a
right-angled polytope with a gluing region $\Omega_A$ isomorphic to a
subcomplex $f(\Omega_A)\subseteq \partial X$ which is a topological
ball, then gluing $A$ to $X$ along $\Omega_A$ does not change the
topology of the boundary $\partial X$.\end{lemma}

\begin{proof} Since
$A$ is an abstract polytope, its boundary is a CW complex that has the
topology of a sphere. Thus, since $\Omega_A$ is a closed ball, the
complement $\overline{\partial A\setminus \Omega_A}$ is also a closed
ball, so the gluing operation does not change the topology.
\end{proof}

In the proof of Theorem \ref{BigTheorem}, we will use Lemma
\ref{GluingLemma} to show that the topology of the spheres $S(i,j)$ is
the same for all $i$ and all $j$.

In all cases, we will show that the concave regions are actually
\textbf{concave stars}. A concave $k$-star is a concave region
consisting of all facets on a polytope that contain a fixed cell $C$ of
codimension $k$ called the \textbf{center cell} of the concave star. In
right-angled objects, Lemma \ref{PolytopeLemma} shows that there are exactly $k$ facets
containing any fixed cell of codimension $k$ in the polytope. In the absence of
1- or 2-circuits, or prismatic 3-circuits, concave $k$-stars are topologically
balls, because they can be constructed by repeatedly attaching balls along smaller balls.

As a degenerate case, if there are no concave ridges in the complex $X$, we
say that each individual facet is a \textbf{concave 1-star}, and is its
own center cell. Thus, in Figures \ref{StageS11} and
\ref{Stage21Prime}-\ref{Stage21}, we say that a domain has been glued
onto each concave 1-star.

\subsection{Collapse and coherence}\label{CollapseSection}

The second operation in constructing the universal cover is \textbf{collapse}.
When we attach a fundamental domain to $S(i,j)$ as part of creating
$S(i,j+1)$, it may happen that some ridges of $S(i,j)$ were flat. As in
Figures \ref{Stage21Prime}-\ref{StageS22}, this causes the new facets
that intersect the old flat ridge to collapse. When we say that facets
`collapse', we are describing a combinatorial operation in a tiling of the
sphere where we delete the `collapsing' facets and the collapsing ridge
and identify their remaining boundaries. We can identify their boundaries because
the collapsing facets are combinatorially the same facet. This operation makes all
the ridges that are identified flat. See Figure \ref{GoodCollapse}.

 \begin{figure}
\scalebox{.5}{\includegraphics{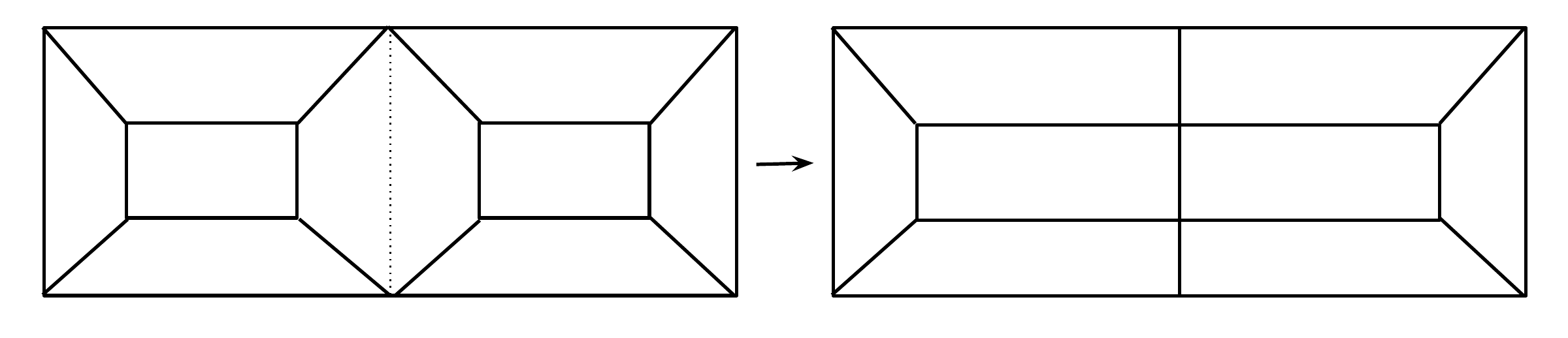}} \caption{A collapsing ridge. The dotted line
on the left is flat; the ridge and the facets containing it are deleted,
and their boundary identified to get the picture on the
right.}\label{GoodCollapse}\end{figure}

This operation may change the topology if all of the facets containing an
interior cell of codimension 2 or greater collapse (see Figure
\ref{BadCollapse}).

 \begin{figure}
\scalebox{.45}{\includegraphics{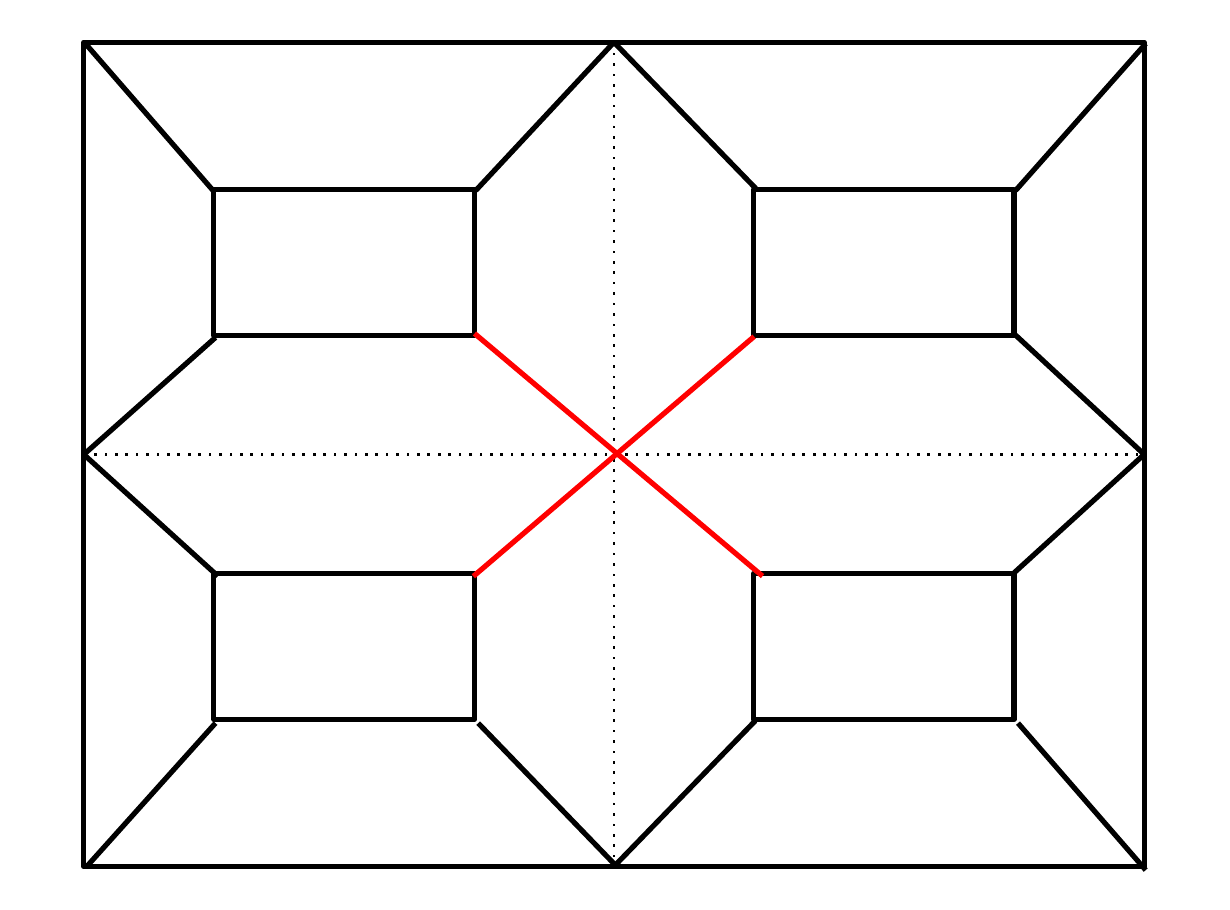}}\caption{The four lines in the center have
flat facets on either side. Deleting those flat facets and identifying
their boundaries leaves the four lines identified and gives a space that
is not a manifold.}\label{BadCollapse} \end{figure}

There is a simple solution. In Figure \ref{BadCollapse}, note that we can
delete all 4 ridges in the center and proceed to identify the remaining
boundary ridges as before. In the universal cover, this corresponds to
Figure \ref{CollapsingEdge}.

 \begin{figure}
\scalebox{.45}{\includegraphics{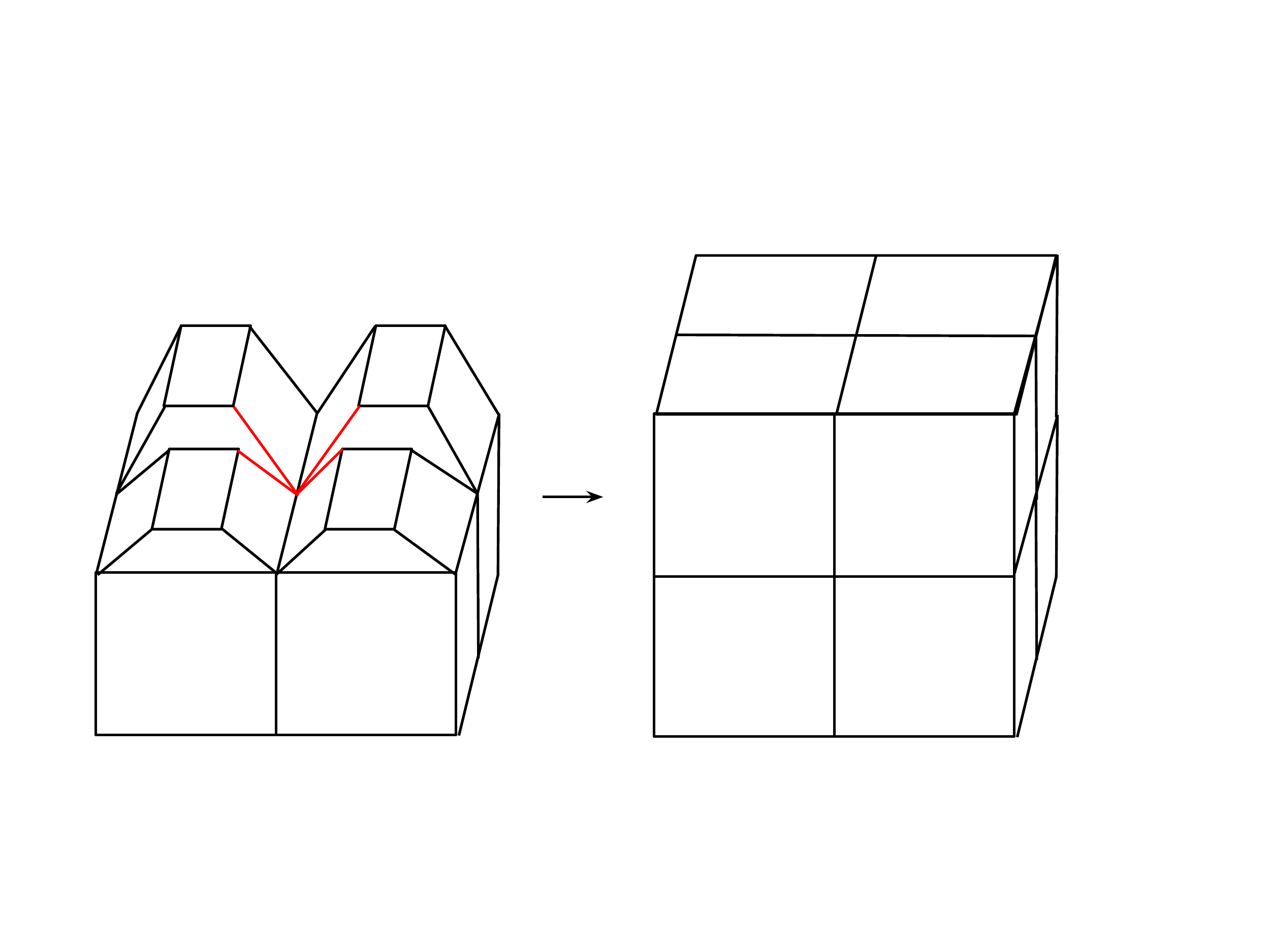}}\caption{All four lines in the center have
flat facets on either side. The correct idea is that when those flat
facets are deleted, the lines in the middle are deleted as
well.}\label{CollapsingEdge}\end{figure}

Notice how the four ridges disappear; this is because we've identified a
complete set of four fundamental domains around the ridge that they all
correspond to in the universal cover, and so the ridge is covered up.

In general, when facets collapse, we delete the collapsing facets and all
cells where every facet containing them is collapsing (i.e. the interior
of the union of the closed collapsing facets), then identify the
remaining boundary ridges. As in Figures \ref{Stage22Prime} and
\ref{StageS22} of Section \ref{CubeSection}, each facet that collapses
may have more than one flat ridge.

We identify the remaining boundary ridges of each collapsed facet $A$
with the remaining boundary ridges of the facet $A'$ it collapsed onto.
This makes the identified ridges flat or ideal flat. Notice that we
can think of this collection of new flat ridges as replacing the old
flat ridges. Thus, this gives us a natural map from the flat ridges of
$S(i,j)$ to the flat ridges of $S(i,j+1)$. This natural map is just
extending the wall that defines the flat ridges.

However, our identification as described above fails if different cells are deleted from $A$ and $A'$. Recall that we delete all
flat facets and all cells where every facet containing them is flat. Thus,
a boundary cell of $A$ is deleted if every other facet containing it is
flat. The same is true for $A'$. Thus, to show that $A$ and $A'$ have
the same boundary cells deleted, it suffices to show that they have the
same set of collapsing neighbors (i.e. if $e$ is a ridge of $A$ and $e'$
is the corresponding ridge of $A'$, then the other facets $B$, $B'$
containing $e$ and $e'$ are either both flat or both non-flat).

\begin{defi} We say that a set $E$ of flat ridges is \textbf{coherent}
if given $e$ in $E$, the reflection of $e$ across any other ridge
intersecting $e$ is also flat. Here, the \textbf{reflection} of a ridge
$e$ in a facet across a ridge $f$ is the unique ridge $e'$ which
intersects $e$, is distinct from $e$, and is a ridge of the facet
containing $f$ which does not contain $e$. \end{defi}

\begin{defi}
The \textbf{flat structure} of a complex $X$ or of a set of ridges $E$ of $X$ is the cell structure given by the flat ridges alone. It can be obtained by identifying any two facets that share a non-flat ridge.
\end{defi}

We summarize and expand on the above discussion with the following:

\begin{lemma}\label{CoLemma} Let $E$ be a set of flat ridges of a
complex $X$ satisfying our standing assumptions. Then collapsing each pair of facets containing a ridge of $E$ gives a new complex $X'$ with the
same topology as $X$, and $X'$ contains a coherent set $E'$ of ridges
whose flat structure is cellularly isomorphic to a subdivision of the flat structure of $E$. \end{lemma}

\begin{proof}
By our standing assumptions, the set of collapsing ridges of each facet $A$ that collapses is a ball, and is homeomorphic to its complement in the boundary of $A$. Thus, collapse doesn't change the topology locally. However, we must show that it is globally well-defined, that is, that it does not have singularities such as those shown in Figure \ref{GlobalCollapse}. But the set $E$ is a subset of the intersections of the cubes with their gluing region, so each ridge that is disjoint from $E$ can only be contained in one facet containing a ridge of $E$, as otherwise we would have a prismatic 3-circuit. This ensures that collapse is globally well-defined.

 \begin{figure}
\scalebox{.45}{\includegraphics{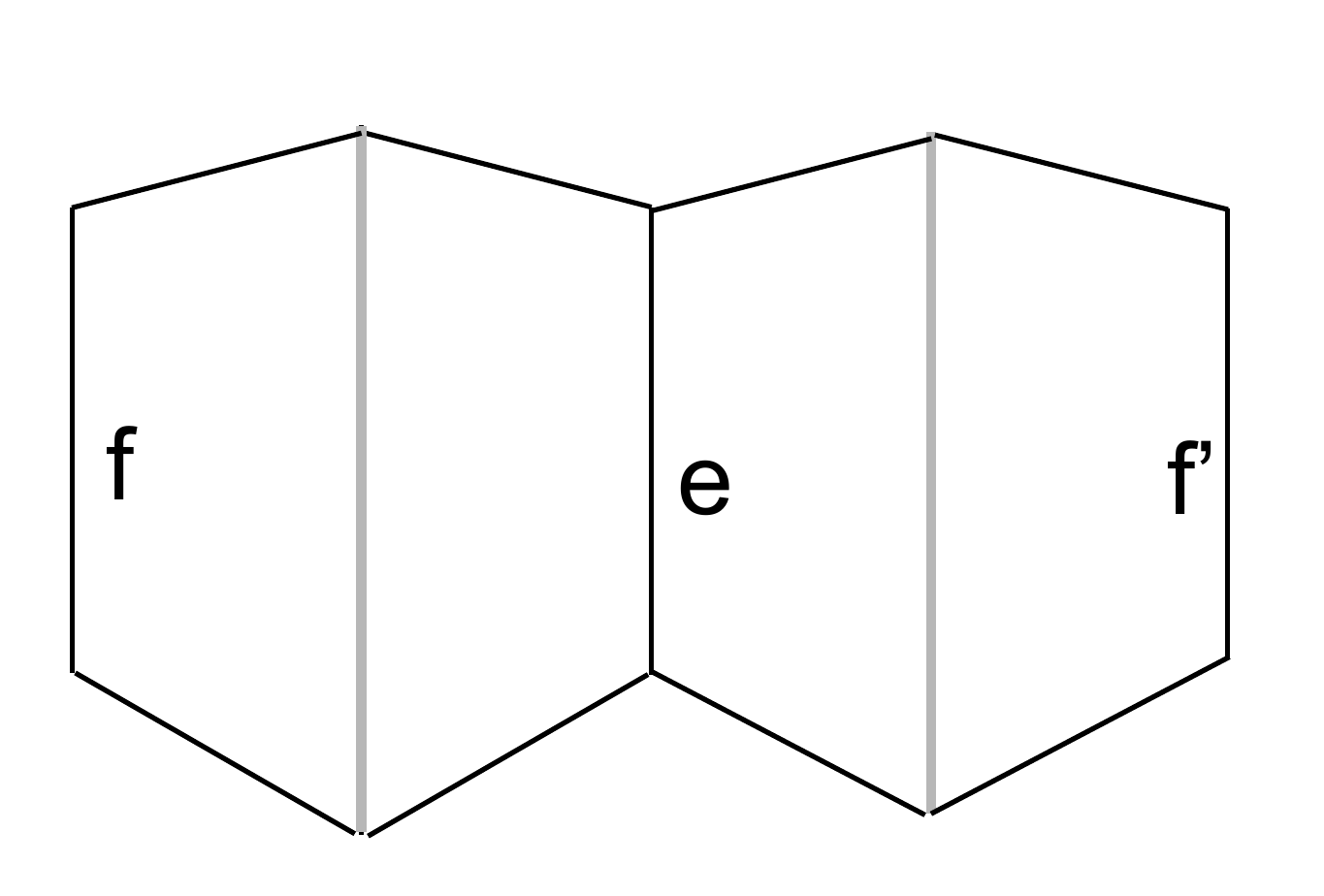}} \caption{If a ridge not intersecting any
collapsing ridge is contained in two different collapsing
facets, it allows cells that are far apart to be identified (such as the ridges $f$,$e$, and $f'$ here).
}\label{GlobalCollapse}\end{figure}

As described above, we delete the interior of the union of all
collapsing facets, and \emph{identify the boundary ridges of collapsing facets
to the boundary ridges of the facets they collapse to}. By our assumptions, the set of collapsing ridges of a facet is a concave star of the boundary of the facet, and so is a single ridge in the flat structure. After collapse, there is at least one and possibly more ridges in the flat structure.  By coherence, all flat structure that existed previously is preserved. Thus we can think of these newly flat (or flat ideal) ridges as a \emph{subdivision} of the
single collapsing ridge that was between the two collapsing facets.

Now we show that the new flat ridges are coherent. Let $A$ be a facet
with a ridge $e$ of any kind and a flat or flat ideal ridge $f$ that
came from collapse. Then either $e$ came from collapse or it did not.

If $e$ came from collapse, then $f$ and $e$ both replaced collapsing
ridges $f'$ and $e'$. But the set of collapsing ridges was coherent, so
the reflection of $f'$ across $e'$ was another collapsing ridge $g'$,
which gives a flat ridge $g$ that is the reflection of $f$ across $e$.

If $e$ did not collapse and is a ridge from a previous stage, then, as
in the previous case, let $f'$ be the ridge that collapsed to make $f$
flat. Then $e$ existed in the previous stage and the reflection $g'$ of
$f'$ across $e$ was flat in that stage by coherence, so it collapsed to
give a flat ridge $g$ which is the reflection of $e$ across $f$.

If $e$ did not collapse and is a new ridge, then let $A$ be the facet that contained $f$ and
collapsed to make $f$ flat. Then $e$ intersected two of the ridges of
$A$, one on either side of $e$. One was $f$, and we'll call the other
$g$. Then when $A$ collapsed, $f$ and $g$ both became flat. Since $g$
was the reflection across $e$ of $f$, this shows that they are still
coherent and completes our proof. \end{proof}

Coherence is an essential tool in the proof of Theorem \ref{BigTheorem},
because it allows us to think of the new flat ridges as a subdivision of
the old flat ridges. The following lemma will be useful in the proof:

\begin{lemma}\label{CoherenceInduction} Assume that a complex $X=\partial Y$, where $Y$ is the union of right-angled polytopes which is topologically
a ball. If there are no concave ridges, then the set of all flat ridges
of $X$ is a coherent set. \end{lemma} \begin{proof}

Define the \textbf{link} of a codimension-3 cell $K$ in $Y$ to be the
simplicial complex whose 0-skeleton is given by ridges containing $K$,
whose edges correspond to facets containing $K$, and whose facets
correspond to fundamental domains containing $K$. Alternatively, the link
of $K$ is the complex one obtains by taking a sufficiently small
2-sphere orthogonal to $K$ and centered at an interior point of $K$, and
looking at the cell structure it inherits from its intersection with
$Y$.

In a right-angled manifold, the link of codimension-3 cell is always an
octahedron, unless it is ideal, in which case the link is half or a fourth of an octahedron. Since $Y$ may have some unidentified facets, the
link will generally be a subset of the octahedron or half-octahedron or
quarter-octahedron. Now, the subset of the octahedron is itself a
right-angled object, and its vertices are convex, flat, concave, or
covered up exactly when the ridges they correspond to are. So if the
link has a vertex which is contained in exactly 3 triangles of the link,
that corresponds to a concave ridge in $X$.

 \begin{figure}
\scalebox{.25}{\includegraphics{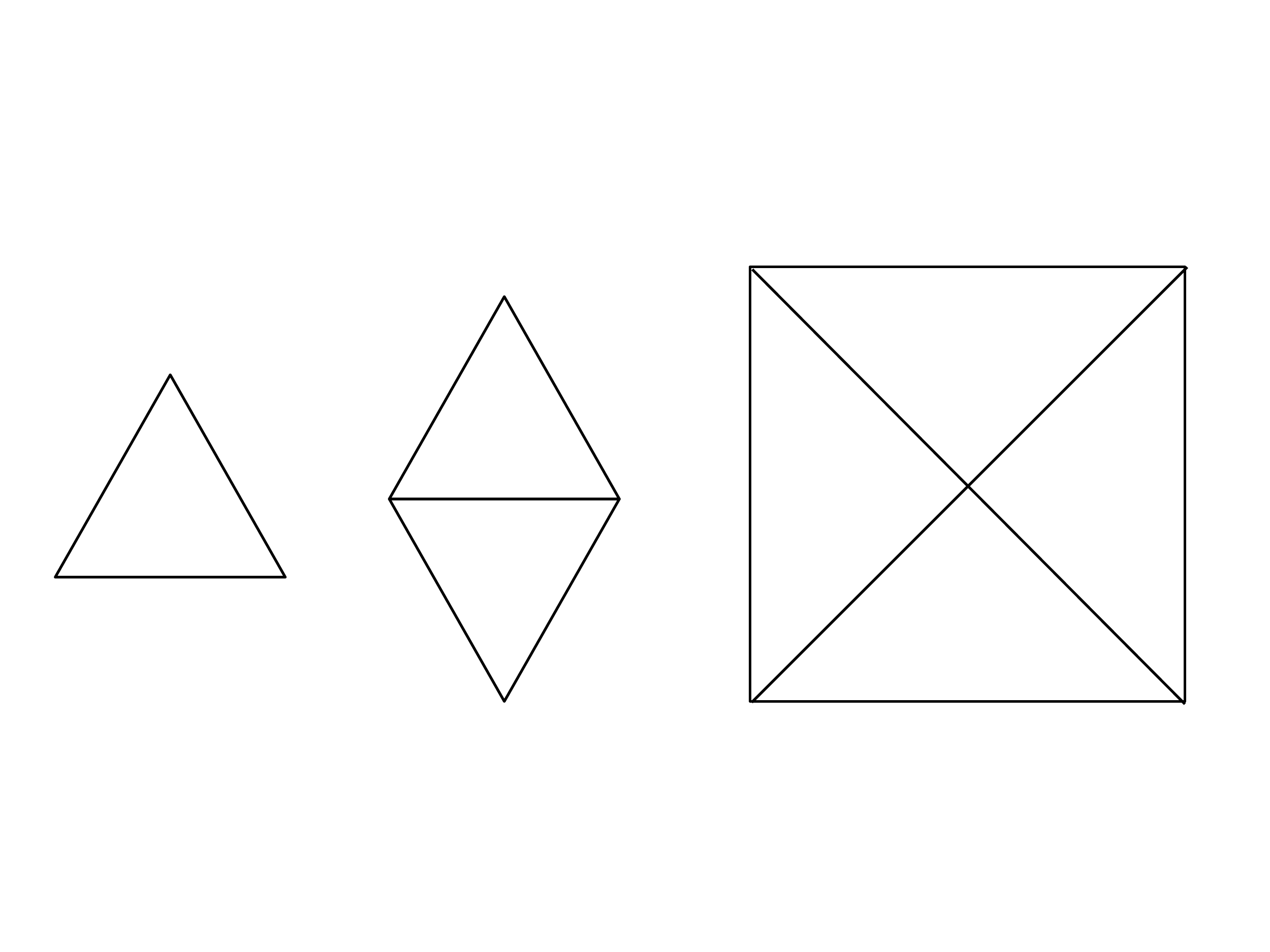}} \caption{These three subsets of the
octahedron are the only proper subsets of the octahedron without concave
vertices.}\label{ConvexSubsets}\end{figure}

Now, we are assuming that there are no concave ridges in $X$. Thus, the
link of every codimension-3 cell has no concave vertices. Figure
\ref{ConvexSubsets} shows all possible links without concave vertices
(the octahedron itself is another possibility, but if the link is a full
octahedron, the cell is not in the boundary of $Y$). Notice that in all
of these subsets, the flat vertices come in opposite pairs. This implies
that the corresponding flat ridges are coherent. Thus, in $X$, because
there are no concave ridges, the flat ridges are coherent.
 \end{proof}

\subsection{The interaction between gluing and collapsing}

The following technical lemmas describe how gluing regions and
collapsing facets interact in certain special situations that will arise
in the proof of Theorem \ref{BigTheorem}.

\begin{lemma}\label{FlatIntersection} Let $X$ be as before, and let $E_1,...,E_n$ be the concave regions of $X$. Assume that we glue on a polytope $P_i$ to each concave
star $E_i$ and collapse all formerly flat ridges to create a complex
$X'$. Then the flat ridges of the
concave regions of $X'$ are only found in the intersection with other
concave regions. \end{lemma}

\begin{proof} We first show that a ridge
$e$ in $X'$ can only be concave if it existed in $X$ as a convex ridge
in the intersection of two concave stars. We do this by eliminating all other
possibilities.

If $e$ was a flat ridge in $X$, then by hypothesis, both facets of $X$
containing it were part of concave stars, and both had polyhedra
attached to them; thus, the flat ridge collapses, and all new ridges
resulting from the collapse are flat.

If $e$ was convex in $X$ but only one facet containing it was part of a
concave region, then $e$ would have become flat after gluing on
polytopes, and not concave.

Thus, $e$ was a convex ridge in $X$ in the intersection of two concave
regions, and it had polytopes glued onto both sides.

Now, we prove that the flat ridges of the concave regions of $X'$ are
only found in the intersections between different concave
regions.

Let $A$ be a facet of $X'$ with a flat ridge $e$ and a concave ridge $f$.
Because $A$ has a concave ridge, $A$ must be a facet of a polytope $P$
glued onto a concave region of $X$. The flat ridge $e$ must have
come from either the boundary of a facet that collapsed or from another intersection of $A$ with the gluing region. But the second possibility would imply the existence of a prismatic 2- or 3-circuit in the polytope. Thus, the flat ridge must have come from the collapse of a neighboring facet $B$ in the polytope (see Figure \ref{FlatEdgeOfPair}).

 \begin{figure}
\scalebox{.4}{\includegraphics{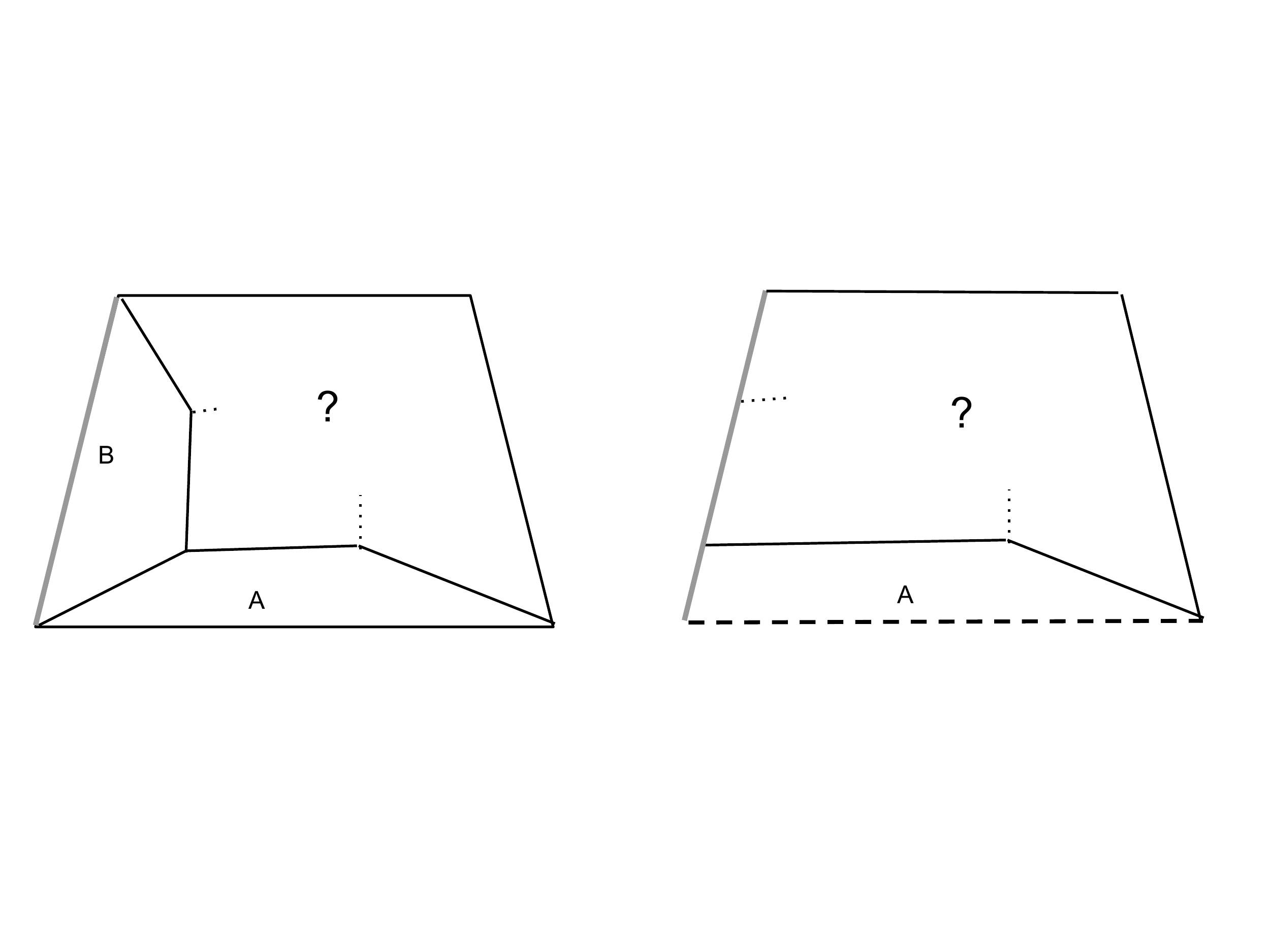}}\caption{The flat ridge of $A$ must have
come from a collapsing neighbor $B$.}\label{FlatEdgeOfPair}\end{figure}

Let $A'$ be the other facet of $X'$ containing the flat ridge $e$. Like
$A$, the facet $A'$ is also part of its own polytope glued onto a concave
region of $X$, and it must have had a neighbor $B'$ that collapsed with
$B$ to form the new flat ridge $e$. To prove the lemma, it suffices to
show that $A'$ has at least one concave ridge.

We first show that $e$ and $f$ intersect each other. Because $f$ is concave, there is a facet $C$ of the gluing region such that $f=A\cap C$. But the facet $B$
that collapsed intersected \emph{every} facet of the gluing region, by
hypothesis. Thus, $A,B,$ and $C$ form a 3-circuit.
Since it cannot be prismatic, the ridges $e$ and $f$ intersect each
other.

Finally, we show that $A'$ (the facet of $X'$ that shares the ridge $f$
with $A$) also has a concave ridge. Now, when $B$ and $B'$ collapsed,
their interiors and their common ridges were deleted and their remaining
boundaries were sewn together, fixing the portion of each boundary that
intersected the collapsing ridge. By coherence, the remaining boundaries
that are glued together are identical. Note that the ridge $g$ of $A$
that collapsed to become $f$ intersected the boundary of its polytope's
gluing region. By symmetry, this means that the ridge $g'$ of $B'$ that
became $f$ intersected the boundary of its polytope's gluing region.
This means that $A$ and $A'$, the facets across $g$ and $g'$ from $B$ and
$B'$, respectively, must also have intersected the gluing region (since
they contained $g$ or $g'$), meaning that each has a concave ridge in
$X'$. This concludes our proof. \end{proof}

The following corollary will be used in the proof of Lemma
\ref{StarLemma}, our final lemma before Theorem \ref{BigTheorem}.

\begin{cor}\label{EasyCor} If a facet $A$ in the complex $X'$ of Lemma
\ref{FlatIntersection} has a concave ridge, then $A$ is a facet of one of
the new polytopes $P_1,...,P_n$ glued onto $X$. \end{cor} \begin{proof}
It was shown in the proof of Lemma \ref{FlatIntersection} that a ridge
$e$ in $X'$ can only be concave if it existed in $X$ as a convex ridge
in the intersection of two concave stars. Each such convex ridge has a
new polytope glued onto either side of it to create $X'$, and thus both
facets in $X'$ that contain it come from new polytopes. \end{proof}

\begin{lemma}\label{StarLemma} Let $X$ be as before, with concave regions $E_1,...,E_n$. Then if $X'$ is the complex obtained by gluing on a
right-angled polytope to each $E_i$, these three conditions hold for
$X'$ and its new concave regions $E'_1,...,E'_m$:
 \begin{enumerate}
 \item each $E'_i$ is a concave $k$-star,
\item if any two facets $A$ and $B$ of distinct $k$-stars
intersect, the ridge $A\cap B$ contains the intersection of the center
cells of the two concave stars, and
\item the intersection of any two center
cells $K_1,K_2$ is a single cell $K_1 \cap K_2$ of codimension 1 in each center cell.
\end{enumerate}
\end{lemma}

\begin{proof} Let
$E_1,E_2$ be distinct concave stars in $X$ that intersect in a ridge
$e$. Let $P_1,P_2$ be the right-angled polytopes glued onto $E_1$ and
$E_2$, respectively. Let $A_1$ and $A_2$ be the facets in $X'$ that
contain $e$ (before collapse, if $e$ was flat in $X$), where $A_1$ comes from $P_1$ and $A_2$ comes from $P_2$. Since the $A_i$
intersect in the ridge $e$, by hypothesis each intersects the center
cell of the gluing region of its respective polytope, meaning it
intersects all $k$ facets of the gluing region.

There are two possibilities: either one of the ridges of the concave
region that $A_1$ is attached to was flat, or they were all convex. If
one was flat, then by coherence, all ridges of the concave region that
intersect $A_1$ are flat (see Figure \ref{ConcaveCoherence}). Since all the ridges are
flat, $A_1$ will collapse over multiple ridges; but $A_1$ can only be
identified with a single other facet, so it must have all of its flat
ridges in common with a single other facet $A_2$, and they must collapse
with each other (as in Figure \ref{Stage22Prime}).

 \begin{figure}
\scalebox{.8}{\includegraphics{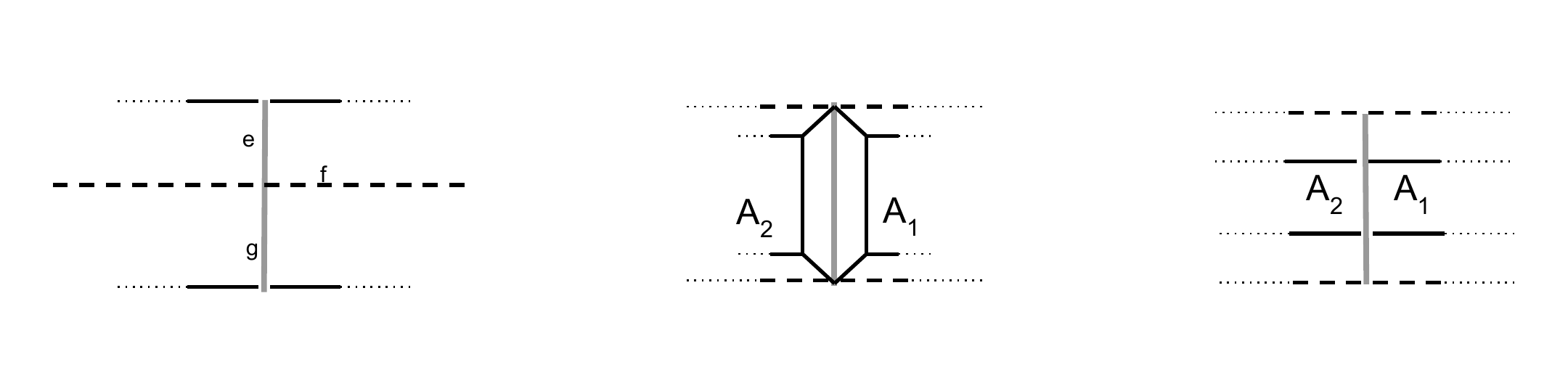}} \caption{If one ridge of a concave region
is flat (here, the ridge $e$), then its reflection across all concave
ridges is flat by coherence (here, $g$ is the reflection of $e$ across
$f$. This implies that any facet (here, the facet $A_1$) of the polytope
glued onto the concave region that has a flat ridge has no concave
ridges. Any such facet will collapse over all of its flat ridges with an
identical facet (here, $A_2)$).}\label{ConcaveCoherence}\end{figure}

Thus, we can assume that all $k$ ridges of $E_1$ that $A_1$ intersects
were convex in $X$ and are concave in $X'$. Then all these ridges have a
common cell, which is $A_1\cap K_1$, where $K_1$ is the center cell of
$E_1$. Now, let $A_2$ be one of the neighboring facets with a concave
ridge; by symmetry with $A_1$, it has $k$ concave ridges, all of which contain $A_2
\cap K_2$. Now, due to Lemma \ref{PolytopeLemma},
$A_1\cap K_1$ is a single cell of codimension 1 in $K_1$, and $A_2\cap
K_2$ is a single cell of codimension 1 in $K_2$. But by hypothesis,
$K_1\cap K_2$ is a single cell of codimension 1 in both $K_1$ and $K_2$,
and $K_1\cap K_2 \subseteq A_1 \cap A_2\subseteq A_1$. Thus, $K_1\cap
K_2 \subseteq A_1\cap K_1$, and we must have equality. By symmetry, we
also have that $K_1\cap K_2 = A_2 \cap K_2$, and, in general, $A_i\cap K_i=K_1\cap K_2$ for all facets $A_i$ intersecting $A_1$ in a concave ridge..

Thus, all facets in the concave region containing $A_1$ have exactly $k$
concave ridges (implying that there are at least $k+1$ facets in the
concave region), and there is a single cell $K'=K_1\cap K_2$ that all
concave ridges of all facets in the concave region contain (which implies
that any two concave facets share a concave ridge with each other). This
implies that the concave region is a concave $(k+1)$-star.

We need to show that condition 2 holds in $X'$. If $A$ and $B$ are facets
in $X'$ that have concave ridges but share only a convex ridge, then
they each came from a new polytope (by Corollary \ref{EasyCor}) and it
must be the same polytope (because they share a convex ridge). But then
because they have a concave ridge, each intersected the center cell $K$
of the gluing region. If $A\cap K$ did not intersect $B\cap K$, then
there would be a facet $C$ of the gluing region which does not intersect
the ridge $A\cap B$; then $A,B$ and $C$ would together form a prismatic
3-circuit. This is a contradiction. Thus, $A\cap K$ intersects $B\cap
K$, and thus, the center cells of their respective concave stars
intersect. Since $A\cap B\cap K$ is non-empty, and since $K$ is the
intersection of $k$ facets, $A\cap B \cap K$ is the non-empty
intersection of $k+2$ facets, and thus is a single cell of dimension one
lower than $A\cap K$ or $B \cap K$.

This completes the proof. \end{proof}

\section{A subdivision rule for right-angled manifolds}\label{MainSection}

The subdivision rules we obtain in this section will depend on the combinatorics of the original polytopes comprising a fundamental domain. In particular, they depend on a new complex derived from the polytopes called the \textbf{inflation}. For those familiar with such constructions, the inflation is almost identical to the handle decomposition (see \cite{Kosinksi}).

\begin{defi}
Let $P$ be a right-angled polytope with boundary $X=\partial P$. Let $I'(X)$ be a new complex such that:
\begin{itemize}
\item For every \emph{cell} $K$ of $X$, there is a \emph{facet} $I'(K)$ in $I'(X)$,
\item If $K$ has codimension $m$ in $X$, we have $I'(K)\cong K \times \Delta^m$,
\item If $K\subseteq K'$ are two cells of $X$ whose dimensions differ by 1, we attach $I'(K)$ to $I'(K')$ by a ridge $K\times \Delta^{m-1} \subseteq K\times \partial \Delta^m \subseteq \partial I'(K)$ and the ridge $K\times \Delta^{m-1} \subseteq \partial K'\times \Delta^{m-1} \subseteq \partial'(K')$, where the map is $(id, inclusion)$ and the choice of the sub simplex $\Delta^{m-1}\subseteq\partial \Delta^m$ is immaterial due to symmetry.
\end{itemize}

Finally, we let $I(X)$ be the quotient of $I'(X)$ obtained by quotienting the simplex factor of $I(K)$ to a point for all ideal cells $K$ (i.e. if $I'(K)=K\times \Delta^m$, we replace it with $I(K)=K\times \{pt.\}\cong K$). Note that this will cause some facets of the adjoining non-ideal cells to collapse. We call $I(X)$ the \textbf{inflation} of $X$.

\end{defi}

We think of the inflation process as expanding each non-ideal cell into a facet while the ideal cells remain fixed.

\textbf{Examples:} In the 3-torus example (Section \ref{CubeSection}), our fundamental domain was a cube. Its inflation is the complex of Figure \ref{TorusS0}.

In the product case, our polytope was an octahedral prism with 4 ideal square facets and 4 non-ideal facets, as well as 2 non-ideal octagons. Its inflation was given in Figure \ref{ProductS0}.

We need only one more definition:

\begin{defi}
Let $K$ be a cell in the complex $X$. Let $S(K)$ be the union of all closed cells intersecting $K$. The \textbf{inflated star} $IS(X)$ is defined as $I(S(K))$, the inflation of $S(K)$.
\end{defi}

\begin{thm}\label{BigTheorem} Let $M$ be a right-angled manifold with a fundamental domain consisting of polytopes $P_1,...,P_n$. Then $M$ has a
finite subdivision rule. The tile types are in 1-1
correspondence with the facets of the inflations $I(\partial P_1)$,...,$I(\partial P_n)$. Each tile corresponding to a facet $I(K)\subseteq I(\partial P_i)$ is subdivided into a complex isomorphic to $\overline{I(\partial P_i)\setminus IS(K)}$, i.e. the complement of the inflated star.\end{thm}

\begin{proof} We prove this by induction using the various lemmas. We
claim by induction that for each $S(i,j)$:
\begin{enumerate}
\item the
topology of $S(i,j)$ is that of a sphere,
\item the concave regions of
$S(i,j)$ are all concave $(j+1)$-stars,
\item the intersection of any
two concave stars contains the intersection of their respective center
cells (which has codimension 1 in each cell),
\item the flat ridges of
facets in concave regions of $S(i,j)$ are only found in the intersection
with other concave regions,
\item the flat ridges of $S(i,j)$ found in
the intersection of concave regions are part of a coherent set for
$j\neq 0$, and
\item the flat ridges of $S(i,0)$ are coherent.
\item a
subdivision of the flat structure of $S(i,j)$ embeds into the flat structure
of $S(i,j+1)$. \end{enumerate}

These claims are all true trivially for $S(0,0)$, which has no concave
or flat ridges.

Assume the claims are true for $S(i,j)$. Then we then create $S(i,j+1)$
by attaching polytopes onto each concave region. By Lemma
\ref{GluingLemma}, this does not change the topology of $S(i,j+1)$. We
then collapse all facets that adjoin formerly flat ridges. By claim 5, the collapsing ridges are coherent. Thus, by Lemma
\ref{CoLemma}, this does not change the topology of $S(i,j+1)$. This
proves claim 1.

Claims 2 and 3 follow from Lemma \ref{StarLemma}.

Claim 4 follows from Lemma \ref{FlatIntersection}.

Claim 5 follows from Lemma \ref{CoLemma}, since all ridges in the
intersection of concave regions come from new polytopes on both sides (by Corollary
\ref{EasyCor}) and thus can only be flat from collapse.

Claim 6 follows from Lemma \ref{CoherenceInduction}.

Claim 7 follows from Lemma \ref{CoLemma}.

This concludes the proof of the claims. The most important of the claims
are claims 1 and 7; the rest are merely used in the induction. By virtue
of these claims, we can identify each stage $S(i,0)$ as a single sphere
with a sequence of cell structures that are subdivisions of each other
coming from the flat ridges $S(i,0)$. The first stage that has flat
ridges is $i=1$, so we let $R^n(X)$ be the sphere with the cell
structure coming from the flat ridges of $S(n+1,0)$.

We now describe the tile types of the subdivision rule. We claim that
there is one tile type for each facet of the inflations $I(\partial P_1),....I(\partial P_n)$. Heuristically, this is because the non-ideal cells of dimension
$k$ are in one to one correspondence with concave $(d-k)$-stars (each cell corresponding to its star), which are the essential building blocks of the tilings.
We now prove this claim. We first show that the flat ridges of
$S(i+1,0)$ partition the sphere into tiles that are in one-to-one
correspondence with the non-ideal \emph{cells} and ideal \emph{facets} of
$S(i,0)$.

Every non-ideal point in $S(i+1,0)$ is contained in the boundary of some
polytope $P$ of $B(i+1,0)\setminus B(i,0)$ (i.e. in the outermost layer), while the ideal points are
exactly the union of all ideal facets of all polytopes in $B(i+1,0)$  (as in Figure \ref{FlatProductS0}).

Let $P_k$ indicate the intersection of a polytope $P$ with the sphere $S(k,0)$. The collection $\{P_{i+1}\}$ for fixed $i$ and varying $P$ tiles $S(i+1,0)$, in the sense that the interiors of the $P_{i+1}$ are
disjoint, and the union of the $P_i$ covers $S(i+1,0)$. This tiling is in fact the tiling of $S(i+1,0)$ given by flat ridges. This follows since the intersection in $S(i+1,0)$ of any two $P_i$ is flat, because the ridges touch at
least 2 polytopes and there are no concave ridges in $S(i+1,0)$.
Conversely, each flat ridge of $S(i+1,0)$ is contained in an
intersection of two polytopes.

Thus, the flat ridges of $S(i+1,0)$ divide the points of $S(i+1,0)$ into
distinct tiles, one for each component of each $P_{i+1}$ (there may be more
than one component if $P_i$ consists entirely of ideal facets, as occurs
for some polytopes in Figure \ref{FlatProductS0}). We can consider two classes of the $\{P_{i+1}\}$: those where $P$ is in $B(i,0)$ and those where $P$ is in $B(i+1,0)\setminus B(i,0)$.

If a polytope in $B(i,0)$ intersects $S(i+1,0)$, it must intersect in
ideal facets only (because all non-ideal facets are covered up by the time
we reach $S(i+1,0)$. Every ideal facet of the polytopes composing $B(i,0)$
intersects $S(i+1,0)$ with all of its ridges being flat, so each ideal
facet of each polytope in $B(i,0)$ is its own tile in the cell structure
on the sphere given by the flat ridges of $S(i+1,0)$.

A \textbf{convex cell} of $S(i,0)$ is a cell which is not contained in \emph{any} flat ridge; or, equivalently, it is a cell in $B(i,0)$ which is contained in only one fundamental domain. By the construction of the $B(i,j)$, the polytopes in
$B(i+1,0)\setminus B(i,0)$ are in one-to-one correspondence with convex
cells of $S(i,0)$, since each convex cell of dimension $k$ becomes the
center cell of a concave $(d-k)$-star that is covered up in creating
stage $S(i,d-k)$. Thus, in the cell structure
on the sphere given by the flat ridges of $S(i+1,0)$, there is exactly one tile for each convex
cell of $S(i,0)$. Two tiles are adjacent exactly when their
corresponding polytopes share a flat ridge in $S(i+1,0)$; but every flat
ridge in $S(i+1,0)$ comes from either:
\begin{enumerate}
\item  facets in two $j$-stars collapsing for some $j$, or
\item a polytope being glued onto only one side of a convex ridge.
\end{enumerate}

Let's consider these two cases. Case 1 occurs when domains are glued simultaneously next to a flat ridge in some stage $S(i,j)$. Case 2 occurs when domains are glued sequentially, with a domain glued on in stage $S(i,j+1)$ next to a convex ridge of a domain glued on in stage $S(i,j)$.

Case 1 occurs when the center cell of one $j$-star (which lives as a convex cell of $S(i,0)$) intersects the center cell of another $j$-star, and their intersection intersects a flat ridge. For example, if two facets of $S(i,0)$ share a flat ridge, then the polytopes glued onto them will share a flat ridge (as in Figures \ref{Stage21Prime} and \ref{Stage21}). Similarly, if two convex ridges of $S(i,0)$ intersect each other and a flat ridge, the polytopes glued onto the concave pairs they belong to will share a flat ridge (as in Figures \ref{Stage22Prime} and \ref{StageS22}, where the convex ridges are the edges on the corners of the cube).

Case 2 occurs when a polytope is glued onto a $j$-star which shares a convex ridge with a facet $B$ without concave ridges. But concave $j$-stars only share convex ridges with facets of polytopes that were glued onto $j-1$-stars. Also, the center cell of any such $j$-star will be contained in the center cell of each of the $j-1$-stars that the polytopes comprising it were glued onto (as described in the proof of Lemma \ref{StarLemma}). Thus, two tiles in the cell structure given by the flat ridges of $S(i+1,0)$ can only share this kind of flat ridge if they correspond to convex cells $K_1,K_2$ where the dimension of $K_1$ is one less than that of $K_2$ and $K_1\subseteq K_2$.

Thus, the tiles in the flat structure corresponding to two
cells of $S(i,0)$ are adjacent when their two cells are the same
dimension and intersect a flat ridge together or when their dimensions
differ by 1 and one is included in the other.

Thus, since we define the $n$-th stage of subdivision to be the cell
structure on the $(d-1)$-sphere given by flat ridges of $S(n+1,0)$, we
can describe it explicitly by taking the cell structure of $S(n,0)$ and
`inflating' each convex cell into its own facet, with ridges
corresponding to containment of cells or to two cells intersecting a
flat ridge. Thus, the flat cell structure is just the inflation of $S(n,0)$.

Only cells lying on the boundary $S(n,0)$ get inflated.

In the examples of Sections \ref{CubeSection} and \ref{ProductSection},
this same inflating process occurred. In the torus case, each
ridge (in this case, each edge) was inflated by taking the product with a 1-simplex and each vertex with a 2-simplex. In
the truncated ideal case, each edge was inflated to a bigon and not
any further, because all vertices were ideal.

We now describe the subdivision of each tile type. Ideal tiles are never
subdivided. Every polytope $P$ that intersects $S(i+1,0)$ in non-ideal
points comes from some stage $S(i,j)$. When gluing a polytope on to form
$S(i,j)$, we first end up with the complement in $P$ of the gluing
region, and then some boundary facets collapse and thus disappear, some get a concave ridge and are subsequently covered up,
and some get a flat ridge and are not covered up. The set of facets that collapse or get a concave
ridge is exactly the set of facets that intersects the underlying
$j$-tuple in $j$ ridges, as described in the proof of Lemma
\ref{StarLemma}. All collapsing facets disappear immediately and all
facets of $S(i,j)$ with concave ridges disappear in $S(i,j+1)$ as we glue
on polytopes to $j$-tuples. Thus, no matter what stage $j$ we glue $P$
onto, the set $P_i$ consists of the complement of a $j$-tuple minus all
facets that intersect the center cell of the $j$-tuple (so, for instance,
polytopes glued onto a 1-tuple delete all facets intersecting that one
facet, as well as the facet itself), followed by inflating every non-ideal
cell as described earlier. Thus, it is simply the inflated star of the cell.

\end{proof}

\begin{cor} There exist group-invariant finite subdivision rules for:
\begin{enumerate}
\item All non-split, prime alternating link
complements,
\item All fully augmented link complements, and
\item All closed 3-manifolds built from right-angled polyhedra.
\end{enumerate}

\end{cor} \begin{proof} All of these manifolds admit a decomposition
into right-angled polytopes. This is due to Menasco \cite{Menasco} for the alternating link complements and Agol and Thurston \cite{Lackenby} for augmented link complements. The third case is trivial. \end{proof}

This corollary improves on the results in \cite{PolySubs}, where we found subdivision rules for all closed 3-manifolds built from \emph{hyperbolic} right-angled polyhedra, which satisfy the additional condition that there are no prismatic 4-circuits.

\section{Subdivision rules for RAAGs}\label{RAAGSection}

Let $G$ be a RAAG with generators $e_1^{\pm 1},...,e_d^{\pm 1}$. In
\cite{CharneyDav}, the authors described a space called the \textbf{Salvetti complex} whose fundamental group is $G$. The Salvetti complex is obtained by taking a bouquet of circles (with one loop per generator), attaching one square for each commutator (by a map of the form $aba^{-1}b^{-1}$), and then attaching an $n$-cube for each $n$-clique in the defining graph of the right-angled Artin group.

This space is a subset of the $d$-dimensional torus
$\mathbb{T}^d$. We construct a space homotopy equivalent to this one
which is more suitable for our purposes; it will also be a subset of
$\mathbb{T}^d$; in fact, it will be a regular neighborhood of the Salvetti complex in $\mathbb{T}^d$.  

\begin{figure}\caption{On the left is the
Salvetti complex for $G=\mathbb{Z}*\mathbb{Z}^2$ (before taking the
quotient); on the right is our space $T_G$, which is homotopy
equivalent. The dotted lines represent deleted edges.}\label{Expansion}
\scalebox{.25}{\includegraphics{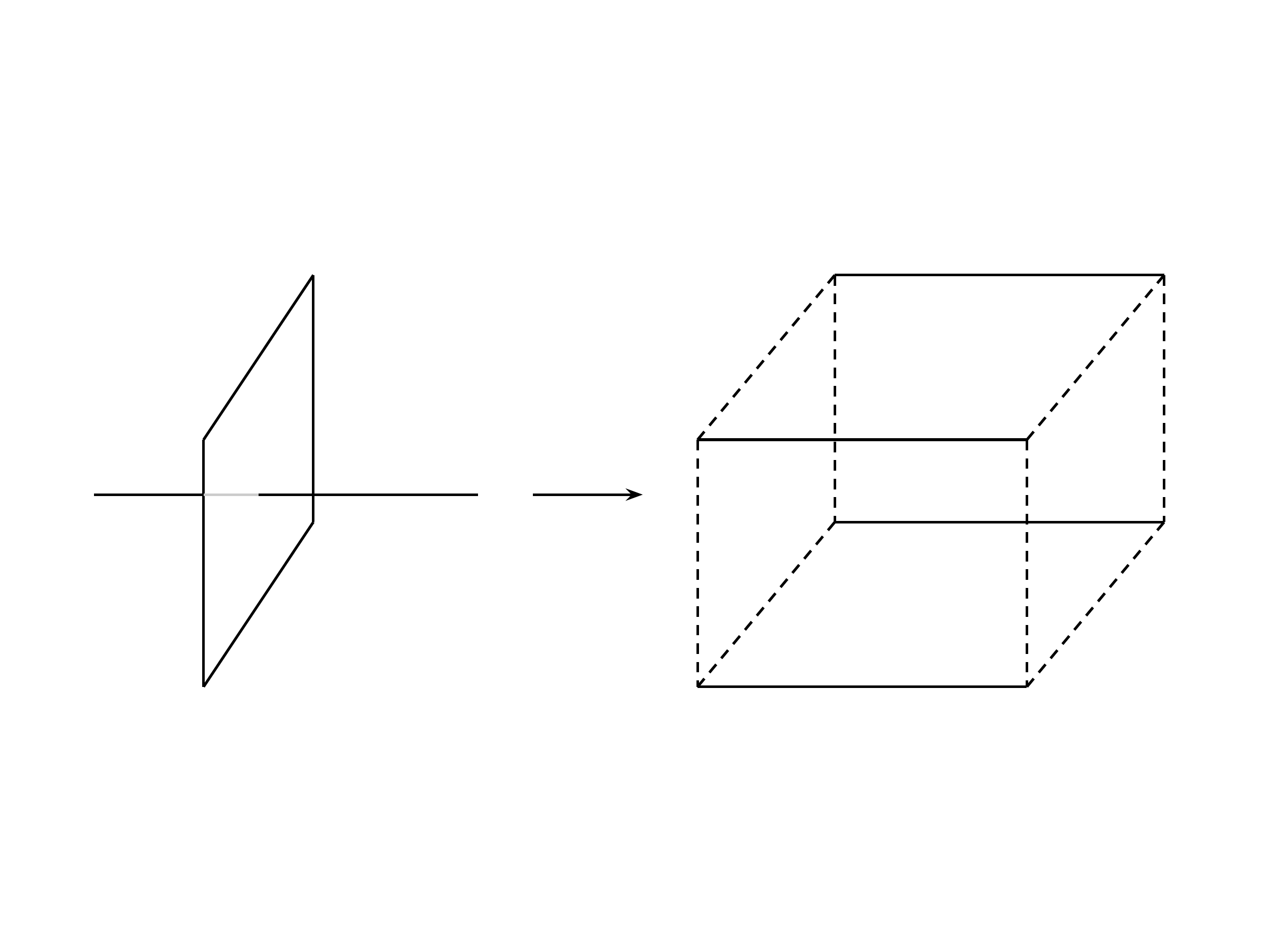}} \end{figure}

Let $I=[-1,1]$, and let $I^d$ be the cube of dimension $d$. Let $H_{ij}$
be the union of the four ridges of this cube defined by the equations
$x_i=\pm 1, x_j=\pm 1$. Finally, let $I'_G=I^d \setminus
\mathop{\bigcup}\limits_{[e_i,e_j]\neq 1} H_{ij}$. Thus, $I'_G$ is obtained
from $I^d$ by deleting all ridges corresponding to non-commuting pairs
of generators. We now let $T'_G=p(I'_G)\subseteq p(I^d)=\mathbb{T}^d$,
where $p$ is the quotient map from the cube to the torus. One can think
of $T'_G$ as the subspace of the torus obtained by snipping out the
portions of the manifold corresponding to some of the commutator
relations (those which are not found in our given RAAG).

\begin{thm}\label{DomainTheorem} The fundamental group of $T'_G$ is $G$.
\end{thm} \begin{proof} The 2-skeleton of $T_d$ can be taken to be
${d\choose 2}$ squares attached to $d$ edges (with a single vertex) to
form ${d\choose 2}$ tori, one for each pair of generators. We can
arrange this 2-skeleton so that the edges are dual to the images of
facets in our cube $I^d$ and so that the squares are dual to ridges. Then
the intersection of the 2-dimensional square and the codimension-2 ridge
is a single point. Thus, removing a ridge in $\mathbb{T}^d$ is
equivalent to puncturing one of the tori in the 2-skeleton, and thus it
corresponds to removing a commutator relation. Since the preimage of
such a ridge in $I^d$ consists of four separate ridges defined by the
equations above for $H_{ij}$, we are done. \end{proof}

Finally, we replace $I'_G$ and $T'_G$ with new cell complexes $I_G$ and
$T_G$ by taking a `truncated' version of both of them. In hyperbolic
geometry, truncating an ideal vertex of a right-angled polyhedron
replaces a deleted vertex of valence 4 with a square, where each corner
of the square intersects exactly one of the 4 edges that originally met
at the vertex. Similarly, we alter the fundamental domain $I'_G$ by
expanding each ideal cell into a facet (similar to the inflation process described in the previous section). The exact details are immaterial; all that we require is that the ideal set in the boundary has the same dimension as the boundary itself, with some cell structure.The manifold $T_G$ then is the quotient of $I_G$.

The exact details of this expansion or truncation are unimportant; it is
useful only because it implies that two non-ideal facets have non-trivial
intersection if and only if the corresponding generators are distinct
and commute. Almost any expansion would work as well. It does not change
the fundamental group.

\newtheorem*{BiggerTheorem}{Theorem \ref{BiggerTheorem}}
\begin{BiggerTheorem}
Every right-angled Artin group has a subdivision rule.
\end{BiggerTheorem}

\begin{proof} This follows immediately from
Theorem \ref{BigTheorem} and Lemma \ref{DomainTheorem}. \end{proof}

Recall from Theorem \ref{BigTheorem} that there is one tile types for each non-ideal \emph{cell} on the boundary of the right-angled fundamental domain, and a finite number of ideal tile types (which never subdivide). Up to symmetry, each non-ideal cell of $\partial I_G$ corresponds to a clique in the defining graph $\Gamma$. To see this, note that each cell of codimension $k$ in $\partial I'_G$ (the precursor to $\partial I_G$) is defined by $k$ equations $x_{i_1}=\pm 1, x_{i_2}=\pm 1,...,x_{i_k}=\pm 1$. Varying the signs of the equation gives other cells that are equivalent up to symmetry. The cell defined by these equations is non-ideal if and only if the generators $g_{i_1},g_{i_2},...,g_{i_k}$ commute. Since the non-ideal cells of $\partial I'_G$ are the same as the non-ideal cells of $\partial I_G$, we see that cliques in $\Sigma$ are in 1-1 correspondence with the symmetry classes of cells of $\partial I_G$, and thus are in 1-1 correspondence with the tile types of the subdivision rule $R$ for the group $G$.

We need a small fact to describe the subdivision. Let $K,K'$ be non-ideal cells of $\partial I_G$ with $K\subseteq K'$, and let $\Sigma_K, \Sigma_{K'}$ be the cliques they correspond to. Then $\Sigma_{K'}\subseteq \Sigma_K$, as $K$ is `cut out' by more equations than $K'$, and thus corresponds to more generators than $K'$.

This helps us describe the subdivision of each tile type; from Theorem \ref{BigTheorem}, we see that each tile type corresponding to a cell $K$ is subdivided into the complement of its inflated star. The facets in the inflated star of a cell $K$ in $\partial I_G$ corresponding to a clique $\Sigma_K$ are in 1-1 correspondence with commuting sets of generators that all commute with every element of $\Sigma_K$.

We now describe all examples for $d=3$. The 3-torus described earlier is
such an example, and the tile types for its subdivision rule are shown
in Figure \ref{TorusSubdivision}. The other possible groups are
$F_2\times \mathbb{Z}$ (shown in Figure \ref{ProductSubdivision});
$F_3$ (shown in Figure \ref{FreeSubdivision}), where $F_n$ is the free
group on $n$ generators; and the free product of $\mathbb{Z}$ with $\mathbb{Z}^2$, shown in Figures \ref{FreeProductX} and \ref{FreeProductSubs}.

 \begin{figure}
\scalebox{.5}{\includegraphics{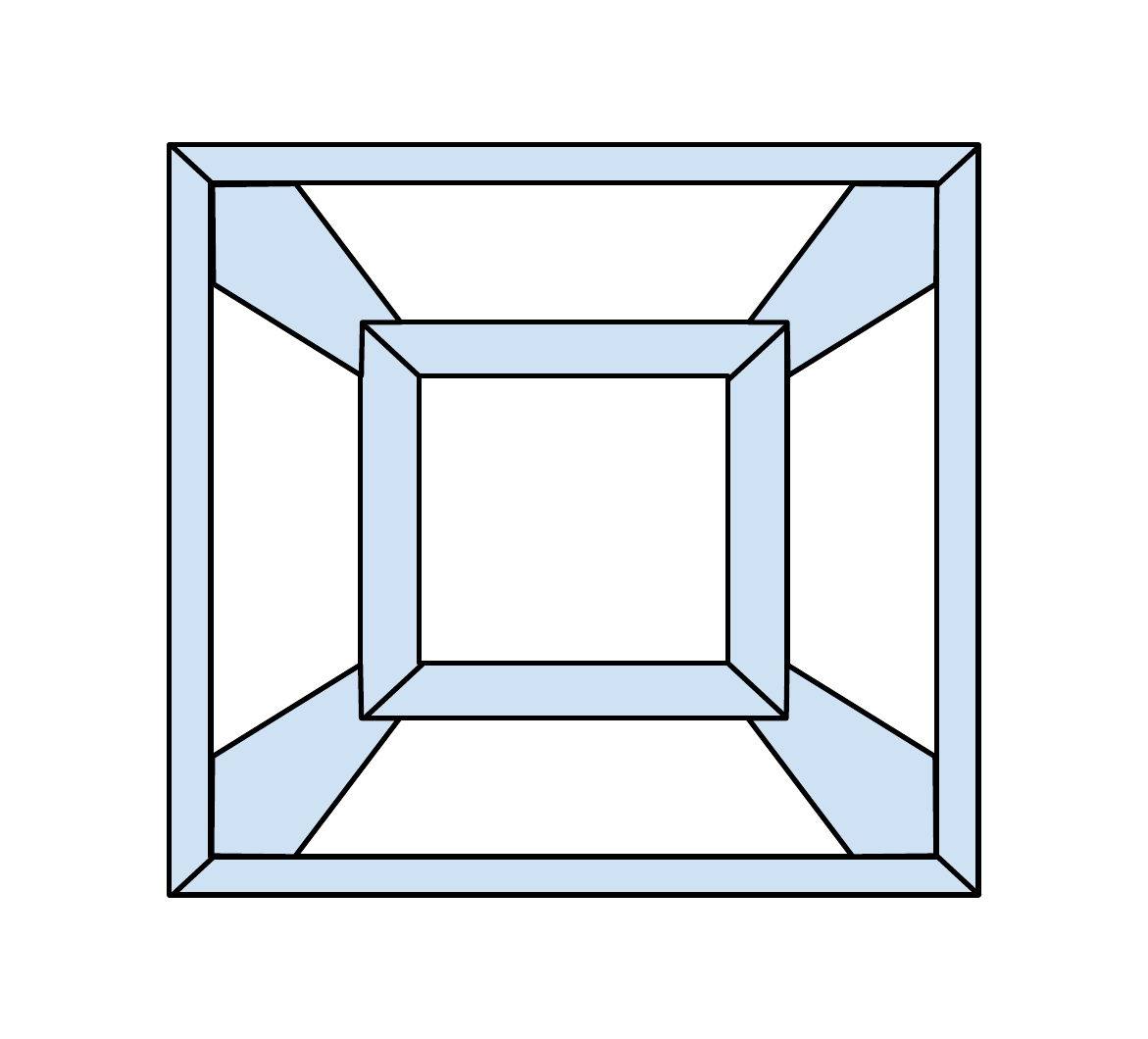}} \caption{The single tile type for the
subdivision rule of $F_3$. The complex $X$ that it acts on looks the
same as the tile type, except the outside is another
face.}\label{FreeSubdivision}\end{figure}

\begin{figure}
\scalebox{.4}{\includegraphics{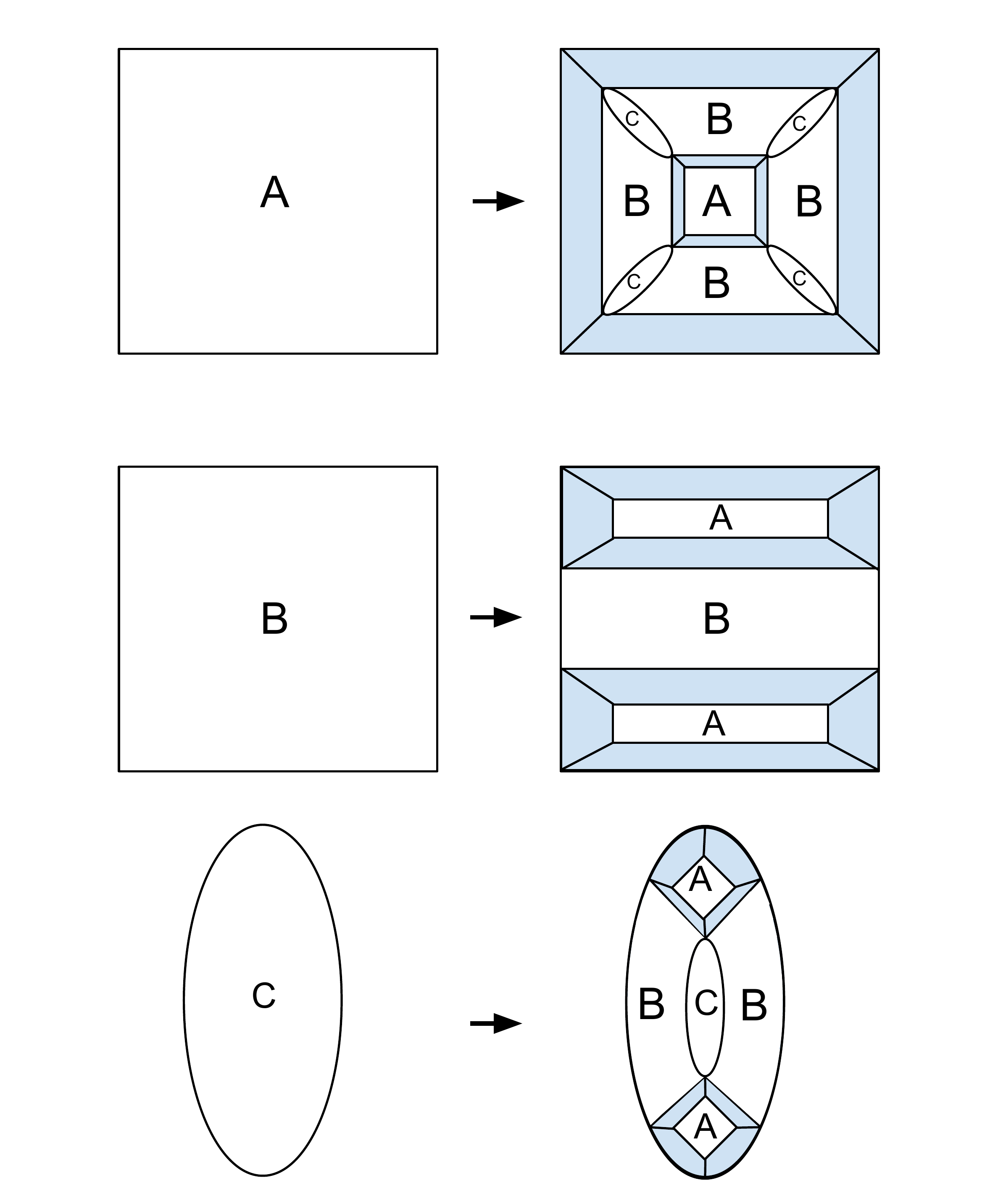}} \caption{The subdivision rule for the free product of $\mathbb{Z}$ with $\mathbb{Z}^2$.}\label{FreeProductSubs}\end{figure}

\begin{figure}
\scalebox{.5}{\includegraphics{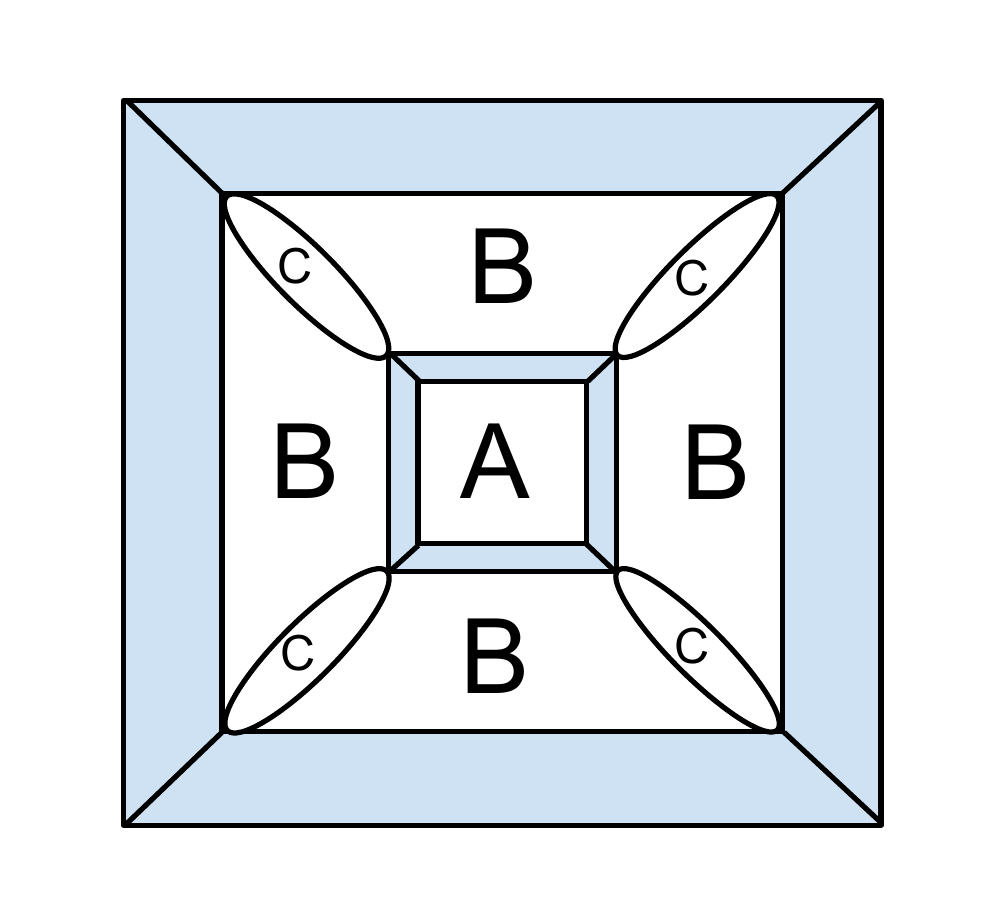}} \caption{The complex that the subdivision rule in Figure \ref{FreeProductSubs} acts on. The outside square is another $A$ tile.}\label{FreeProductX}\end{figure}

\section{Subdivision rules for special cubulated groups}\label{CubulatedSection}

To describe the subdivision rules for cubulated groups, we will need a few definitions.

Following Wise \cite{Wise}, we define a \textbf{cube complex} to be a cell complex obtained by gluing cubes together along facets by isometries. Since we are gluing by isometries, cube complexes have standard metrics. A \textbf{special cube complex} is a cube complex that is nonpositively curved in this metric, and that avoids certain pathologies in the gluing structure. For the purposes of this paper, it will be sufficient to use the following theorem due to Haglund and Wise \cite{haglund2008special}):

\begin{thm}
Let $X$ be a non positively curved cube complex. Then $X$ is special if and only if there is a local isometry of $X$ into the Salvetti complex for some right-angled Artin group $G$.
\end{thm}

In this section, we will also use a specific fact about the structure of right-angled Artin groups, as described by Charney \cite{Charney}. We use slightly different notation from \cite{Charney}. Given a right-angled Artin group $G$, a \textbf{spherical set} $\Sigma$ is a mutually commuting set of standard generators $h_1,...,h_k$  of $G$, where no generator appears with its inverse. A \textbf{spherical subgroup} $\langle \Sigma \rangle$ is the subgroup generated by the elements of $\Sigma$, together with their inverses.  Thus, two distinct spherical sets can generate the same spherical subgroup.

Each word $g=g_1g_2...g_n$ of generators of a $RAAG$ can be decomposed uniquely into sub words $w_i$ such that:

\begin{enumerate}
\item each word $w_i$ lies in a spherical subgroup $\langle \Sigma_i\rangle$, and
\item each word $w_i$ is \emph{maximal} with respect to property 1 (i.e. is the word of longest length satisfying property 1).
\end{enumerate}

Although this decomposition is unique,  each subword $w_i$ has multiple representatives. We can fix a unique representative of each subword by introducing diagonal generators:

\begin{defi}
Given a spherical set $\Sigma=\{g_1,...,g_n\}$, the \textbf{diagonal generator} $t_\Sigma$ is the product $g_1g_2\cdots g_n$. We say that a diagonal generator $t_\Sigma$ is \textbf{subordinate} to another generator $t_{\Sigma^\prime}$ if $\Sigma \subseteq \Sigma^\prime$.
\end{defi}

Using diagonal generators, we can represent each $w_i$ by a word $t_{\Sigma_k}\cdots t_{\Sigma_2}t_{\Sigma_1}$, where each $\Sigma_j$ is subordinate to $\Sigma_{j-1}$. This representation is now unique; for instance, in $\mathbb{Z}^3$, the word $a^5b^{-2}c^3$ decomposes as $t_a t_a t_a t_{ac}t_{ac}t_{ab^{-1}c}t_{ab^{-1}c}$. Thus, every word in a general RAAG can now be given a standard representative.

We will also need the definition of a cone type of a graph:

\begin{defi}
Let $\Gamma$ be a graph with a fixed origin $\mathcal{O}$. A \textbf{cone} of $\Gamma$ is the set $\Gamma_v$ of all points in $\Gamma$ that can be reached by a geodesic segment beginning at $\mathcal{O}$ and passing through $v$. We say that two vertices $v,w$ have the same \textbf{cone type} if the cones $\Gamma_v$ and $\Gamma_w$ are isometric by an isometry sending $v$ to $w$.
\end{defi}

In a finite subdivision rule, the tile type of a non-ideal tile is determined by the cone type of vertex in the history graph that it corresponds to.

\newtheorem*{BiggestTheorem}{Theorem \ref{BiggestTheorem}}
\begin{BiggestTheorem}
The fundamental group of a compact special cube complex $Y$
has a finite subdivision rule. If there is a local isometry of $Y$ into a Salvetti complex of a RAAG $A$, then the subdivision rule for $A$ contains a copy of the subdivision rule for $Y$.
\end{BiggestTheorem}
 \begin{proof} The
special cube complex $Y$ has a local isometry into the Salvetti complex
$X$ for a RAAG $\pi_1(X)$. This extends to a local isometry
$\widetilde Y \hookrightarrow \widetilde X$. The lifts in $\widetilde X$
of the basepoint $b_0$ of $X$ are in 1-to-1 correspondence with elements
of $\pi_1(X)$, and can be written as the set $\{g\tilde b_0 |g\in
\pi_1(X)\}$.

Now, we consider the lifts of $b_0$ that lie in $\widetilde
Y$. The set $\widetilde Y$ is convex (being a CAT(0) subcomplex). Recall that every element of a right-angled Artin group can be written uniquely as a product of maximal diagonal generators.
Given an element $h\in \pi_1(Y)\subseteq \pi_1(X)$, let $t_1t_2...t_n$ be its decomposition into diagonal generators, and let $\hat h$ be the
element of $\pi_1(X)$ given by $\hat h=t_2...t_n$. Then $\hat h\tilde b_0$ is a lift of $b_0$ in
$\widetilde X$, but we do not yet know if it lies in $\widetilde Y$. However, there is
clearly a geodesic in $\widetilde X$ (with its CAT(0) metric) from $\tilde b_0$ to $h \tilde b_0$
going through $\hat h \tilde b_0$. Because $\widetilde Y$ is convex, the
point $\hat h \tilde b_0$ must lie in $\widetilde Y$.

Now, consider the history graph $K$ of the subdivision rule $R$
corresponding to the RAAG $\pi_1(X)$. The vertices of $K$ are the
vertices of the Cayley graph of $\pi_1(X)$. Consider the induced
subgraph $H$ whose vertices are only those group elements $g$ of $\pi_1(X)$ where
the lifts $g\tilde b_0$ of the basepoint lie in the subcomplex
$\widetilde Y$. By the preceding  portion of the proof, if an element
$h$ of $\pi_1(X)$ is in $\pi_1(Y)$, then its predecessor in the history
graph $\hat h$ lies in $\widetilde Y$. Thus, the induced subgraph $H$ is
star convex at the origin and can be interpreted as the history graph of
a subdivision rule $S$ obtained by labeling as ideal all facets of
$R^n(X)$ except those that correspond to vertices of $H$ (Note: if it were not star convex, then some `ideal' facets would subdivide into non-ideal facets, which is not allowed).

Thus, we have a subdivision rule $S$, but we need to show that it is
finite. This is the same as showing that the history graph $H$ has
finitely many cone types. Now, the history graph $K$ of $R$ has
finitely many cone types. Let $h$ be a vertex of $H$. Then
$cone_H(h)=cone_K(h) \cap H$. This is true because $H$ is star convex
and an induced subgraph, so given any point $a$ of $cone_K(h) \cap H$,
there is a geodesic in $H$ from the origin to $a$ which is a geodesic in
$K$, meaning that $cone_K(h) \cap H$ lies in $cone_H(h)$. The reverse
inclusion is clear.

Finally, note that the action of $\pi_1(Y)$ on $\widetilde Y$ is
cocompact, implying that, under the action of $\pi_1(Y)$, there are only
finitely many equivalence classes of basepoint lifts $h \tilde
b_0$. On the other hand, since $S$ is a finite subdivision rule, there
are only finitely many cone types of vertices in $K$. So, given a vertex
$h$ in $H$ corresponding to a lifted basepoint $h\tilde b_0$, there are
only finitely many possibilities for its equivalence class under the
$\pi_1(Y)$ action and its cone type in $K$. Fix representatives $\tilde
b_1,...,\tilde b_n$ of each possible combination of cone type and equivalence class under the $\pi_1(Y)$ action. Then given any other
vertex $h'$ in $H$ corresponding to a lifted basepoint $h'
\tilde b_0$, there is some $\tilde b_i$ which has the same cone type in
$K$ and which is sent to $h'\tilde b_0$ by the action of some element
$y$ of $\pi_1(Y)$.  This implies that $cone_H(h')=cone_K(h') \cap
H=cone_K(yh) \cap H=cone_K(yh) \cap H\approx cone_K(h) \cap
y^{-1}H=cone_K(h) \cap H=cone_H(h)$. Thus, the two cone types in $H$ are
the same, and there are only finitely many cone types, which means that
$S$ is a finite subdivision rule.
\end{proof}

\section{Quasi-isometry properties of subdivision rules}\label{QuasiProofs}

In this section, we present the proofs of Theorems \ref{Quasi1Thm}-\ref{Quasi4Thm}.

Recall from the Introduction that the history graph of each subdivision rule is quasi-isometric to the group it is associated to. This means that subdivision rules can be used to study quasi-isometry
properties of these groups.

\begin{defi}
Let $X,Y$ be metric spaces with metrics $d_1,d_2$. Then a function $f:X\rightarrow Y$ is a \textbf{quasi-isometric embedding} if there is a constant $C$ such that for all $x_1,x_2$ in $X$, $$\frac{1}{C}d_2(f(x_1),f(x_2))-C\leq d_1(x_1,x_2) \leq C d_2(f(x_1),f(x_2)+C.$$ A quasi-isometric embedding is a \textbf{quasi-isometry} if every point of $Y$ is within some fixed distance $D$ of the image of $X$ under $f$.
\end{defi}

Quasi-isometric spaces have the same `large scale' structure. The following theorems all deal with properties that are invariant under quasi-isometry.

\begin{defi}
A \textbf{growth function} for a group $G$ is a function $g:\mathbb{N}\rightarrow\mathbb{N}$ such that $g(n)$ is the number of elements of $G$ of distance $n$ from the origin in the word metric with some finite generating set.
\end{defi}

Since growth functions depend on the generating set, they are not unique. However, the degree of growth (polynomial of degree $d$, exponential, etc.) is a quasi-isometry invariant. 

\newtheorem*{Quasi1Thm}{Theorem \ref{Quasi1Thm}}
\begin{Quasi1Thm}
The growth function of $G$ is the number of non-ideal cells in $R^n(X)$.
\end{Quasi1Thm}
\begin{proof}
By construction, the non-ideal cells of $R^n(X)$ are in 1-1 correspondence with the vertices in the sphere of radius $n$ in the history graph $\Gamma$, which is quasi-isometric to $G$.
\end{proof}

\newtheorem*{Quasi2Thm}{Theorem \ref{Quasi2Thm}}
\begin{Quasi2Thm}
If the subdivision rule $R$ has mesh approaching 0 (meaning that each path crossing non-ideal tiles eventually gets subdivided), then the group $G$ is $\delta$-hyperbolic.
\end{Quasi2Thm}

\begin{proof}
This is Theorem 6 of \cite{myself2}.
\end{proof}

\begin{defi} Let $X$ be a metric space. Then an \textbf{end} of $X$ is a sequence  $E_1\subseteq E_2 \subseteq E_3\subseteq...$ such that each $E_i$ is a component of $X\setminus B(0,n)$, the complement in $X$ of a ball of radius $n$ about the origin.

If $G$ is a group, then an \textbf{end} of $G$ is an end of a Cayley graph for $G$ with the word metric.
\end{defi}

\newtheorem*{Quasi3Thm}{Theorem \ref{Quasi3Thm}}
\begin{Quasi3Thm}
The number of ends of $G$ is the same as the number of components of $X\setminus \bigcup R_I^n(X)$, where $R^n_I$ is the union of the ideal tiles of $R^n(X)$.
\end{Quasi3Thm}

\begin{proof}
Let $A_n=X\setminus R_I^n(X)$, and let $A=\bigcap A_n=X\setminus \bigcup R_I^n(X)$. Let $B$ be a component of $A$. Then $B$ is contained in a component $B_n$ of $A_n$. The set $B_n$ is connected and locally path connected, so it is path connected. Now, given $x$ and $y$ in $A$, we can find a path $\alpha_n$ in $B_n$ between $x$ and $y$ for each $n$. In the history graph $\Gamma$, the points $x$ and $y$ correspond to geodesic rays from the origin, and the paths $\alpha_n$ connecting them correspond to paths in the $n$-sphere in $\Gamma$, which lies outside the ball $B(0,n)$ of radius $n$ about the origin. Thus, the rays corresponding to $x$ and $y$ are connected by paths lying outside of $B(0,m)$ for any $m$, so are in the same end of $\Gamma$.

Thus, all rays corresponding to points of $B$ are in the same end of $\Gamma$.

On the other hand, let $x$ be in one component $B$ of $A$ and $y$ in another component $C$. Then there is some $n$ such that $B_n\cap C_n=\emptyset$. But this means that there is no path in $\Gamma\setminus B(0,n)$ from the ray representing $x$ to the ray representing $y$, since any such path would imply the existence of a similar path in the sphere of radius $n$ in $\Gamma$, which is dual to $A_n$. (This path can be obtained explicitly by pushing every vertex in $\alpha$ to its predecessor in the sphere of radius $n$, and sending any edge in the path to the corresponding edge between the projection of its two vertices). Thus, the two geodesics lie in different ends of $\Gamma$.
\end{proof}

Finally, subdivision rules can be used to study divergence. There are several definitions; we use the definition given in \cite{behrstock2012divergence}:

\begin{defi}
 The \textbf{divergence of a geodesic ray} $\gamma$ in a metric space $X$ is the function of $n$ given by the length of the shortest path in $X\setminus B(\gamma(0),n)$ between $\gamma(-n)$ and $\gamma(n)$. The \textbf{divergence} of a group is the supremum of the divergence of all geodesics in the Cayley graph of $G$. \end{defi}

In many spaces, the distance between the endpoints is given by a path lying on
the surface of the sphere. This is true for history graphs, as we will show in Theorem \ref{Quasi4Thm}. So, the divergence is essentially the diameter
of the sphere of radius $n$ as a metric space with the induced metric.
The subdivision rules described in this paper give an exact description
of the cell structure of the sphere of radius $n$, and the actual metric
on the sphere of radius $n$ is quasi-isometric to the `chunky metric'
given by counting the minimum number of cells that a path between two
points must intersect, and the quasi-isometry constants are independent
of $n$. Thus, we can just find the diameter of the sphere in this
`chunky metric'. For instance, in the subdivision rules in Figure
\ref{TorusSubdivision} for the 3-torus, we can take two points $a,b$ in
the subdivision complex $X$ that are in the intersection of infinitely
many $C$-tiles. Then the length of any path between $a$ and $b$ in
$R^n(X)$ grows linearly with $n$, which is what is expected for this
manifold \cite{behrstock2012divergence}. The same holds true for the
subdivision in Figure \ref{ProductSubdivision}. In the subdivision rule
shown in Figure \ref{FreeSubdivision}, there is no path between distinct
components, which is expected, as this is a free group. The same holds for the subdivision depicted in Figures \ref{FreeProductSubs} and \ref{FreeProductX}.

In general, divergence is not a quasi-isometry invariant, but it becomes so under a new equivalence relation. If $f(n),g(n)$ are two divergence functions, we write $f(n)\preceq g(n)$ if $f(n)\leq g(an+b)+an+b$ for some constants $a,b$. If $f(n)\preceq g(n)$ and $g(n) \preceq f(n)$, we say that $f$ and $g$ are \textbf{equivalent}.

\newtheorem*{Quasi4Thm}{Theorem \ref{Quasi4Thm}}
\begin{Quasi4Thm}
If $R$ is a subdivision rule acting on a space $X$ associated to a group $G$, then the diameter $diam_X(n)$ of $R^n(X)$ is an upper bound to the divergence of $G$, i.e. $div_G(n)\preceq diam_X(n)$. Conversely, if is a geodesics $\gamma$ in the history graph of $R$ realizing $diam_X(n)$ (i.e. with $d_{R^n(X)}(\gamma(n),\gamma(-n))=diam_X(n)$)) then $diam_X(n)\preceq div_G(n)$.
\end{Quasi4Thm}

\begin{proof}
The history graph $\Gamma$ of $\{R^n(X)\}$ is quasi-isometric to the Cayley graph $C=C(G)$; in fact, there is a fixed $K>0$ such that for all $y$ in $C$, the history graph $\Gamma$ is $K$-quasi-isometric to $G$ by a quasi-isometry taking the origin in $\Gamma$ to $y$ in $C$. This implies that the divergence of $C$ is equivalent to the divergence of $\Gamma$ measured from the origin.

Thus, we only need to bound the divergence of $\Gamma$ measured from the origin. Look at $\Gamma\setminus B(0,n)$, and choose two points $x,y$ at distance $n$ from the origin. All such points lie in the dual graph of $R^n(X)$, which we call $\Gamma_n$. We need to measure the minimum distance of a path between $x$ and $y$ that avoids the ball of distance $n$; such a path is called an \textbf{avoidant path}. We claim that any minimal avoidant path $\alpha$ from $x$ to $y$ lies completely in $\Gamma_n$. This is because we can canonically `project' $\alpha$ avoiding $B(0,n)$ onto $\Gamma_n$ by replacing each vertex in the path with its unique ancestor in $\Gamma_n$, and replacing each edge between two vertices with the edge between the projections of its endpoints. This projection collapses all vertical edges in $\alpha$ and collapses all horizontal edges whose endpoints have a common projection. Since every path that is not yet in $\Gamma_n$ must have vertical edges, projection will shorten all paths except those that lie in $\Gamma_n$. Thus, the minimal avoidant path $\alpha$ lies in $\Gamma_n$.

Thus, the length of any avoidant path between two geodesics rays is bounded above by the diameter of $\Gamma_n$, which is $diam_X(n)$, so the length of an avoidant path connecting the two halves of a bi-infinite geodesic is also bounded above. This proves the first statement of the theorem.

On the other hand, if $\gamma$  is a geodesic realizing the diameter, its divergence is given by the diameter of $\Gamma_n$. Because $div_{\Gamma}(n)$ is a supremum, it is bounded below by the divergence of $\gamma$, and thus by $diam_X(n)$. This proves the second statement of the theorem.
\end{proof}

\bibliographystyle{plain} \bibliography{RAAGbib}

\begin{thebibliography}{10}

\bibitem{Agol}
I.~Agol, D.~Groves, and J.~Manning.
\newblock The virtual haken conjecture.
\newblock 2012.
\newblock preprint.

\bibitem{behrstock2012divergence}
Jason Behrstock and Ruth Charney.
\newblock Divergence and quasimorphisms of right-angled artin groups.
\newblock {\em Mathematische Annalen}, 352(2):339--356, 2012.

\bibitem{behrstock2012cubulated}
Jason Behrstock and Mark~F Hagen.
\newblock Cubulated groups: thickness, relative hyperbolicity, and simplicial
  boundaries.
\newblock {\em arXiv preprint arXiv:1212.0182}, 2012.

\bibitem{behrMann}
Jason Behrstock and Walter Neumann.
\newblock Quasi-isometric classification of non-geometric 3-manifold groups.
\newblock {\em Journal f\"ur die reine und angewandte Mathematik},
  2012:101--120, 2012.

\bibitem{behrMannJan}
Jason Behrstock, Walter Neumann, and Tadeusz Januszkiewicz.
\newblock Quasi-isometric classification of some high dimensional right-angled
  artin groups.
\newblock {\em Groups, Geometry and Dynamics}, 4:681--692, 2010.

\bibitem{AsympGeom}
Mladen Bestvina, Bruce Kleiner, and Michah Sageev.
\newblock The asymptotic geometry of right-angled artin groups.
\newblock {\em Geometry and Topology}, 12:1653--1699, 2008.

\bibitem{conformal}
J.~W. Cannon.
\newblock The combinatorial {R}iemann mapping theorem.
\newblock {\em Acta Mathematica}, 173(2):155--234, 1994.

\bibitem{Rich}
J.~W. Cannon, W.~J. Floyd, and W.~R. Parry.
\newblock Sufficiently rich families of planar rings.
\newblock {\em Annales Academi{\ae} Scientiarum Fennic{\ae} Mathematica},
  24:265--304, 1999.

\bibitem{subdivision}
J.~W. Cannon, W.~J. Floyd, and W.~R. Parry.
\newblock Finite subdivision rules.
\newblock {\em Conformal Geometry and Dynamics}, 5:153--196, 2001.

\bibitem{hyperbolic}
J.~W. Cannon and E.~L. Swenson.
\newblock Recognizing constant curvature discrete groups in dimension 3.
\newblock {\em Transactions of the American Mathematical Society},
  350(2):809--849, 1998.

\bibitem{Combinatorial}
J.W. Cannon.
\newblock The combinatorial structure of cocompact discrete hyperbolic groups.
\newblock {\em Geom. Dedicata}, 16:123--148, 1984.

\bibitem{Charney}
R.~Charney.
\newblock An introdution to right-angled {A}rtin groups.
\newblock {\em Geom. Dedicata}, 125:141--158, 2007.

\bibitem{CharneyDav}
R.~Charney and M.~Davis.
\newblock Finite {K}($\pi$,1)s for {A}rtin groups.
\newblock In Alan Mycroft, editor, {\em Prospects in topology (Princeton, NJ,
  1994)}, volume 138 of {\em Ann. of Math. Stud.}, pages 110--124. Princeton
  University Press, 1995.

\bibitem{Gersten}
S.M. Gersten.
\newblock Divergence in 3-manifold groups.
\newblock {\em Geometric and Functional Analysis GAFA}, 4(6):633--647, 1994.

\bibitem{Gromov1981groups}
Mikhael Gromov.
\newblock Groups of polynomial growth and expanding maps.
\newblock {\em Publications Math{\'e}matiques de l'IH{\'E}S}, 53(1):53--78,
  1981.

\bibitem{haglund2008special}
Fr{\'e}d{\'e}ric Haglund and Daniel~T Wise.
\newblock Special cube complexes.
\newblock {\em Geometric and Functional Analysis}, 17(5):1551--1620, 2008.

\bibitem{hsu1999linear}
Tim Hsu and Daniel~T Wise.
\newblock On linear and residual properties of graph products.
\newblock {\em Michigan Math. J}, 46(2):251--259, 1999.

\bibitem{Kosinksi}
A.~Kosinksi.
\newblock {\em Differential Manifolds}, volume 138 of {\em Pure and Applied
  Mathematics}.
\newblock Academic Press, 1992.

\bibitem{Lackenby}
M.~Lackenby.
\newblock The volume of hyperbolic alternating link complements.
\newblock {\em Proc. London Math. Soc.}, 88:204--224, 2004.
\newblock With an appendix by Ian Agol and Dylan Thurston.

\bibitem{Menasco}
W.~Menasco.
\newblock Closed incompressible surfaces in alternating knot and link
  complements.
\newblock {\em Topology}, 23:37--44, 1984.

\bibitem{myself}
B.~Rushton.
\newblock Alternating links and subdivision rules.
\newblock Master's thesis, Brigham Young University, 2009.

\bibitem{PolySubs}
B.~Rushton.
\newblock Constructing subdivision rules from polyhedra with identifications.
\newblock {\em Alg. and Geom. Top.}, 12:1961--1992, 2012.

\bibitem{CubeSubs}
B.~Rushton.
\newblock A finite subdivision rule for the n-dimensional torus.
\newblock {\em Geometriae Dedicata}, pages 1--12, 2012.

\bibitem{myself2}
B.~Rushton.
\newblock {\em Subdivision Rules, 3-manifolds and Circle Packings}.
\newblock PhD thesis, Brigham Young University, 2012.

\bibitem{stallings1971group}
John~R Stallings.
\newblock Group theory and 3-manifolds.
\newblock In {\em Actes du Congres International des Math{\'e}maticiens (Nice,
  1970), Tome}, volume~2, pages 165--167, 1971.

\bibitem{tits1972free}
Jacques Tits.
\newblock Free subgroups in linear groups.
\newblock {\em Journal of Algebra}, 20(2):250--270, 1972.

\bibitem{Wise}
D.~Wise.
\newblock {\em From Riches to Raags: 3-Manifolds, Right-Angled Artin Groups,
  and Cubical Geometry}.
\newblock {CBMS} Regional Conference Series in Mathematics. American
  Mathematical Society, 2012.

\end{thebibliography}

\end{document}